\documentclass[a4paper,10pt,reqno]{amsart}
\usepackage{dsfont,amsmath,amsfonts,amscd,amssymb,graphicx,mathrsfs,eufrak,upgreek,pdflscape,colonequals,euscript,amsmath,dsfont}
\usepackage[usenames]{color}
\usepackage{setspace}
\usepackage{array,multirow,booktabs,longtable}
\usepackage{mathtools}
\usepackage{comment}

\usepackage{enumitem}
\setlist[enumerate]{format=\normalfont}

\renewcommand*\partial{\textnormal{\reflectbox{6}}}

\newcolumntype{C}[1]{>{\centering\let\newline\\\arraybackslash\hspace{0pt}}m{#1}}

\usepackage{array,multirow,booktabs,longtable}
\newcommand{\marginparstretch}{0.6}
\let\oldmarginpar\marginpar
\renewcommand\marginpar[1]{\-\oldmarginpar[\framebox{\setstretch{\marginparstretch}\begin{minipage}{\marginparwidth}{\raggedleft\tiny #1}\end{minipage}}]{\framebox{\setstretch{\marginparstretch}\begin{minipage}{\marginparwidth}{\raggedright\tiny #1}\end{minipage}}}}

\newcommand{\xdownarrow}[1]{%
  {\left\downarrow\vbox to #1{}\right.\kern-\nulldelimiterspace}
}

\usepackage{amsmath,amssymb,amscd,bm,graphicx,tikz-cd,upgreek,dsfont,amsfonts,tikz}
\usetikzlibrary{cd,calc}
\usepackage[pagebackref]{hyperref}

\hypersetup{
  colorlinks   = true,          
  urlcolor     = blue,          
  linkcolor    = purple,          
  citecolor   = blue             
}

\definecolor{Pink}{RGB}{230 56 243}
\definecolor{Blue}{RGB}{0 19 147}
\definecolor{Green}{RGB}{66 147 41}
\definecolor{Grey}{RGB}{102 102 102}
\definecolor{Orange}{RGB}{237 107 45}
\definecolor{Red}{RGB}{234 53 159}

\usepackage{tikz,mathrsfs}
\usetikzlibrary{arrows,decorations.pathmorphing,decorations.pathreplacing,positioning,shapes.geometric,shapes.misc,decorations.markings,decorations.fractals,calc,patterns,plotmarks}

\tikzset{
        cvertex/.style={circle,draw=black,inner sep=1pt,outer sep=3pt},
        vertex/.style={circle,fill=black,inner sep=1pt,outer sep=3pt},
        DB/.style={circle,draw=black,circle,fill=black,inner sep=0pt, minimum size=4pt},
        DBlue/.style={circle,draw=black,circle,fill=white,inner sep=0pt, minimum size=4pt},
        DW/.style={circle,draw=black,inner sep=0pt, minimum size=4pt},
        star/.style={circle,fill=yellow,inner sep=0.75pt,outer sep=0.75pt},
        tvertex/.style={inner sep=1pt,font=\scriptsize},
        gap/.style={inner sep=0.5pt,fill=white}}

\newtheorem{thm}{Theorem}[section]
\newtheorem{prop}[thm]{Proposition}
\newtheorem{lemma}[thm]{Lemma}
\newtheorem{cor}[thm]{Corollary}

\newtheorem{conj}[thm]{Conjecture}

\theoremstyle{definition} 
\newtheorem{defin}[thm]{Definition}

\newtheorem{setup}[thm]{Setup}

\newtheorem{remark}[thm]{Remark}

\newtheorem{convention}[thm]{Convention}

\newtheorem{notation}[thm]{Notation}

\numberwithin{equation}{section}

\def\coh{\mathop{\rm coh}\nolimits}

\def\Hom{\mathop{\rm Hom}\nolimits}

\def\RHom{\mathop{\rm {\bf R}Hom}\nolimits}

\def\End{\mathop{\rm End}\nolimits}
\def\Ext{\mathop{\rm Ext}\nolimits}

\def\charac{\mathop{\rm char}\nolimits}

\def\Db{\mathop{\rm{D}^b}\nolimits}

\newcommand{\K}{\mathop{{}_{}\mathbb{K}}\nolimits}

\def\ab{\mathop{\rm ab}\nolimits}

\def\RHom{{\rm{\bf R}Hom}}

\newcommand\Curve{\mathrm{C}}

\newcommand{\N}{\mathbb N}
\newcommand{\Z}{\mathbb Z}

\newcommand{\C}{\mathbb C}

\newcommand{\scrC}{\EuScript{C}}
\newcommand{\scrD}{\EuScript{D}}
\newcommand{\scrE}{\EuScript{E}}
\newcommand{\scrF}{\EuScript{F}}

\newcommand{\scrI}{\EuScript{I}}
\newcommand{\scrJ}{\EuScript{J}}

\newcommand{\scrN}{\EuScript{N}}
\newcommand{\scrO}{\EuScript{O}}

\newcommand{\scrU}{\EuScript{U}}

\newcommand{\scrX}{\EuScript{X}}

\newcommand{\Ainf}{\mathrm{A}_\infty}

\newcommand{\dcyc}{\updelta}

\newcommand{\llangle}{\langle\kern -2.5pt\langle}
\newcommand{\rrangle}{\rangle\kern -2.5pt\rangle}
\newcommand{\llcurve}{\{\kern -3pt\{}
\newcommand{\rrcurve}{\}\kern -3pt\}}
\newcommand{\llsq}{[\![}
\newcommand{\rrsq}{]\!]}
\newcommand{\lcl}{(\kern -2.5pt(}
\newcommand{\rcl}{)\kern -2.5pt)}
\newcommand{\Jac}{\scrJ\mathrm{ac}}

\renewcommand{\phi}{\upvarphi}

\newcommand{\ring}{\C\llangle x,y\rrangle}
\newcommand{\freering}{\C\langle x,y\rangle}

\newcommand{\ringtwo}{\C\llangle x,y\rrangle}
\newcommand{\commringtwo}{\C\llsq x,y\rrsq}

\newcommand{\kx}{\mathsf{r}}
\newcommand{\ky}{\mathsf{s}}
\newcommand{\ordP}{\mathsf{t}}

\newcommand{\vC}{\check{\mathrm{C}}}
\newcommand{\asf}{\mathsf{a}}
\newcommand{\bsf}{\mathsf{b}}
\newcommand{\csf}{\mathsf{c}}
\newcommand{\dsf}{\mathsf{d}}
\newcommand{\esf}{\mathsf{e}}
\newcommand{\fsf}{\mathsf{f}}
\newcommand{\gsf}{\mathsf{g}}
\newcommand{\hsf}{\mathsf{h}}

\newcommand{\ksf}{\mathsf{k}}

\newcommand{\msf}{\mathsf{m}}

\newcommand{\xsf}{\mathsf{x}}
\newcommand{\ysf}{\mathsf{y}}
\newcommand{\zsf}{\mathsf{z}}
\newcommand{\Bsf}{\mathsf{B}}

\newcommand{\Xsf}{\mathsf{X}}
\newcommand{\Ysf}{\mathsf{Y}}
\newcommand{\Zsf}{\mathsf{Z}}
\newcommand{\Gsf}{\mathsf{G}}
\newcommand{\Hsf}{\mathsf{H}}
\newcommand{\Ksf}{\mathsf{K}}
\newcommand{\Ssf}{\mathsf{S}}
\newcommand{\Tsf}{\mathsf{T}}
\newcommand{\Usf}{\mathsf{U}}

\renewcommand{\upxi}{\mathfrak{s}}	
\newcommand{\dfrak}{\mathfrak{d}}
\newcommand{\dd}{\mathds{D}}
\DeclareMathOperator{\diag}{diag}
\DeclareMathOperator{\Mat}{Mat}

\DeclareMathOperator{\scrEnd}{\EuScript{E}\rm nd}
\DeclareMathOperator{\scrHom}{\EuScript{H}\rm om}
\DeclareMathOperator{\Neck}{\EuScript{N}\rm eck}
\DeclareMathOperator{\Mono}{\EuScript{M}\rm ono}
\DeclareMathOperator{\NeckPoly}{\EuScript{N}}
\DeclareMathOperator{\Orb}{\EuScript{O}\mathrm{rb}}

\definecolor{Pink}{RGB}{230 56 243}
\definecolor{Blue}{RGB}{0 19 147}
\definecolor{Green}{RGB}{66 147 41}
\definecolor{Grey}{RGB}{102 102 102}
\definecolor{Orange}{RGB}{237 107 45}
\definecolor{Red}{RGB}{234 53 159}

\makeatletter
\newcommand{\Scalecenter}[2][1]{\mathpalette\Scalecenter@{{#1}{#2}}}
\newcommand{\Scalecenter@}[2]{\Scalecenter@@#1#2}
\newcommand{\Scalecenter@@}[3]{%
  \vcenter{\hbox{\scalebox{#2}{$\m@th#1#3$}}}%
}
\makeatother

\DeclareMathOperator{\redbullet}{{\color{Orange}\bullet}}
\DeclareMathOperator{\bluebullet}{{\color{Blue}\Scalecenter[0.7]{\blacktriangle}}}
\DeclareMathOperator{\greenbullet}{{\color{Green}\Scalecenter[0.61]{\blacksquare}}}

\tikzset{
W/.style={circle,draw=black,circle,fill=white,inner sep=0pt, minimum size=4pt},
B/.style={circle,draw=black!80!white,circle,fill=Blue,inner sep=0pt, minimum size=4pt},
G/.style={circle,draw=black!80!white,circle,fill=Green,inner sep=0pt, minimum size=4pt},
Or/.style={circle,draw=black!80!white,circle,fill=Orange,inner sep=0pt, minimum size=4pt},
P/.style={circle,draw=black!80!white,circle,fill=Pink,inner sep=0pt, minimum size=4pt},
R/.style={circle,draw=Orange,line width=0.75pt,circle,fill=white,inner sep=0pt, minimum size=4pt},  
}

\def\centerarc[#1](#2)(#3:#4:#5)
    { \draw[#1] ($(#2)+({#5*cos(#3)},{#5*sin(#3)})$) arc (#3:#4:#5); }

\begin{document}
\title{Derived deformation theory of crepant curves}
\author{Gavin Brown}
\address{Gavin Brown, Mathematics Institute, Zeeman Building, University of Warwick, Coventry, CV4 7AL, UK.}
\email{G.Brown@warwick.ac.uk}
\author{Michael Wemyss}
\address{Michael Wemyss, School of Mathematics and Statistics,
University of Glasgow, University Place, Glasgow, G12 8QQ, UK.}
\email{michael.wemyss@glasgow.ac.uk}
\begin{abstract}
This paper determines the full derived deformation theory of certain smooth rational curves $\Curve$ in Calabi-Yau 3-folds, by determining all higher $\Ainf$-products in its controlling DG-algebra.   This geometric setup includes very general cases where $\Curve$ does not contract, cases where the curve neighbourhood is not rational, all known simple smooth $3$-fold flops, and all known divisorial contractions to curves.  As a corollary, it is shown that the noncommutative deformation theory of $\Curve$ is described via a superpotential algebra derived from what we call free necklace polynomials, which are elements in the free algebra obtained via a closed formula from combinatorial gluing data.  The description of these polynomials, together with the above results, establishes a suitably interpreted string theory prediction due to Ferrari \cite{Ferrari}, Aspinwall--Katz \cite{AspinwallKatz} and Curto--Morrison \cite{CM}.  Perhaps most significantly, the main results give both the language and evidence to finally formulate new contractibility conjectures for rational curves in CY 3-folds, which lift Artin's celebrated results \cite{Artin} from surfaces. 
\end{abstract}
\thanks{The authors were supported by EPSRC grant EP/R034826/1 and by the ERC Consolidator Grant 101001227 (MMiMMa).}
\maketitle

\setcounter{tocdepth}{1}
\tableofcontents

\section{Introduction}
Rational curves in Calabi--Yau 3-folds $\Curve\subset\scrX$ are fundamental building blocks of geometry, and yet some of even their most basic properties remain wide open.  The main reason, in contrast to when $\scrX$ is a surface, is that various key properties of  $\Curve$ in $\scrX$ including its formal neighbourhood, its contractibility and its deformation theory, are structural and not combinatorial.  
The controlling structure of all three properties now lies in homological and noncommutative algebra.

\medskip
The purpose of this paper is to describe the full, derived, deformation theory of $\Curve$ in $\scrX$, for the general curves described in \S\ref{subsec:geosetup} below.  This information is then used to (1) give the first general closed-formula description of the associated noncommutative deformation algebra of $\Curve$, (2) build, in the context of mirror symmetry, a B-side geometric model from the purely algebraic data of a given superpotential, and (3) use the evidence from these results, and others, to finally be in a position to formulate and conjecture contractibility criteria for formal neighbourhoods of curves in 3-folds.  Noncommutative deformations are the language and framework needed in order to lift Artin's 60-year old work \cite{Artin} from surfaces to 3-folds.

\subsection{Geometric Setup}\label{subsec:geosetup}
For many reasons, some explained in \S\ref{subsec:conjArt}, the crucial open and key benchmarking case for all the problems listed above is when $\Curve\cong\mathbb{P}^1$ is a single curve, with normal bundle $\scrO(-3)\oplus\scrO(1)$.  This is a vast set of curves, an uncharted zoo, known to contain an array of different and surprising geometric behaviour.  It is thus important to understand such curves, and to establish their deformation properties.

Deformation theory is local, and so this problem at once reduces to a local model, patching together two copies of affine space (after completion, when necessary).  In full generality, there is no known description of this model.  However inside this general situation sits a very large subset, which is expected to be representative of general behaviour. This is provided by the string theory literature \cite{Ferrari, Katz, CM}, where a smooth rational curve $\Curve$ inside a 3-fold $\scrX$ is built as follows.  

Set $\uplambda_{00}=\uplambda_{10}=\uplambda_{01}=0$, and choose finitely many scalars $\uplambda_{jk}$ where $j,k\geq 0$ and $j+k\geq 2$.  Given this choice of data, consider the scheme $\scrX$ defined by the two open  patches $\scrU_1=\mathbb{A}^3_{a,v_2,v_1}$ and $\scrU_2=\mathbb{A}^3_{b,w_2,w_1}$ glued along
\begin{equation}\label{eq!glue!intro}
(a,v_2,v_1)\stackrel{a\neq 0}{\longleftrightarrow}(a^{-1}, \,a^{-1}v_2, \,\,a^{3}v_1+\sum_{j,k\geq 0} \uplambda_{jk}a^{2-k}v_2^{j+k-1}).
\end{equation}
The locus $v_2=v_1=0$ glues with the locus $w_2=w_1=0$ to create a curve $\Curve\cong\mathbb{P}^1\subset\scrX$.  

Both the scheme $\scrX$, and the properties of the curve $\Curve\subset\scrX$, depend heavily on the choice of $\uplambda_{jk}$, but this is suppressed from  the notation.  We show in \ref{lem:normalbundle} that $\Curve\subset\scrX$ has normal bundle $\scrO(-3)\oplus\scrO(1)$ if and only if there are no small degree terms in \eqref{eq!glue!intro}.   Given our motivation, for the remainder of this introduction we assume that this condition holds.

\subsection{Free Necklace Polynomials}
To state our main results requires some new noncommutative objects, which we briefly summarise here. 
Given two integers $j,k\geq 0$ such that $n\colonequals j+k\geq 1$, consider the set $\Neck_{j,k}$ consisting of coloured $n$-gons, or necklaces, where there are precisely $j$ nodes shaded $\begin{tikzpicture} \node at (0,0) [G]{};\end{tikzpicture}$ and precisely $k$ unshaded nodes $\begin{tikzpicture} \node at (0,0) [R]{};\end{tikzpicture}$.  

The cyclic group $\mathbb{Z}_n$ acts on the set $\Neck_{j,k}$ via clockwise rotation by $2\uppi/n$, partitioning $\Neck_{j,k}$ into a set orbits, written $\Orb_{j,k}$.  To each $m\in\Orb_{j,k}$ we then prescribe a monomial $p_m$ in the free algebra $\freering$. This is achieved by choosing a representative of the orbit $m$, then starting in the bottom left corner and working clockwise, writing $x$ for $\begin{tikzpicture} \node at (0,0) [G]{};\end{tikzpicture}$ and $y$ for~$\begin{tikzpicture} \node at (0,0) [R]{};\end{tikzpicture}$.  
As an example, $\mathbb{Z}_6$ partitions the 15 elements of $\Neck_{4,2}$ into three orbits: the following illustrates  representatives of the three orbits, together with the corresponding monomial $p_m\in\freering$, and the number of elements in that orbit. 
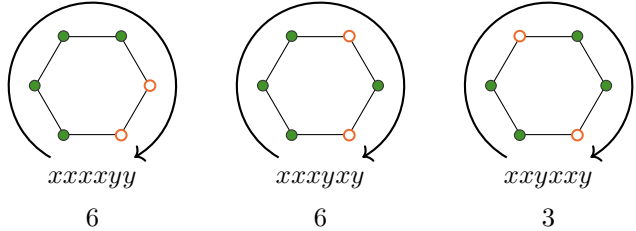
\begin{center}
\begin{tikzpicture}[inner sep=0.5mm,place/.style={circle,draw=blue!50,fill=blue!20,thick}]
\node[name=a, regular polygon, regular polygon sides=6, minimum size=1.5cm, draw] at (0,0) {}; 
\node[name=b, regular polygon, regular polygon sides=6, minimum size=1.5cm, draw] at (3,0) {}; 
\node[name=c, regular polygon, regular polygon sides=6, minimum size=1.5cm, draw] at (6,0) {}; 
\centerarc[thick,<-](0,0)(-60:240:1.1);
\centerarc[thick,<-](3,0)(-60:240:1.1);
\centerarc[thick,<-](6,0)(-60:240:1.1);
\node at (a.corner 1) [G] {};
\node at (a.corner 2) [G] {};
\node at (a.corner 3) [G] {};
\node at (a.corner 4) [G] {};
\node at (a.corner 5) [R] {};
\node at (a.corner 6) [R] {};
\node at (b.corner 1) [R] {};
\node at (b.corner 2) [G] {};
\node at (b.corner 3) [G] {};
\node at (b.corner 4) [G] {};
\node at (b.corner 5) [R] {};
\node at (b.corner 6) [G] {};
\node at (c.corner 1) [G] {};
\node at (c.corner 2) [R] {};
\node at (c.corner 3) [G] {};
\node at (c.corner 4) [G] {};
\node at (c.corner 5) [R] {};
\node at (c.corner 6) [G] {};
\node at (0,-1.75) {$6$};
\node at (3,-1.75) {$6$};
\node at (6,-1.75) {$3$};
\node at (0,-1.25) {$xxxxyy$};
\node at (3,-1.25) {$xxxyxy$};
\node at (6,-1.25) {$xxyxxy$};
\end{tikzpicture} 
\end{center}
Now, given $j,k$, the \emph{free necklace polynomial} $\NeckPoly_{j,k}(x,y)$ is defined to be
\[
\NeckPoly_{j,k}(x,y)\colonequals  \frac{1}{j+k}\sum_{m\in\Orb_{j,k}}|m| \cdot p_m,
\]
which is a well defined element of $\freering$, up to cyclic permutation.  

For its commutative version, consider $\NeckPoly^{\,\ab}_{j,k}$, which is defined to be the image of $\NeckPoly_{j,k}$ under the abelianisation map $\freering\to\mathbb{C}[x,y]$.  The difference between $\NeckPoly_{j,k}$ and  $\NeckPoly^{\,\ab}_{j,k}$ becomes stark as $j+k$ increases.   To calibrate, $\NeckPoly_{4,2}=x^4y^2+x^3yxy+\frac{1}{2}x^2yx^2y$ whilst $\NeckPoly_{4,2}^{\,\ab}=\frac{5}{2}x^4y^2$.

\subsection{Main Results}
Given a curve $\Curve$ in $\scrX$ locally modelled on \eqref{eq!glue!intro}, as explained in \S\ref{sec:mixedDGA}--\ref{sec:DGAC} there is a DG-algebra $\scrC$ that controls the deformation theory of $\scrO_\Curve$, in the sense that its $\Ainf$-products describe the prorepresenting hull of the deformation functor.  There are many such abstract models for $\scrC$, with the challenge being to construct one where the $\Ainf$-products can be calculated.  

Consider the graded vector space $\mathsf{A}=\bigoplus_{i\in\mathbb{Z}}\mathsf{A}_i$ where
\[
\mathsf{A}_i=
\begin{cases}
\mathbb{C} & \mbox{if }i=0,3\\
\mathbb{C}^2 & \mbox{if }i=1,2\\
0&\mbox{else}.
\end{cases}
\]
Write $\xsf,\ysf$ for the basis of $\mathsf{A}_1$, and $\Xsf,\Ysf$ for the basis of $\mathsf{A}_2$, and $\upxi$ for the basis of $\mathsf{A}_3$. The following is our main result.
\begin{thm}[\ref{thm: main}]\label{thm: main intro}
Given a curve $\Curve$ in $\scrX$ locally modelled on \eqref{eq!glue!intro},  the following defines an $\Ainf$-structure $\{\mathsf{m}_n\}_{n\geq 2}$ on the vector space $\mathsf{A}$ above. Furthermore,  the resulting $\Ainf$-algebra is quasi-isomorphic to the \textnormal{DG}-algebra $\scrC=\RHom_\scrX(\scrO_\Curve,\scrO_\Curve)$.
\begin{enumerate}
\item\label{intro:one} For any $n\geq 2$ and any decomposition $n=j+k$ with $j,k\geq 0$,
\[
\msf_{n}(\underbrace{\mathsf{x},\hdots,\mathsf{x}}_j,\underbrace{\mathsf{y},\hdots,\mathsf{y}}_k)=\uplambda_{j+1,k}\,\Xsf+\uplambda_{j,k+1}\Ysf,
\] 
where the $\uplambda$'s are the coefficients from the glue in \eqref{eq!glue!intro}.  
\item\label{intro:two}  More generally, $\msf_n$ with $n\geq 2$ applied only to degree one inputs (so, combinations of $\mathsf{x}$ and $\mathsf{y}$) does not depend on the order of those degree one inputs, and thus is determined by \eqref{intro:one} above. 
\item
The only other non-zero products are 
\[
\begin{array}{lccl}
\msf_2(\xsf,\Xsf) = -\upxi &\quad&\msf_2(\ysf,\Ysf) = -\upxi\phantom{.}   \\
\msf_2(\Xsf,\xsf) = \phantom{-}\upxi &&\msf_2(\Ysf,\ysf) = \phantom{-}\upxi .
\end{array}
\]\end{enumerate}
\end{thm}
Retrospectively, the fact that the order of the inputs does not matter in \ref{thm: main intro}\eqref{intro:two} should be viewed as the remnants of all the data being determined from the geometric setup \eqref{eq!glue!intro}, which is described from commutative gluing data.

\medskip 
There are two main remarkable aspects of \ref{thm: main intro}.  Perhaps the first is that describing the full $\Ainf$-structure is possible at all, never mind with such precision on the coefficients~$\uplambda_{jk}$.  The second is that the proof goes counter to expectations: we do not assume the abstract existence of an $\Ainf$-structure then argue it has nice properties. Instead, the $\Ainf$-structure is constructed from the ground up, using the Kadeishvili algorithm \cite{Kadeishvili}.  One of the main novelties is that it is possible to construct, in \S\ref{sec: construct locally free res},  a uniform locally free resolution $\scrE$ of $\scrO_{\Curve}$. In turn this allows us to choose a uniform basis, which makes the computation of the full $\Ainf$-structure possible.

\medskip
Although the $\Ainf$-structure in \ref{thm: main intro} turns out to be a cyclic $\Ainf$-structure in the sense of Kontsevich--Soibelman \cite{KS}, this is a consequence of the proof, not an input.  It is also unsurprising that there is some reasonable description of at least part of the $\Ainf$-structure, given that the physics literature  \cite{AspinwallKatz,CM, Ferrari} predicts both a commutative and also some form of `matrix' (=noncommutative) superpotential.  As sketched in \S\ref{subsec:physics}, the matrix prediction of \cite{Ferrari} turns out to be consistent with, and mathematically explained by, the main result \ref{thm: main intro}. Much of this paper, and our broader work, is inspired by this physics prediction.

\subsection{Corollaries}
The first main consequences of \ref{thm: main intro} are to deformation theory.  Given a curve $\Curve$ in $\scrX$ locally described by \eqref{eq!glue!intro}, the sheaf $\scrO_\Curve\in\coh\scrX$ has an associated noncommutative (NC) deformation functor, recalled in \S\ref{subsec:NCdeftheory}. This functor always admits a prorepresenting hull, $\Lambda_{\mathrm{def}}$, called the NC deformation algebra (see \S\ref{subsec:NCdeftheory}).

Abstractly $\Lambda_{\mathrm{def}}$ is always a superpotential algebra, however describing the superpotential has been a key open question.  The following gives the first closed-formula description.
\begin{cor}[\ref{cor:NCdefs}]\label{cor:mainintro}
The \textnormal{NC} deformation algebra $\Lambda_{\mathrm{def}}$ of $\scrO_\Curve\in\coh\scrX$ is described by
\[
\Lambda_{\mathrm{def}}\cong\Jac(\mathsf{W})=\frac{\ringtwo}{\lcl\dcyc_{x}\mathsf{W}, \dcyc_{y}\mathsf{W}\rcl}
\]
where $\mathsf{W}=\sum\uplambda_{jk}\NeckPoly_{j,k}\in\ringtwo$ is the sum of free necklace polynomials, and the $\uplambda_{jk}$ are the data in the glue \eqref{eq!glue!intro}.
\end{cor}

The above then immediately implies a classical 1972 result of Namba \cite{Namba, Katz},  proved using complex analysis methods and the existence of Kuranishi spaces.  As is standard, taking the abelianisation of $\Lambda_{\mathrm{def}}$, which just means formally commuting the variables, recovers commutative deformations of $\scrO_\Curve$.  After formally commuting variables, each $p_m$ in $\NeckPoly_{j,k}(x,y)$ becomes $x^jy^k$, and so $\NeckPoly_{j,k}(x,y)^{\ab}=\frac{1}{j+k}{j+k \choose k}x^jy^k$.  

\begin{cor}[\ref{cor:Cdefs}, Namba, Katz]
Classical commutative deformations of $\scrO_\Curve\in\coh\scrX$ are prorepresented by
\[
\Lambda_{\mathrm{def}}^{\ab}\cong\Jac(\mathsf{W})^{\ab}=\frac{\commringtwo}{(\dcyc_{x}\mathsf{V}, \dcyc_{y}\mathsf{V})}
\]
where $\mathsf{V}=\sum\uplambda_{jk}\NeckPoly_{j,k}^{\,\ab}=\sum\frac{\uplambda_{jk}}{j+k}{j+k \choose k}x^jy^k$.
\end{cor}

The other consequence of \ref{thm: main intro} is categorical.  Given a quiver with superpotential $(Q,\mathsf{W})$, Ginzburg associates a 3-CY category $\scrD_\mathsf{W}$.  It is a basic question in mirror symmetry to find geometric models for such categories, on both the A- and B-sides.  
\begin{cor}[\ref{cor:mirror}]\label{cor:mirror intro}
Let $\mathsf{W}\in\mathbb{C}\langle x,y\rangle
$ and consider the associated \textnormal{3-CY} category $\scrD_\mathsf{W}$.  If there exist scalars $\uplambda_{jk}$ for which $\mathsf{W}=\sum\uplambda_{jk}\NeckPoly_{j,k}$,  then there exists a \textnormal{3}-fold $\scrX$ and smooth rational curve $\Curve\subset\scrX$ such that 
\[
\Db(\coh\scrX)\supset \langle \scrO_\Curve\rangle\cong \scrD_\mathsf{W}.
\]
\end{cor}
The above gives the first hint that perhaps not all superpotentials can be realised, whilst simultaneously not giving any hint on what general reasonable restrictions might be, given that necklace polynomials are not preserved under automorphisms.  After imposing strong extra conditions such as $\dim_\mathbb{C}\Lambda_{\mathrm{def}}<\infty$, the realisation problem is expected to be true \cite[1.11]{BW2}.

\subsection{Contractibility}\label{subsec:conjArt}
In 1962 Artin \cite{Artin} established that the contractibility of curves in smooth surfaces is (suitably locally) a combinatorial problem.  This is no longer true for 3-folds. It has been an open question as to what should replace this in higher dimension, and only special cases such as smooth curves currently have a criterion \cite{Jimenez}.

Corollary~\ref{cor:mainintro} gives a presentation of $\Lambda_{\mathrm{def}}$.  Based on empirical evidence from many computer algebra calculations \cite{magma} on these presentations, and from theoretical evidence in \cite{BW2}, we conjecture that in the 3-fold setting Artin's combinatorics are replaced by numerical properties of $\Lambda_{\mathrm{def}}$. One of the advantages of the NC deformation theory technology is that it finally permits clean statements for multi-curves.

\begin{conj}[`3-fold Artin contractibility']\label{conj1}
Let $\Curve$ be a connected union of $n$ smooth rational curves in a \textnormal{CY} \textnormal{3}-fold, and write $\Lambda_{\mathrm{def}}$ for its multi-pointed \textnormal{NC} deformation algebra. Then, in a formal neighbourhood of $\Curve$, the following are equivalent.
\begin{enumerate}
\item $\Curve$ contracts to a point, with an isomorphism elsewhere.
\item $\dim_\mathbb{C}\Lambda_{\mathrm{def}}<\infty$.
\end{enumerate}
\end{conj}
This conjecture is a reinterpretation of Artin's results for surfaces, where crepant curves contract if and only if they are in ADE formation. The NC deformation algebra in this case is the preprojective algebra, which is finite dimensional if and only if the curve configuration is ADE.  Note also for 3-folds that the conjecture is numerical, but not combinatorial: to illustrate this, whilst the superpotential determined in \ref{thm: main intro} is a combinatorial object, determining whether the numerical condition $\dim_\mathbb{C}\Lambda_{\mathrm{def}}<\infty$ holds is much more subtle.

\medskip
The direction ($\Rightarrow$) of \ref{conj1} is known, by \cite{DW1}.  In the case of single curves,  \cite[1.10]{BW2} shows that the only case of ($\Leftarrow$) that remains open is when the normal bundle is $\scrO(-3)\oplus\scrO(1)$. The results of this paper give strong evidence towards, but do not quite yet prove, that final case.

\subsection*{Acknowledgements}
The authors thank Yujiro Kawamata and Franco Rota for helpful conversations, and the referee for many helpful and clarifying remarks.

\subsection*{Conventions}
We work over an algebraically closed field $\mathbb{K}$.  Commas will be suppressed on the subscripts on scalars wherever possible, so e.g.\ $\uplambda_{j,k}=\uplambda_{jk}$.  A further restriction to $\mathbb{C}$, or any algebraically closed field of characteristic zero, is made in \S\ref{sec:DefThCor}.   This is needed for the deformation theory to work nicely, and also to be able to integrate relations to obtain a superpotential, where the free necklace polynomials first appear.

\section{Setup and Preliminaries}

This section introduces $\Curve\subset\scrX$, describes its basic properties, and sets notation that will be used throughout.

\subsection{Gluing of a \texorpdfstring{$(-3,1)$}{(-3,1)} neighbourhood}
Choose finitely many non-zero scalars $\uplambda_{jk}\in\mathbb{K}$ for all $j,k\geq 0$, where $\uplambda_{0,0}=\uplambda_{1,0}=\uplambda_{0,1}=0$, and consider the scheme $\scrX$ defined by gluing the two open  patches 
$\scrU_1=\mathbb{A}^3_{a,v_2,v_1}$ and
$\scrU_2=\mathbb{A}^3_{b,w_2,w_1}$ along
\begin{equation}\label{eq!glue}
(a,v_2,v_1)\stackrel{a\neq 0}{\longmapsto}(a^{-1}, \,a^{-1}v_2, \,\,a^{3}v_1+\sum_{j,k\geq 0} \uplambda_{jk}a^{2-k}v_2^{j+k-1}).
\end{equation}
It is easily checked that the inverse map is given by
\begin{equation}\label{eq!glueinv}
(b,w_2,w_1)\stackrel{b\neq 0}{\longmapsto}(b^{-1},\,b^{-1}w_2,\,\, b^{3}w_1-\sum_{j,k\geq 0} \uplambda_{jk}b^{2-j}w_2^{j+k-1}).
\end{equation}
Write  $\scrU = \{\scrU_1,\scrU_2\}$ for the above open cover, with common open $\scrU_{12}=\scrU_1\cap\scrU_2$.

The ideal $(v_1,v_2)$ glues with the ideal $(w_1,w_2)$ to give an ideal sheaf $\scrI\subset\scrO_\scrX$, the quotient of which defines $\Curve\subset\scrX$.  

\begin{remark}\label{rem:allzerodegen}
If \emph{all} $\uplambda_{j,k}$ are zero, then $\scrX$ is the total space of the vector bundle $\scrO_{\mathbb{P}^1}(-3)\oplus\scrO_{\mathbb{P}^1}(1)$.  Furthermore, in this case the zero section $\Curve=\mathbb{P}^1$ is very well known to have noncommutative deformation algebra $\mathbb{K}\llangle x,y\rrangle$, and commutative deformation algebra $\mathbb{K}\llsq x,y\rrsq$.  In \S\ref{sec:DefThCor} we will interpret this as being given by the zero potential.
\end{remark}

\begin{lemma}\label{lem:qucompactSep}
$\scrX$ is quasi-compact and separated.
\end{lemma}
\begin{proof}
Since $\scrX$ has the finite open cover $\scrU_1\cup\scrU_2$, where each $\scrU_i$ is Spec of a polynomial ring, it follows that $\scrX$ is a noetherian scheme \cite[p83]{Hartshorne}, and thus in particular is quasi-compact \cite[3.1.1, Ex~I.1.7(b)]{Hartshorne}.

Now $\scrX$ is obtained from the separated $\mathbb{K}$-schemes $\scrU_1=\mathbb{A}^3$ and $\scrU_2=\mathbb{A}^3$, glued together via the open subsets $\upphi\colon\{a\neq 0\}\to\{b\neq 0\}$ in \eqref{eq!glue} above.   Now $\scrX\times_\mathbb{K}\scrX$ is covered by $\{\scrU_1\times\scrU_1,\scrU_1\times\scrU_2,\scrU_2\times\scrU_1,\scrU_2\times\scrU_2\}$ \cite[\href{https://stacks.math.columbia.edu/tag/01JS}{Tag 01JS}]{stacks}, and so by the standard  \cite[\href{https://stacks.math.columbia.edu/tag/01KJ}{Tag 01KJ}]{stacks} the scheme $\scrX$ is separated if and only if the `diagonal' map 
\[
\scrU_{12}\to \scrU_1\times\scrU_2
\]
is a closed immersion. Since $\scrU_1$ and $\scrU_2$ are affine, this condition is equivalent to the surjectivity of the map
\[
 \mathbb{K}[a,v_2,v_1]\otimes_\mathbb{K}\mathbb{K}[b,w_2,w_1] \to \mathbb{K}[b^{\pm 1},w_2,w_1]
 \]
 sending $a\otimes b\mapsto \upphi(a)b$.  Since clearly $a^t\otimes w_1^{i_1}w_2^{i_2}\mapsto b^{-t}w_1^{i_1}w_2^{i_2}$ and $1\otimes b^tw_1^{i_1}w_2^{i_2}\mapsto b^tw_1^{i_1}w_2^{i_2}$, the map clearly hits the monomial basis of  $\mathbb{K}[b^{\pm 1},w_2,w_1]$, and thus is surjective.
 \end{proof}

\begin{lemma}[Ferrari {\cite{Ferrari}}]\label{lem:normalbundle}
With the notation as above, 
\[
N_{\Curve|\scrX}\cong
\scrO(-3)\oplus \scrO(1) \iff \uplambda_{20}=\uplambda_{11}=\uplambda_{02}=0.
\]
In the remaining case where some $\uplambda_{20}, \uplambda_{11},\uplambda_{02}$ is nonzero, set $\Updelta=\uplambda_{11}^2-\uplambda_{20}\uplambda_{02}$. 
Then
\[
N_{\Curve|\scrX}\cong
\left\{
\begin{array}{lll}
\scrO(-2)\oplus \scrO &\iff \Updelta= 0\\
\scrO(-1)\oplus \scrO(-1) &\iff \Updelta\neq 0.
\end{array}
\right.
\]
\end{lemma}
\begin{proof}
Working mod $\scrI^2=(v_1,v_2)^2$, the glue \eqref{eq!glue!intro} becomes
\begin{equation}
(a,v_2,v_1)\stackrel{a\neq 0}{\longleftrightarrow}(a^{-1}, \,a^{-1}v_2, \,\,a^{3}v_1+\sum_{j+k\leq 2} \uplambda_{jk}a^{2-k}v_2^{j+k-1}).\label{eqn:gluemodI2}
\end{equation}
Since $\uplambda_{00}=\uplambda_{10}=\uplambda_{01}=0$ by convention, by inspection of the known gluing of the total space $\scrO_{\mathbb{P}^1}(-3)\oplus\scrO_{\mathbb{P}^1}(1)$, certainly the above curve has normal bundle $\scrO(-3)\oplus\scrO(1)$ if the displayed sum in the right hand side of \eqref{eqn:gluemodI2} is zero, equivalently $\uplambda_{20}=\uplambda_{11}=\uplambda_{02}=0$.

If one of $\uplambda_{20}, \uplambda_{11}, \uplambda_{02}$ is nonzero, then mod $\scrI^2$ the gluing is 
\begin{equation}
(a,v_2,v_1)\stackrel{a\neq 0}{\longleftrightarrow}(a^{-1}, \,a^{-1}v_2, \,\,a^{3}v_1+(\uplambda_{20}a^{2}+\uplambda_{11}a+\uplambda_{02})v_2).\label{eqn:gluemodI2B}
\end{equation}
It is then a result of Ferrari \cite[Appendix B]{Ferrari} that the normal bundle is $(r-1,-r-1)$, where $r$ is the corank of the quadratic form $\begin{psmallmatrix}\uplambda_{20}&\uplambda_{11}\\\uplambda_{11}&\uplambda_{02}\end{psmallmatrix}$, and so the result follows.
\end{proof}

\subsection{The A and B Polynomials}\label{subsec:AandB}

From here on, we will consider the following setup.
\begin{setup}\label{setup:key}
Consider the glue \eqref{eq!glue}, which by definition is given by specifying finitely many nonzero $\uplambda_{jk}$.  We will assume that:
\begin{enumerate}
\item\label{setup:key1}  $\uplambda_{00}=\uplambda_{10}=\uplambda_{01}=0$, to ensure the existence of a closed curve.
\item\label{setup:key2}  $\uplambda_{20}=\uplambda_{11}=\uplambda_{02}=0$, by \ref{lem:normalbundle} to ensure that the normal bundle is $\scrO(-3)\oplus \scrO(1)$.
\item\label{setup:key3} That not all $\uplambda_{jk}$ are zero, to exclude the easy degenerate case described in \ref{rem:allzerodegen}.
\end{enumerate}
\end{setup}
Under \ref{setup:key}, the following three constants $\ordP, \kx, \ky\in\N$ will naturally appear throughout the analysis, as will the following polynomials $A$ and $B$.

\begin{notation}\label{not: key abc notation}
Set $\ordP=\mathrm{min}\{ j+k \mid \uplambda_{j,k}\neq 0\}$.   By \ref{setup:key}\eqref{setup:key3} $\ordP$ is defined, and by \ref{setup:key}\eqref{setup:key1}\eqref{setup:key2}, $\ordP\geq 3$.
Furthermore, we may write
\begin{equation}
\begin{aligned}
\sum \uplambda_{jk}a^{2-k}v_2^{j+k-1} &= a^{2-\kx}v_2^{\ordP-1}A\\
\sum \uplambda_{jk}b^{2-j}w_2^{j+k-1} &= b^{2-\ky}w_2^{\ordP-1}B\label{eqn:first eq}
\end{aligned}
\end{equation}
where $A=A(a,v_2)\in\scrO_{\scrU_1}$ and $B=B(b,w_2)\in\scrO_{\scrU_2}$ are \emph{polynomials},
and neither $a$ nor $v_2$ divides $A$, and neither $b$ nor $w_2$ divides $B$.
In particular, $2-\kx\in\Z$ is the most negative power of $a$ appearing in \eqref{eq!glue}
and $2-\ky$ is the most negative power of $b$ in \eqref{eq!glueinv}. 
Again using \ref{setup:key}\eqref{setup:key1}\eqref{setup:key2},  both $\kx\geq 0$ and $\ky\geq 0$.
\end{notation}

\begin{lemma}\label{eq!BD}
Restricting $A\in\scrO_{\scrU_1}$ and $B\in\scrO_{\scrU_2}$ to the common open subset $\scrO_{\scrU_{12}}$,
then $B = b^{\kx+\ky-\ordP}A$, equivalently $A = a^{\kx+\ky-\ordP}B$.
\end{lemma}
\begin{proof}
In the notation of \ref{not: key abc notation}, the final components of the glue equations \eqref{eq!glue} and \eqref{eq!glueinv} can be written
\begin{equation}\label{eq!complex}
\begin{aligned}
b^{2-\kx}w_1 &= a^{\kx+1}v_1 + v_2^{\ordP-1}A\\
a^{2-\ky}v_1 &= b^{\ky+1}w_1 - w_2^{\ordP-1}B.
\end{aligned}
\end{equation}
Rearranging for $A$ and $B$ gives the result.
\end{proof}

The following more refined notation will be needed for inductive arguments later.
\begin{notation}\label{def:AandB}
Given \eqref{eqn:first eq}, for all $i\geq 3$ define $A_i\in\mathbb{K}[a]$ and $B_i\in\mathbb{K}[b]$ by 
\begin{align*}
A_i&=\uplambda_{i,0}a^{\kx}+\uplambda_{i-1,1}a^{\kx-1}+\hdots+\uplambda_{0,i}a^{\kx-i}  \\
B_i&=\uplambda_{0,i}b^{\ky}+\uplambda_{1,i-1}b^{\ky-1}+\hdots+\uplambda_{i,0}b^{\ky-i} .
\end{align*}
\end{notation}
By definition of $\ordP$ (in \ref{not: key abc notation}), $A_3=\hdots=A_{\ordP-1}=0$ and $B_3=\hdots=B_{\ordP-1}=0$.  
The following is the graded piece analogue of \ref{eq!BD}.
\begin{lemma}\label{lem:AiBi1}
For all $i\geq 3$, we have $A_i=a^{\kx+\ky-i}B_i$ and so $v_2^{i-\ordP}A_i=a^{\kx+\ky-\ordP} w_2^{i-\ordP}B_i$.  
\end{lemma}
\begin{proof}
The statement $A_i=a^{\kx+\ky-i}B_i$ is immediate by inspection, since $a=b^{-1}$ on $\scrU_{12}$.  The second statement follows, since by the glue \eqref{eq!glue} $w_2=a^{-1}v_2$.
\end{proof}
Inspecting the graded pieces of \eqref{eqn:first eq}, note that
\[
\sum_{j+k=i} \uplambda_{jk}a^{2-k}v_2^{j+k-1} = a^{2-\kx}v_2^{i-1}A_i
\quad\text{and}\quad
\sum_{j+k=i}  \uplambda_{jk}b^{2-j}w_2^{j+k-1} = b^{2-\ky}w_2^{i-1}B_i.
\]
Consequently, the $A$ and $B$ defined in \ref{not: key abc notation} can be written
\begin{align}
A&=A_{\ordP}+v_2A_{\ordP+1}+v_2^2A_{\ordP+2}+\hdots\nonumber \\
&=\underbrace{v_2^{3-\ordP}A_3+v_2^{4-\ordP}A_4+\hdots+v_2^{-1}A_{\ordP-1}}_{=0}+\underbrace{A_{\ordP}}_{\ne0}
+v_2A_{\ordP+1}+v_2^2A_{\ordP+2}+\hdots \label{eqn:Aassum}
\end{align}
and similarly $B=w_2^{3-\ordP}B_3+\hdots+w_2^{-1}B_{\ordP-1}+B_{\ordP}+w_2B_{\ordP+1}+w_2^2B_{\ordP+2}+\hdots$ with $B_t\ne0$ where all terms $w_2^{-i}B_{t-i}$ with negative $w_2$ exponents are zero.

\subsection{Sheaves on the neighbourhood}\label{sec:sheavesnbd}
For $n\in\mathbb{Z}$, consider the locally free sheaves $\scrO(n)$ defined on $\scrX$ by taking the rank one free module on $\scrU_2$ and the rank one free module on $\scrU_1$, and gluing them on the common intersection $\scrU_{12}$ via the isomorphism
\[
\begin{array}{rcl}
\mathbb{K}[b^{\pm 1},w_2,w_1]&\to&\mathbb{K}[a^{\pm 1},v_2,v_1]\\
f(b,w_2,w_1)&\mapsto & a^{n}\cdot f(a^{-1},a^{-1}v_2,a^{3}v_1+\sum_{j,k\geq 0} \uplambda_{jk}a^{2-k}v_2^{j+k-1}).
\end{array}
\]
A map between sheaves $\phi\colon \scrO(m) \rightarrow \scrO(n)$ is determined
by two polynomial maps  in a diagram
\begin{equation}\label{eqn:glueconvention}
\begin{array}{c}
\begin{tikzpicture}[scale=0.75]
\node (A) at (0,2) {On $\scrU_1$:};
\node (B) at (0,0) {On $\scrU_2$:};
\node (A1) at (3,2) {$\scrO_{\scrU_1}(m)$};
\node (A2) at (8,2) {$\scrO_{\scrU_1}(n)$};
\node (B1) at (3,0) {$\scrO_{\scrU_2}(m)$};
\node (B2) at (8,0) {$\scrO_{\scrU_2}(n)$};
\draw[->] (B1)--node[left]{$\scriptstyle a^m$}(A1);
\draw[->] (B2)--node[left]{$\scriptstyle  a^n$}(A2);
\draw[->] (A1)--node[above]{$\scriptstyle \phi_1(a,v_2,v_1)$}(A2);
\draw[->] (B1)--node[above]{$\scriptstyle \phi_2(b,w_2,w_1)$}(B2);
\end{tikzpicture}
\end{array}
\end{equation}
that commutes after substituting for glueing expressions \eqref{eq!glue}.
Although it is natural to define $\phi_1$ and $\phi_2$ in the coordinates on their patch,
we may use expressions in any of the variables that remain regular on the given patch
and commute up to the glue, e.g. in \ref{thm: locallyfreeres} below. The above picture will be abbreviated
\[
\begin{tikzpicture}
\node (A) at (0,0) {$\phi\colon \scrO(m)$};
\node (B) at (3.5,0) {$\scrO(n)$};
\draw[->] (A)--node[above]{$\scriptstyle \phi_1(a,v_2,v_1)$}node[below]{$\scriptstyle \phi_2(b,w_2,w_1)$}(B);
\end{tikzpicture}
\]
leaving the transition maps implicit, and similarly maps between direct sums of such bundles will be represented
\[
\begin{tikzpicture}
\node (A) at (0,0) {$\bigoplus\scrO(m_i)$};
\node (B) at (3.5,0) {$\bigoplus\scrO(n_j)$};
\draw[->] (A)--node[above]{$\scriptstyle \Phi_1(a,v_2,v_1)$}node[below]{$\scriptstyle \Phi_2(b,w_2,w_1)$}(B);
\end{tikzpicture}
\]
where $\Phi_i$ are represented by matrices representing the maps on the two charts,
and the transition maps are given by the matrices $\diag(a^{m_i})$ and $\diag(a^{n_j})$.

\section{Constructing the DG-model}

For any choice of scalars $\uplambda_{jk}$ satisfying \ref{setup:key}, consider the scheme $\scrX$ defined by the gluing rule in \eqref{eq!glue}, which contains $\Curve\cong\mathbb{P}^1$.  This section builds a particular DG-algebra that controls the deformation theory of $\scrO_\Curve$.  

In \S\ref{sec: construct locally free res} below, which is the key new construction, we exhibit a uniform locally free resolution $\scrE$ of $\scrO_\Curve$.  Using this, together with standard results involving homological DG-algebras (in \S\ref{subsec:HomDGA}) and  \v{C}ech enhancements (in \S\ref{sec:mixedDGA}), we then exhibit a complex $(\scrC,\dd)$ which computes the modules $\Ext_\scrX^i(\scrO_\Curve,\scrO_\Curve)$. After some some sign adjustments to multiplication (in \S\ref{sec:DGAC}), we obtain an explicit DG-algebra $(\scrC,\star,\dd)$ that models $\RHom_\scrX(\scrO_\Curve,\scrO_\Curve)$.

\subsection{Locally Free Resolution of \texorpdfstring{$\scrO_\Curve$}{OC}}\label{sec: construct locally free res}

This subsection constructs a locally free resolution of the structure sheaf $\scrO_{\Curve}$.
Recall the polynomials $A\in\C[a,v_2]$ and $B\in\C[b,w_2]$ from \ref{not: key abc notation}
and the abbreviated notation for maps of sheaves on $\scrX=\scrU_1\cup\scrU_2$
from \S\ref{sec:sheavesnbd}.

\begin{thm}\label{thm: locallyfreeres}
Under Setup~\textnormal{\ref{setup:key}}, the following is a locally free resolution of $\scrO_{\Curve}$.
\[
\begin{tikzpicture}
\node (A) at (-1.85,0) {$0$};
\node (B) at (-0.5,0) {$\scrO({\scriptstyle -\kx-\ky})$};
\node (C1) at (2.75,0.75) {$\scrO({\scriptstyle 1-\kx-\ky})$};
\node () at (2.75,0.4) {$\oplus$};
\node (C2) at (2.75,0) {$\scrO({\scriptstyle1-\ky})$};
\node () at (2.75,-0.35) {$\oplus$};
\node (C3) at (2.75,-0.75) {$\scrO({\scriptstyle1-\kx})$};
\node (D1) at (8,0.8) {$\scrO({\scriptstyle2-\ky})$};
\node () at (8,0.4) {$\oplus$};
\node (D2) at (8,0) {$\scrO({\scriptstyle-1})$};
\node () at (8,-0.4) {$\oplus$};
\node (D3) at (8,-0.8) {$\scrO({\scriptstyle2-\kx})$};
\node (E) at (11.25,0) {$\scrO$};
\draw[->] (A)--(B);
\draw[->] (B)--
node[above]{$\begin{psmallmatrix}v_2 \\ -a^{\kx+1}\\ -1\end{psmallmatrix}$}
node[below]{$\begin{psmallmatrix}w_2 \\ -1\\ -b^{\ky+1}\end{psmallmatrix}$}
(C2);
\draw[->] (C2)--
node[above]{$\begin{psmallmatrix}
a^{\kx+1}&v_2&0\\
v_2^{\ordP -2}\!A&-v_1&b^{2-\kx}w_1\\
-1&0&-v_2\end{psmallmatrix}$}
node[below]{$\begin{psmallmatrix}
1&w_2&0\\
w_2^{\ordP-2}\!B&-a^{2-\ky}v_1&w_1\\
-b^{\ky+1}&0&-w_2\end{psmallmatrix}$}
(D2);
\draw[->] (D2)--
node[above]{$\begin{psmallmatrix}v_1 & v_2& b^{2-\kx}w_1\end{psmallmatrix}$}
node[below]{$\begin{psmallmatrix}a^{2-\ky}v_1 &w_2& w_1\end{psmallmatrix}$}
(E);
\end{tikzpicture}
\]
where in the top $b^{2-\kx}w_1=a^{\kx+1}v_1 + v_2^{\ordP-1}A$, and in the bottom $a^{2-\ky}v_1= b^{\ky+1}w_1 - w_2^{\ordP-1}B$.
\end{thm}

\begin{proof}
We first check that the given maps on charts glue to maps of sheaves as indicated.
On the left, since $w_2=a^{-1}v_2$ and $b=a^{-1}$, we have
\[
\begin{psmallmatrix}a^{1-\kx-\ky}&0&0\\0&a^{1-\ky}&0\\0&0&a^{1-\kx}\end{psmallmatrix}
\begin{psmallmatrix}a^{-1}v_2\\-1\\-a^{-\ky-1}\end{psmallmatrix} =
\begin{psmallmatrix}a^{-\kx-\ky}v_2\\-a^{1-\ky}\\-a^{-\kx-\ky}\end{psmallmatrix}=
\begin{psmallmatrix}v_2\\-a^{\kx+1}\\-1\end{psmallmatrix}
\cdot a^{-\kx-\ky}
\]
so the local charts glue to a map $\scrO(-\kx-\ky)\rightarrow\scrO(1-\kx-\ky)\oplus\scrO(1-\ky)\oplus\scrO(1-\kx)$.

For the middle map, substituting also for $B = a^{\ordP-\kx-\ky}A$ by \ref{eq!BD}, we have
\begin{align*}
\begin{psmallmatrix}a^{\kx+1}&v_2&0\\v_2^{\ordP -2}\!A&-v_1&b^{2-\kx}w_1\\-1&0&-v_2\end{psmallmatrix}
\begin{psmallmatrix}a^{1-\kx-\ky}&0&0\\0&a^{1-\ky}&0\\0&0&a^{1-\kx}\end{psmallmatrix}
&=
\begin{psmallmatrix}a^{2-\ky}&a^{1-\ky}v_2&0\\a^{1-\kx-\ky}v_2^{\ordP -2}\!A&-a^{1-\ky}v_1&a^{-1}w_1\\-a^{1-\kx-\ky}&0&-a^{1-\kx}v_2\end{psmallmatrix} \\
&=\begin{psmallmatrix}a^{2-\ky}&0&0\\0&a^{-1}&0\\0&0&a^{2-\kx}\end{psmallmatrix}
\begin{psmallmatrix}
1&a^{-1}v_2&0\\
a^{2-\kx-\ky}v_2^{\ordP-2}\!A&-a^{2-\ky}v_1&w_1\\
-a^{-\ky-1}&0&-a^{-1}v_2\end{psmallmatrix}
\end{align*}
so the local charts glue to a map as indicated.
Finally, on the right,
\[
\begin{psmallmatrix}v_1&v_2&a^{\kx-2}w_1\end{psmallmatrix} 
\begin{psmallmatrix}a^{2-\ky}&0&0\\0&a^{-1}&0\\0&0&a^{2-\kx}\end{psmallmatrix} =
\begin{psmallmatrix}a^{2-\ky}v_1&a^{-1}v_2&w_1\end{psmallmatrix}
\]
so the local charts glue to a map to the trivial bundle.
The sequence of maps of sheaves is clearly a complex, so it remains to check that it is exact and has cokernel $\scrO_\Curve$.

Exactness of the complex is local, and it is enough to check exactness on the two patches separately. We do this for the first patch in coordinates $\scrO_{\scrU_1}=\C[a,v_2,v_1]$, with the
second patch being similar. The left-most nonzero map is clearly injective. If $(r_1,r_2,r_3)^T$ belongs to the kernel of $\scrO^3\rightarrow\scrO^3$, then
$r_1=-v_2r_3$ and
\[
0 = a^{\kx+1}r_1 + v_2r_2 = (-a^{\kx+1}r_3 + r_2)v_2
\]
so that $r_2 = a^{\kx+1}r_3$, as $\scrO_{\scrU_1}$ is a domain.
Thus $(r_1,r_2,r_3)^T$ lies in the image and the complex is exact at this point.

If $(r_1,r_2,r_3)^T$ belongs to the kernel of the right-hand map $\scrO^3\rightarrow\scrO$ then
by \eqref{eq!complex}
\[
v_1r_1 + v_2r_2 + (a^{\kx+1}v_1 + v_2^{\ordP-1}A)r_3 = 0.
\]
Rewriting this as $(r_1+a^{\kx+1}r_3)v_1 = (-r_2 - v_2^{\ordP-2}Ar_3)v_2$
shows that $r_1+a^{\kx+1}r_3 = Cv_2$ and $-r_2-v_2^{\ordP-2}Ar_3 = Cv_1$ for some $C\in\scrO_{\scrU_1}$, since $\scrO_{\scrU_1}$ is a UFD.
Thus $(r_1,r_2,r_3)^T$ is the image of $(-r_3,C,0)^T$, as required.

Finally, it is clear that $\scrO_\Curve$ is the cokernel, since the equation $b^{2-\kx}w_1=a^{\kx+1}v_1 + v_2^{\ordP-1}A$, together with $\ordP\ge2$, shows that $b^{2-\kx}w_1\in(v_1,v_2)=I_C$.
\end{proof}

\subsection{Homological DG-algebra}\label{subsec:HomDGA}
To immediately ease notation, write
\[
\scrE\colonequals\quad
0\to\scrE_3 \xrightarrow{\dsf_{3}} \scrE_{2} \xrightarrow{\dsf_{2}} \scrE_{1} \xrightarrow{\dsf_{1}} \scrE_0 \to 0
\]
for the complex of locally free sheaves on $\scrX$ constructed in \ref{thm: locallyfreeres}, and consider the homological DG-algebra $(\scrEnd_\scrX(\scrE),\circ, \updelta)$ defined by $\scrEnd_\scrX(\scrE)=\bigoplus_{i\in\mathbb{Z}}\scrHom^i(\scrE,\scrE) $, where
\begin{equation}
\scrHom^i(\scrE,\scrE) = \bigoplus_{n\in\Z} \scrHom(\scrE_n,\scrE_{n-i}).\label{eqn:homgE}
\end{equation}
By \eqref{eqn:homgE}, a homogeneous element $\bsf\in\scrHom^i(\scrE,\scrE)$ of $\scrEnd_\scrX(\scrE)$ decomposes as $\bsf = (\bsf_n)_{n\in\Z}$ with $\bsf_n\in\scrHom(\scrE_n,\scrE_{n-i})$, and we will refer to $\bsf_n$ as the {\em $n$th component} of the homogenous element $\bsf$.

Composition $\circ$ gives rise to a product on $\scrEnd_\scrX(\scrE)$ which preserves the grading.  Further, there is a differential  $\updelta\colon \scrHom^i(\scrE,\scrE)\to\scrHom^{i+1}(\scrE,\scrE)$ whose $n$th component is
\begin{equation}\label{eqn:updelta}
(\updelta\asf)_n \colonequals \dsf_{n-i}\circ \asf_n - (-1)^i \asf_{n-1} \circ \dsf_{n}.
\end{equation}
In the following picture, visually this is down-across $-(-1)^{i}$ across-down.
\[
\begin{tikzpicture}
\node (B1) at (-0.25,1) {$\scrE_n$};
\node (C1) at (3.25,1) {$\scrE_{n-1}$};
\draw[->] (B1)--node[above]{$\scriptstyle\dsf_n$}(C1);
\node (C2) at (-0.25,-1) {$\scrE_{n-i}$};
\node (D2) at (3.25,-1) {$\scrE_{n-i-1}$};
\draw[->] (C2)--node[below]{$\scriptstyle\dsf_{n-i}$}(D2);
\draw[->](B1)--node[left,pos=0.4]{$\scriptstyle\asf_n$}(C2);
\draw[->](C1)--node[right,pos=0.4]{$\scriptstyle\asf_{n-1}$}(D2);
\end{tikzpicture}
\]

\begin{lemma}\label{lem:updelta}
$\updelta^2(\asf)=0$ and $\updelta(\asf_{|\scrU}) = \updelta(\asf)_{|\scrU}$ 
for any $\asf\in\scrHom^i(\scrE,\scrE)$ and any open $\scrU\subset\scrX$.
\end{lemma}
\begin{proof}
The first claim is standard, and the second claim is the statement that $\dsf_n$
and $\asf$ are maps of sheaves, so are defined by composition on the open $\scrU$.
\end{proof}

\subsection{\v{C}ech complex}
This subsection sets notation for the \v{C}ech complex $(\vC(\scrU,\scrF),\mathfrak{d})$ of a sheaf $\scrF$ on $\scrX$ relative to the open cover $\scrX=\scrU_1\cup\scrU_2$.  Recall that $\scrU_{12}=\scrU_1\cap\scrU_2$.

The only nonzero terms in the \v{C}ech complex are
$\vC^0 = \scrF|_{\scrU_1}\oplus\scrF|_{\scrU_2}$ and $\vC^1 = \scrF|_{\scrU_{12}}$.
For $(\asf_1,\asf_2)\in\vC^0$, the coboundary $\mathfrak{d}\colon\vC^0\rightarrow\vC^1$
is the map
\begin{equation}\label{eqn:frakdelta}
\mathfrak{d}(\asf_1,\asf_2) = {\asf_1}_{|\scrU_{12}}-{\asf_2}_{|\scrU_{12}}.
\end{equation}

\subsection{The total complex}\label{sec:mixedDGA}
Following \cite{AspinwallKatz}, consider $\scrC = \bigoplus_{n\in\mathbb{Z}} \scrC_n$ where
\[
\scrC_n \colonequals \bigoplus_{p+q=n} \vC^p(\scrU,\scrHom^q(\scrE,\scrE)),
\]
with differential $\dd =\mathfrak{d} + (-1)^p\updelta$ described below. This is the \v{C}ech complex of the sheaf of DG-algebras $\scrEnd_\scrX(\scrE)$ from \S\ref{subsec:HomDGA}, c.f.\ \cite[01FP]{stacks}.

Unpacking this, since only $\vC^0$ and $\vC^1$ are nonzero, a homogenous element $\asf\in\scrC_i$ is a triple of vectors of homomorphisms
\[
\asf = (\asf_1, \asf_2, \asf_{12})
\]
where 
\[
\asf_1 \in \bigoplus_{n\in\Z} \scrHom_{\scrU_1}(\scrE_n,\scrE_{n-i}),\,\,
\asf_2 \in \bigoplus_{n\in\Z} \scrHom_{\scrU_2}(\scrE_n,\scrE_{n-i}),\,\,
\asf_{12} \in \bigoplus_{n\in\Z} \scrHom_{\scrU_{12}}(\scrE_n,\scrE_{n-i+1}) 
\]
The differential $\dd\colon \scrC_i\rightarrow\scrC_{i+1}$ is defined by
\begin{equation}
(\asf_1,\asf_2,\asf_{12})\mapsto (\updelta \asf_1, \updelta \asf_2, (\asf_{1|\scrU_{12}} - \asf_{2|\scrU_{12}}) - \updelta \asf_{12})\label{eqn:defdd}
\end{equation}
where $\updelta$ is from \eqref{eqn:updelta} and we have applied $\dfrak$ from \eqref{eqn:frakdelta}.


As is well-known, the complex $(\scrC,\dd)$ computes the cohomology $\Ext^i_\scrX(\scrO_\Curve,\scrO_\Curve)$.
\begin{prop}\label{prop:Exti}
$\mathrm{H}^i(\scrC)\cong \Ext^i_\scrX(\scrO_\Curve,\scrO_\Curve)$ for all $i\in\mathbb{Z}$.
\end{prop}
\begin{proof}
$(\scrC,\dd)$ as defined above is equal to the complex $\Hom^\bullet_{\mathrm{D}^{\mathrm{b}}_{\infty}(\scrX)}(\scrE,\scrE)$ as defined in \cite[\S8.2.1]{Dbook}.  Thus
\[
\mathrm{H}^i(\scrC)\cong \mathrm{H}^i(\Hom^\bullet_{\mathrm{D}^{\mathrm{b}}_{\infty}(\scrX)}(\scrE,\scrE))\cong \Hom_{\Db(\scrX)}(\scrE,\scrE[i])\cong \Ext^i_\scrX(\scrO_\Curve,\scrO_\Curve),
\] 
where the middle isomorphism is e.g.\ \cite[p587]{Dbook}.
\end{proof}

\subsection{The \v{C}ech enhancement}\label{sec:DGAC}
Using the well-known \v{C}ech enhancement described in e.g.\ \cite[p587]{Dbook} or \cite[\S1.2]{CS}, the complex $(\scrC,\dd)$ can be upgraded to a DG-algebra.  The only slightly subtle point is the sign on the composition, which we explicitly recall here.

\begin{defin}
Given homogeneous $\asf=(\asf_1,\asf_2,\asf_{12})\in\scrC_i$
and $\bsf=(\bsf_1,\bsf_2,\bsf_{12})\in\scrC_j$, define
\begin{equation}\label{eqn:star}
\asf\star \bsf \colonequals
(\asf_1\circ\bsf_1, \,\,\,\asf_2\circ\bsf_2,\,\,\, \asf_{12}\circ {\bsf_2}_{|\scrU_{12}} + (-1)^i{\asf_1}_{|\scrU_{12}}\circ \bsf_{12})
\in\scrC_{i+j}
\end{equation}
and extend $\star$ to all of $\scrC$ by linearity.
\end{defin}

To set notation for the next result,  choose an injective resolution
\[
0\to\scrO_\Curve\to\scrI_0\to\scrI_1\to\hdots
\] 
of $\scrO_\Curve$. Then consider the DG-algebra $\End^{\mathrm{DG}}_\scrX(\scrI)=\bigoplus_{t\in\mathbb{Z}}\End^{\mathrm{DG}}_\scrX(\scrI)_t$, where
\[
\End^{\mathrm{DG}}_\scrX(\scrI)_t \colonequals \{ (f_s)_{s\in\mathbb{Z}}\mid f_s\colon \scrI_s\to \scrI_{s+t}\}.
\]
Multiplication is given by composition, and the differential is defined as in \S\ref{subsec:HomDGA}.

\begin{prop}\label{prop:DGAcheck}
$(\scrC,\star,\dd)$ is a \textnormal{DG}-algebra, and this is quasi-isomorphic to the \textnormal{DG}-algebra $\End^{\mathrm{DG}}_\scrX(\scrI)$.
\end{prop}
\begin{proof}
The first statement can be checked manually. Alternatively, as in the proof of \ref{prop:Exti}, $(\scrC,\dd)$ equals the complex $\Hom^\bullet_{\mathrm{D}^{\mathrm{b}}_{\infty}(\scrX)}(\scrE,\scrE)$.  Now $\mathrm{D}^{\mathrm{b}}_{\infty}(\scrX)$ is in fact a DG category, under the composition described in \cite[p587]{Dbook}.  In our restricted setting, with only two open affine sets, this translates precisely into the operation $\star$ defined above (see also \cite{AspinwallKatz}). The second statement follows since $\scrX$ is quasi-compact and separated by \ref{lem:qucompactSep}, so the \v{C}ech DG enhancement of perfect complexes is quasi-equivalent to the injective DG enhancement of perfect complexes; see \cite[\S1.2]{CS} or \cite[3.19]{LS}.
\end{proof}

\section{Generators and Homotopies}

This section constructs certain elements of the DG-algebra $(\scrC,\star,\dd)$ from \S\ref{sec:DGAC}.  In \S\ref{subsec:degone}--\ref{subsec:degthree} some key elements of $\scrC$ are introduced. Various easy relations between these elements involving both $\dd$ and $\star$ are then computed in \S\ref{sec:dd} and \S\ref{sec:star}.  Inductive notation is introduced in \S\ref{sec:inductivesetup}. The totality of elements considered is summarised in \S\ref{rem:shapes2}. The whole section is elementary, essentially being nothing more than multiplication of matrices, with care taken about degrees.

To set and ease notation throughout, under Setup~\ref{setup:key} write
\[
0\to\scrE_3 \xrightarrow{\dsf_{3}} \scrE_{2} \xrightarrow{\dsf_{2}} \scrE_{1} \xrightarrow{\dsf_{1}} \scrE_0 \to 0
\]
for the complex of locally free sheaves on $\scrX$ constructed in \ref{thm: locallyfreeres}
that resolves $\scrO_\Curve$.

\subsection{Degree One Generators}\label{subsec:degone}
Consider $\xsf \colonequals (\xsf_1, \xsf_2, 0) \in\scrC_1$, where
\[
\begin{tikzpicture}
\node () at (-2.75,0) {$\xsf _1=$};
\node (B1) at (-1,1) {$\scrE_3$};
\node (C1) at (3.25,1) {$\scrE_2$};
\node (D1) at (7,1) {$\scrE_1$};
\node (E1) at (8,1) {$\scrE_0$};
\draw[->] (B1)--
node[above]{$\begin{psmallmatrix}v_2 \\ -a^{\kx+1}\\ -1\end{psmallmatrix}$}
(C1);
\draw[->] (C1)--
node[above]{$\begin{psmallmatrix}
a^{\kx+1}&v_2&0\\
v_2^{\ordP-2}\!A&-v_1&b^{2-\kx}w_1\\
-1&0&-v_2\end{psmallmatrix}$}
(D1);
\draw[->] (D1)--
(E1);

\node (B2) at (-2,-1.5) {$\scrE_3$};
\node (C2) at (-1,-1.5) {$\scrE_2$};
\node (D2) at (3.25,-1.5) {$\scrE_1$};
\node (E2) at (7,-1.5) {$\scrE_0$};
\draw[->] (B2)--
(C2);
\draw[->] (C2)--
node[above]{$\begin{psmallmatrix}
a^{\kx+1}&v_2&0\\
v_2^{\ordP-2}\!A&-v_1&b^{2-\kx}\!w_1\\
-1&0&-v_2\end{psmallmatrix}$}
(D2);
\draw[->] (D2)--
node[above]{$\begin{psmallmatrix}v_1 & v_2& b^{2-\kx}w_1\end{psmallmatrix}$}
(E2);

\draw[->](B1)--node[left,pos=0.35]{$\begin{psmallmatrix}-1\\0\\0\end{psmallmatrix}$}(C2);
\draw[->](C1)--node[gap,pos=0.35]{$\begin{psmallmatrix}0&-1&0\\v_2^{\ordP-3}A&0&0\\0&0&1\end{psmallmatrix}$}(D2);
\draw[->](D1)--node[right,pos=0.35]{$\begin{psmallmatrix}0&-1&0\end{psmallmatrix}$}(E2);
\end{tikzpicture}
\]
\[
\begin{tikzpicture}
\node () at (-2.75,0) {$\xsf_2=$};
\node (B1) at (-1,1) {$\scrE_3$};
\node (C1) at (3.25,1) {$\scrE_2$};
\node (D1) at (7,1) {$\scrE_1$};
\node (E1) at (8,1) {$\scrE_0$};
\draw[->] (B1)--
node[below]{$\begin{psmallmatrix}w_2 \\ -1\\ -b^{\ky+1}\end{psmallmatrix}$}
(C1);
\draw[->] (C1)--
node[below]{$\begin{psmallmatrix}
1&w_2&0\\
w_2^{\ordP-2}\!B&-a^{2-\ky}v_1&w_1\\
-b^{\ky+1}&0&-w_2\end{psmallmatrix}$}
(D1);
\draw[->] (D1)--
(E1);

\node (B2) at (-2,-1.5) {$\scrE_3$};
\node (C2) at (-1,-1.5) {$\scrE_2$};
\node (D2) at (3.25,-1.5) {$\scrE_1$};
\node (E2) at (7,-1.5) {$\scrE_0$};
\draw[->] (B2)--
(C2);
\draw[->] (C2)--
node[below]{$\begin{psmallmatrix}
1&w_2&0\\
w_2^{\ordP-2}\!B&-a^{2-\ky}v_1&w_1\\
-b^{\ky+1}&0&-w_2\end{psmallmatrix}$}
(D2);
\draw[->] (D2)--
node[below]{$\begin{psmallmatrix}a^{2-\ky}v_1 &w_2& w_1\end{psmallmatrix}$}
(E2);

\draw[->](B1)--node[left,pos=0.65]{${\scriptstyle b}\begin{psmallmatrix}-1\\0\\0\end{psmallmatrix}$}(C2);
\draw[->](C1)--node[gap,pos=0.65]{${\scriptstyle b}\begin{psmallmatrix}0&-1&0\\w_2^{\ordP-3}B&0&0\\0&0&1\end{psmallmatrix}$}(D2);
\draw[->](D1)--node[right,pos=0.65]{${\scriptstyle b}\begin{psmallmatrix}0&-1&0\end{psmallmatrix}$}(E2);
\end{tikzpicture}
\]
Under the sign convention defining $\updelta$ in \eqref{eqn:updelta}, which for odd degree is down-across+across-down,  $\updelta(\xsf_1)=0=\updelta(\xsf_2)$.  Further, it follows from \eqref{eqn:glueconvention} and \ref{eq!BD} that each pair of vertical maps glue to give a global map, so $\mathfrak{d}(\xsf_1,\xsf_2)=0$.  Hence $\dd \xsf=0$.

Similarly, write $\ysf= (\ysf_1,\ysf_2,0)\in\scrC_1$ where $\ysf_1$
is defined by multiplying all the vertical maps of $\xsf_1$ by $a$, and $\ysf_2$ by dividing all the vertical maps of $\xsf_2$ by $b$.
It follows, in a similar way, that $\dd(\ysf)=0$.

\begin{lemma}\label{lem:Ext1linInd}
With notation as above, $\xsf$ and $\ysf$ give linearly independent elements of $\mathrm{H}^1(\scrC)\cong\Ext^1(\scrO_\Curve,\scrO_\Curve)\cong\mathbb{K}^2$.
\end{lemma}

\begin{proof}
The first isomorphism holds by \ref{prop:Exti} and the second holds since $\scrN_{\Curve|\scrX}\cong\scrO(-3)\oplus\scrO(1)$.

Both $\xsf$ and $\ysf$ are closed, as observed above,
so suppose there exists $\upalpha = (\upalpha_1, \upalpha_2, \upalpha_{12})\in\scrC_0$ such that $\dd \upalpha = \uplambda\xsf + \upmu\ysf$ for some $\uplambda,\upmu\in\K$.
Restricting to $\scrU_1$,
we have $\updelta(\upalpha_1) = \uplambda\xsf_1 + \upmu\ysf_1$,
and by \eqref{eqn:updelta} and \ref{thm: locallyfreeres} the component $\scrE_1\rightarrow\scrE_0$ of this element is
\[
\updelta(\upalpha_1) = (v_1,v_2,a^{\kx+1}v_1 + v_2^{\ordP-1}A)M + c(v_1,v_2,a^{\kx+1}v_1 + v_2^{\ordP-1}A)
\]
for some $3\times 3$ matrix $M$ with entries in $\scrO_{\scrU_1}$, and some $c\in\scrO_{\scrU_1}$. In particular, all entries of $\updelta(\upalpha_1)$
lie in the ideal $\langle v_1,v_2\rangle\subset\scrO_{\scrU_1}$ and so vanish on $\Curve$.
In contrast, the second entry $-(\uplambda+\upmu a)$ of $\uplambda\xsf_1 + \upmu\ysf_1 = \uplambda(0,-1,0) + \upmu(0,-a,0)$ does not vanish on $\Curve$ unless $\uplambda=\upmu=0$.
\end{proof}

\subsection{Shapes and Degree One Homotopies}\label{subsec:shapes}

Alongside closed elements of $\scrC$ such as $\xsf$ and $\ysf$ of \S\ref{subsec:degone}
that appear in Theorem~\ref{thm: main intro},
it is convenient to define simpler elements of $\scrC$ that we call {\em shapes},
which are akin to elementary matrices.
As we will see repeatedly, these shapes will help distinguish between
identities that hold for trivial reasons --- matrix products vanishing because
all terms have a factor of zero --- and those that vanish because some nontrivial cancellation occurs.

With this in mind, and adopting the convention that degree one elements are denoted in lower case sans font, consider $(\gsf_{1}, \gsf_{2},0)\in\scrC_1$ defined by
\[
\begin{tikzpicture}
\node () at (-2.5,0) {$\gsf_{1} =$};
\node (B1) at (-0.25,1) {$\scrE_3$};
\node (C1) at (3.25,1) {$\scrE_2$};
\node (D1) at (6.25,1) {$\scrE_1$};
\node (E1) at (8,1) {$\scrE_0$};
\draw[->] (B1)--
node[above]{$\begin{psmallmatrix}v_2 \\ -a^{\kx+1}\\ -1\end{psmallmatrix}$}
(C1);
\draw[->] (C1)--
node[above]{$\begin{psmallmatrix}
a^{\kx+1}&v_2&0\\
v_2^{\ordP -2}\!A&-v_1&b^{2-\kx}w_1\\
-1&0&-v_2\end{psmallmatrix}$}
(D1);
\draw[->] (D1)--
(E1);
\node (B2) at (-1.75,-1) {$\scrE_3$};
\node (C2) at (-0.25,-1) {$\scrE_2$};
\node (D2) at (3.25,-1) {$\scrE_1$};
\node (E2) at (6.25,-1) {$\scrE_0$};
\draw[->] (B2)--
(C2);
\draw[->] (C2)--
node[above]{$\begin{psmallmatrix}
a^{\kx+1}&v_2&0\\
v_2^{\ordP -2}\!A&-v_1&b^{2-\kx}w_1\\
-1&0&-v_2\end{psmallmatrix}$}
(D2);
\draw[->] (D2)--
node[above]{$\begin{psmallmatrix}v_1 & v_2& b^{2-\kx}w_1\end{psmallmatrix}$}
(E2);
\draw[->](B1)--node[left,pos=0.4]{$\begin{psmallmatrix}0\\0\\0\end{psmallmatrix}$}(C2);
\draw[->](C1)--node[right,pos=0.4]{$\begin{psmallmatrix}0&0&0\\0&0&1\\0&0&0\end{psmallmatrix}$}(D2);
\draw[->](D1)--node[right,pos=0.4]{$\begin{psmallmatrix}0&0&1\end{psmallmatrix}$}(E2);
\end{tikzpicture}
\]
with $\gsf_2$ being given by the same three vertical matrices, but considered on the other chart.

Similarly, define $(\hsf_{1}, \hsf_{2},0)\in\scrC_1$ by
\[
\begin{tikzpicture}
\node () at (-2.5,0) {$\hsf_{1}=$};
\node (B1) at (-0.25,1) {$\scrE_3$};
\node (C1) at (3.25,1) {$\scrE_2$};
\node (D1) at (6.25,1) {$\scrE_1$};
\node (E1) at (8,1) {$\scrE_0$};
\draw[->] (B1)--
node[above]{$\begin{psmallmatrix}v_2 \\ -a^{\kx+1}\\ -1\end{psmallmatrix}$}
(C1);
\draw[->] (C1)--
node[above]{$\begin{psmallmatrix}
a^{\kx+1}&v_2&0\\
v_2^{\ordP -2}\!A&-v_1&b^{2-\kx}w_1\\
-1&0&-v_2\end{psmallmatrix}$}
(D1);
\draw[->] (D1)--
(E1);
\node (B2) at (-1.75,-1) {$\scrE_3$};
\node (C2) at (-0.25,-1) {$\scrE_2$};
\node (D2) at (3.25,-1) {$\scrE_1$};
\node (E2) at (6.25,-1) {$\scrE_0$};
\draw[->] (B2)--
(C2);
\draw[->] (C2)--
node[above]{$\begin{psmallmatrix}
a^{\kx+1}&v_2&0\\
v_2^{\ordP -2}\!A&-v_1&b^{2-\kx}w_1\\
-1&0&-v_2\end{psmallmatrix}$}
(D2);
\draw[->] (D2)--
node[above]{$\begin{psmallmatrix}v_1 & v_2& b^{2-\kx}w_1\end{psmallmatrix}$}
(E2);
\draw[->](B1)--node[left,pos=0.4]{$\begin{psmallmatrix}0\\0\\0\end{psmallmatrix}$}(C2);
\draw[->](C1)--node[right,pos=0.4]{$\begin{psmallmatrix}0&0&0\\0&1&0\\0&0&0\end{psmallmatrix}$}(D2);
\draw[->](D1)--node[right,pos=0.4]{$\begin{psmallmatrix}-1&0&0\end{psmallmatrix}$}(E2);
\end{tikzpicture}
\]
with $\hsf_2$ being given by the same three vertical matrices, but considered on the other chart.

\begin{notation}\label{not:definitionk}
For $\gsf$ and $\hsf$ defined above, define $\ksf_i\in\scrC_1$ by
\[
\ksf_i\colonequals
\begin{cases}
(a^i\cdot\gsf_1,\,\,b^{\kx-2-i}\cdot \gsf_2,0)&\mbox{if }0\leq i\leq \kx-2\\
(a^{i-(\kx+1)}\cdot\hsf_1,\,\,b^{(\kx+\ky-1)-i}\cdot\hsf_2,0)&\mbox{if }\kx+1\leq i\leq \kx+\ky-1
\end{cases}
\]
where terms such as $a^i\cdot \gsf_1$ should be interpreted as multiplying all the three vertical maps of $\gsf_1$ by the coefficient $a^i$.
\end{notation}

\begin{remark}\label{rem:shapes}
In \ref{not:definitionk}, and in similar definitions such as \ref{def:XY} below, we say that $\ksf_i$ is
{\em based on the shape} $\gsf$ (or $\hsf$), as a reminder that its three components are
polynomial multiples of the corresponding components of $\gsf$ (or $\hsf$).
This is not strictly essential, but it does help to navigate the array of constructions, and is summarised in \S\ref{rem:shapes2}.
\end{remark}

Lastly consider the shape $(\zsf_{1}, \zsf_{2},0)\in\scrC_1$, defined as
\[
\begin{tikzpicture}
\node () at (-2.5,0) {$\zsf_{1} =$};
\node (B1) at (-0.25,1) {$\scrE_3$};
\node (C1) at (3.25,1) {$\scrE_2$};
\node (D1) at (6.25,1) {$\scrE_1$};
\node (E1) at (8,1) {$\scrE_0$};
\draw[->] (B1)--
node[above]{$\begin{psmallmatrix}v_2 \\ -a^{\kx+1}\\ -1\end{psmallmatrix}$}
(C1);
\draw[->] (C1)--
node[above]{$\begin{psmallmatrix}
a^{\kx+1}&v_2&0\\
v_2^{\ordP -2}\!A&-v_1&b^{2-\kx}w_1\\
-1&0&-v_2\end{psmallmatrix}$}
(D1);
\draw[->] (D1)--
(E1);
\node (B2) at (-1.75,-1) {$\scrE_3$};
\node (C2) at (-0.25,-1) {$\scrE_2$};
\node (D2) at (3.25,-1) {$\scrE_1$};
\node (E2) at (6.25,-1) {$\scrE_0$};
\draw[->] (B2)--
(C2);
\draw[->] (C2)--
node[above]{$\begin{psmallmatrix}
a^{\kx+1}&v_2&0\\
v_2^{\ordP -2}\!A&-v_1&b^{2-\kx}w_1\\
-1&0&-v_2\end{psmallmatrix}$}
(D2);
\draw[->] (D2)--
node[above]{$\begin{psmallmatrix}v_1 & v_2& b^{2-\kx}w_1\end{psmallmatrix}$}
(E2);
\draw[->](B1)--node[left,pos=0.4]{$\begin{psmallmatrix}0\\0\\0\end{psmallmatrix}$}(C2);
\draw[->](C1)--node[right,pos=0.4]{$\begin{psmallmatrix}0&0&0\\-1&0&0\\0&0&0\end{psmallmatrix}$}(D2);
\draw[->](D1)--node[right,pos=0.4]{$\begin{psmallmatrix}0&0&0\end{psmallmatrix}$}(E2);
\end{tikzpicture}
\]
with $\zsf_2$ being given by the same three vertical matrices, but considered on the other chart.  This shape will appear again as a killing homotopy in \S\ref{sec:dd} and \S\ref{sec:inductivesetup}.

\subsection{Degree Two Shapes and Generators} \label{subsec:degtwo}
Adopting the convention that degree two elements will be denoted in upper case sans font, consider the shape
 $\Zsf=(\Zsf_1,\Zsf_2,0)\in\scrC_2$ defined as
\[
\begin{tikzpicture}
\node () at (0,0) {$\Zsf_1=$};
\node (C1) at (3.25,1) {$\scrE_3$};
\node (D1) at (6.25,1) {$\scrE_2$};
\node (E1) at (7.75,1) {$\scrE_1$};
\node (F1) at (9.25,1) {$\scrE_0$};
\draw[->] (E1)--(F1);
\draw[->] (C1)--
node[above]{$\begin{psmallmatrix}v_2 \\ -a^{\kx+1}\\ -1\end{psmallmatrix}$}
(D1);
\draw[->] (D1)--(E1);
\node (B2) at (0.25,-1) {$\scrE_3$};
\node (C2) at (1.75,-1) {$\scrE_2$};
\node (D2) at (3.25,-1) {$\scrE_1$};
\node (E2) at (6.25,-1) {$\scrE_0$};
\draw[->] (B2)--(C2);
\draw[->] (C2)--(D2);
\draw[->] (D2)--
node[above]{$\begin{psmallmatrix}v_1 & v_2& b^{2-\kx}w_1\end{psmallmatrix}$}(E2);
\draw[->](C1)--node[left]{$\begin{psmallmatrix}0\\-1\\0\end{psmallmatrix}$}(D2);
\draw[->](D1)--node[right]{$\begin{psmallmatrix}-1&0&0\end{psmallmatrix}$}(E2);
\end{tikzpicture}
\]
with $\Zsf_2$ being given by the same two vertical matrices, but considered on the other chart.

\begin{notation}\label{def:XY}
Set $\Xsf= (a^\kx\cdot \Zsf_{1}, \,\,b^{\ky-1}\cdot\Zsf_{2}, 0)\in\scrC_2$ and $\Ysf= (a^{\kx-1}\cdot \Zsf_{1},\,\, b^{\ky}\cdot\Zsf_{2}, 0)\in\scrC_2$.
\end{notation}
Thus both $\Xsf,\Ysf\in\scrC_2$ are based on the shape $\Zsf$.
Under the sign convention defining $\updelta$ in \eqref{eqn:updelta}, which for even degree is down-across$-$across-down, $\updelta(\Xsf_1)=0=\updelta(\Xsf_2)$.  Further, it follows from \eqref{eqn:glueconvention} and \eqref{eq!glue} that the pair of vertical maps $(\scrE_3\rightarrow\scrE_0)_{|\scrU_i}$ glue to give a global map, so $\mathfrak{d}(\Xsf_1,\Xsf_2)=0$.  Hence $\dd \Xsf=0$, and similarly $\dd\Ysf=0$.

These give generators of $\mathrm{H}^2(\scrC)$.

\begin{lemma}\label{lem:Ext2linInd}
$\Xsf$ and $\Ysf$ are linearly independent elements of $\mathrm{H}^2(\scrC)\cong\Ext^2(\scrO_\Curve,\scrO_\Curve)\cong\mathbb{K}^2$.
\end{lemma}
\begin{proof}
This is similar to \ref{lem:Ext1linInd}, with only linear independence to check.
Suppose there exists $\csf = (\csf_1, \csf_2, \csf_{12})\in\scrC_1$ such that $\dd \csf = \uplambda\Xsf + \upmu\Ysf$ for some $\uplambda,\upmu\in\K$.
Restricting to $\scrU_1$,
we have $\updelta(\csf_1) = \uplambda\Xsf_1 + \upmu\Ysf_1$,
and by \eqref{eqn:updelta} and \ref{thm: locallyfreeres} the component $\scrE_2\rightarrow\scrE_0$ of this element is
\[
\updelta(\csf_1) = (v_1,v_2,a^{\kx+1}v_1 + v_2^{\ordP-1}A)M + (e_1,e_2,e_3) \Mat(\dsf_2)
\]
for some $3\times 3$ matrix $M$ with entries in $\scrO_{\scrU_1}$, and some $e_i\in\scrO_{\scrU_1}$, where $\Mat(\dsf_2)$ is the matrix of $\dsf_2$
on $\scrU_1$ from \ref{thm: locallyfreeres}.
Considering only the first component of this, and equating to that of $\uplambda\Xsf_1 + \upmu\Ysf_1$, we have at once that
\begin{equation}\label{eq!c3modC}
a^{\kx+1} e_1 - e_3 \equiv -\uplambda a^\kx -\upmu a^{\kx-1} \mod I_\Curve = \left<v_1,v_2\right>.
\end{equation}
Consider now $\csf_{12}$, which is an element in $\bigoplus_{n\in\Z} \scrHom_{\scrU_{12}}(\scrE_n,\scrE_{n})$, and work in the coordinates on $\scrU_{12}$ induced from $\scrU_1$.  It is immediate that the components of
$\updelta\csf_{12}\colon\scrE_1\rightarrow\scrE_0$
all lie in $I_\Curve$, since the two summands of this map each factor through $\dsf_1$.
Thus the condition that $(\dd\csf)_{12}=0$ implies that
\begin{equation}\label{eq!vanishmodC}
{\csf_1}_{|\scrU_{12}} - {\csf_2}_{|\scrU_{12}} \equiv 0 \mod I_\Curve.
\end{equation}
Expressing the component $\csf_2\colon\scrE_1\rightarrow\scrE_0$ in coordinates by
$(e_1',e_2',e_3')$ for $e_i'\in\scrO_{\scrU_2}$, then modulo $I_\Curve$ each $e_i'$ is a
polynomial in $b\in\scrO_{\scrU_2}$. In particular, modulo $I_\Curve$ and working
in the coordinates of $\scrU_{12}\subset\scrU_{1}$ with $b=a^{-1}$,
\eqref{eq!vanishmodC} implies for $e_3'(b)$ that
\[
e_3(a) \equiv a^{\kx-2} e_3'(a^{-1})  \mod I_\Curve.
\]
The righthand side is a polynomial in $a$ of degree at most $\kx-2$,
so \eqref{eq!c3modC} implies that $\uplambda=\upmu=0$.
\end{proof}

Now consider the shape $(\Gsf_1,\Gsf_2,0)\in\scrC_2$, defined as
\[
\begin{tikzpicture}
\node () at (0,0) {$\Gsf_1=$};
\node (C1) at (3.25,1) {$\scrE_3$};
\node (D1) at (6.25,1) {$\scrE_2$};
\node (E1) at (7.75,1) {$\scrE_1$};
\node (F1) at (9.25,1) {$\scrE_0$};
\draw[->] (E1)--(F1);
\draw[->] (C1)--
node[above]{$\begin{psmallmatrix}v_2 \\ -a^{\kx+1}\\ -1\end{psmallmatrix}$}
(D1);
\draw[->] (D1)--(E1);
\node (B2) at (0.25,-1) {$\scrE_3$};
\node (C2) at (1.75,-1) {$\scrE_2$};
\node (D2) at (3.25,-1) {$\scrE_1$};
\node (E2) at (6.25,-1) {$\scrE_0$};
\draw[->] (B2)--(C2);
\draw[->] (C2)--(D2);
\draw[->] (D2)--
node[above]{$\scriptstyle\left(\begin{smallmatrix}v_1 & v_2& b^{2-\kx}w_1\end{smallmatrix}\right)$}(E2);
\draw[->](C1)--node[left]{$\begin{psmallmatrix}0\\0\\0\end{psmallmatrix}$}(D2);
\draw[->](D1)--node[right]{$\begin{psmallmatrix}0&0&-1\end{psmallmatrix}$}(E2);
\end{tikzpicture}
\]
with $\Gsf_2$ being given by the same two vertical matrices, but considered on the other chart.

Similarly, define $(\Hsf_1,\Hsf_2,0)$ where
\[
\begin{tikzpicture}
\node () at (0,0) {$\Hsf_1=$};
\node (C1) at (3.25,1) {$\scrE_3$};
\node (D1) at (6.25,1) {$\scrE_2$};
\node (E1) at (7.75,1) {$\scrE_1$};
\node (F1) at (9.25,1) {$\scrE_0$};
\draw[->] (E1)--(F1);
\draw[->] (C1)--
node[above]{$\begin{psmallmatrix}v_2 \\ -a^{\kx+1}\\ -1\end{psmallmatrix}$}
(D1);
\draw[->] (D1)--(E1);
\node (B2) at (0.25,-1) {$\scrE_3$};
\node (C2) at (1.75,-1) {$\scrE_2$};
\node (D2) at (3.25,-1) {$\scrE_1$};
\node (E2) at (6.25,-1) {$\scrE_0$};
\draw[->] (B2)--(C2);
\draw[->] (C2)--(D2);
\draw[->] (D2)--
node[above]{$\scriptstyle\left(\begin{smallmatrix}v_1 & v_2& b^{2-\kx}w_1\end{smallmatrix}\right)$}(E2);
\draw[->](C1)--node[left]{$\begin{psmallmatrix}0\\0\\0\end{psmallmatrix}$}(D2);
\draw[->](D1)--node[right]{$\begin{psmallmatrix}0&-1&0\end{psmallmatrix}$}(E2);
\end{tikzpicture}
\]
with $\Hsf_2$ being given by the same two vertical matrices, but considered on the other chart.  The following should be viewed as the degree two version of \ref{not:definitionk}, with slightly different superscripts, and slightly larger intervals.

\begin{notation}\label{def:Ki}
With $\Gsf$ and $\Hsf$ defined above, define $\Ksf_i\in\scrC_2$ by
\[
\Ksf_i\colonequals
\begin{cases}
(a^i\cdot\Gsf_1,\,\,b^{\kx-1-i}\cdot \Gsf_2,0)&\mbox{if }0\leq i\leq \kx-1\\
 (a^{i-(\kx+1)}\cdot\Hsf_1,\,\,b^{\kx+\ky-i}\cdot\Hsf_2, 0)&\mbox{if }\kx+1\le i\le \kx+\ky.
\end{cases}
\]
\end{notation}

\subsection{Degree Three Shapes and Generator}\label{subsec:degthree}
Adopting the convention that degree three elements will be denoted in lower case mathfrak, consider the shape $(\mathfrak{s}_1,\mathfrak{s}_2,0)\in\scrC_3$ defined as
\[
\begin{tikzpicture}
\node () at (0,0) {$\mathfrak{s}_1=$};
\node (C1) at (3.25,1) {$\scrE_3$};
\node (D1) at (6.25,1) {$\scrE_2$};
\node (E1) at (7.75,1) {$\scrE_1$};
\node (F1) at (9.25,1) {$\scrE_0$};
\draw[->] (E1)--(F1);
\draw[->] (C1)--
node[above]{$\begin{psmallmatrix}v_2 \\ -a^{\kx+1}\\ -1\end{psmallmatrix}$}
(D1);
\draw[->] (D1)--(E1);
\node (B2) at (-1.25,-1) {$\scrE_3$};
\node (C2) at (0.25,-1) {$\scrE_2$};
\node (D2) at (1.75,-1) {$\scrE_1$};
\node (E2) at (3.25,-1) {$\scrE_0$};
\draw[->] (B2)--(C2);
\draw[->] (C2)--(D2);
\draw[->] (D2)--(E2);
\draw[->](C1)--node[left]{$\scriptstyle -1$}(E2);
\end{tikzpicture}
\]
with $\mathfrak{s}_2$ being given by the same vertical matrix, but considered on the other chart.

\begin{notation}\label{def:s}
Set $\upxi= (a^{\kx}\cdot \mathfrak{s}_{1}, b^{\ky}\cdot\mathfrak{s}_{2},0)\in\scrC_{3}$.
\end{notation}
It is easy to verify that $\dd(\upxi)=0$, and it is clear that $\upxi$ is a basis element for $\mathrm{H}^3(\scrC)\cong\mathbb{K}$.


\subsection{Elementary Relationships Involving \texorpdfstring{$\dd$}{D}}\label{sec:dd}
This subsection establishes some elementary relationships between the degree one elements in \S\ref{subsec:degone},  the degree two elements in \S\ref{subsec:degtwo}, and the degree three elements in \S\ref{subsec:degthree} under the differential $\dd$.   

\begin{lemma}\label{lem:DztoZhgeneral}
For any $f\in\scrO_{\scrU_1}$ and $g\in\scrO_{\scrU_2}$,
\begin{align*}
\dd(f\cdot\gsf_1,\, g\cdot\gsf_2,\,0)&= (\phantom{a^{\kx+1}}f\cdot \Zsf_1,\,\, b^{\ky+1}g\cdot \Zsf_2,\,\,(f-a^{\kx-2}g)\cdot \gsf_1|_{\scrU_{12}})\\
\dd(f\cdot\hsf_1,\, g\cdot\hsf_2,\,0)&= (a^{\kx+1}f\cdot \Zsf_1,\,\, \phantom{b^{\ky-1}}g\cdot \Zsf_2,\,\,(f-a^{\ky-2}g)\cdot \hsf_1|_{\scrU_{12}})\\
\dd(f\cdot\zsf_1,\, g\cdot\zsf_2,\,0)&= (\,\,\phantom{a}v_2f\cdot \Zsf_1,\,\,\,\, \phantom{b}w_2g\cdot \Zsf_2,\,\,(f-a^{\kx+\ky-2}g)\cdot \zsf_1|_{\scrU_{12}}).
\end{align*}
\end{lemma}
\begin{proof}
(1) By definition, $\dd(f\cdot\gsf_1,\, g\cdot\gsf_2,\,0)=(\updelta(f\cdot \gsf_1), \updelta(g\cdot\gsf_2), f\cdot\gsf_1|_{12}-g\cdot\gsf_2|_{12})$.  We first claim that $\updelta(\gsf_1)=\Zsf_1$ and $\updelta(\gsf_2)=b^{\ky-1}\cdot\Zsf_2$.
Under the sign convention in \eqref{eqn:updelta}, which for degree one is down-across+across-down, the differential $\updelta(\gsf_1)$ equals
\[
\begin{tikzpicture}
\node (C1) at (4.75,1) {$\scrE_3$};
\node (D1) at (6.25,1) {$\scrE_2$};
\node (E1) at (7.75,1) {$\scrE_1$};
\node (F1) at (9.25,1) {$\scrE_0$};
\draw[->] (E1)--(F1);
\draw[->] (C1)--(D1);
\draw[->] (D1)--(E1);
\node (B2) at (1.75,-1) {$\scrE_3$};
\node (C2) at (3.25,-1) {$\scrE_2$};
\node (D2) at (4.75,-1) {$\scrE_1$};
\node (E2) at (6.25,-1) {$\scrE_0$};
\draw[->] (B2)--(C2);
\draw[->] (C2)--(D2);
\draw[->] (D2)--(E2);
\draw[->](C1)--node[left]{$\begin{psmallmatrix}0&0&0\\0&0&1\\0&0&0\end{psmallmatrix}\begin{psmallmatrix}v_2 \\ -a^{\kx+1}\\ -1\end{psmallmatrix}$}(D2);
\draw[->](D1)--node[right]{$\begin{psmallmatrix}v_1 & v_2& b^{2-\kx}w_1\end{psmallmatrix}\begin{psmallmatrix}0&0&0\\0&0&1\\0&0&0\end{psmallmatrix}+\begin{psmallmatrix}0&0&1\end{psmallmatrix}\begin{psmallmatrix}
a^{\kx+1}&v_2&0\\
v_2^{\ordP -2}\!A&-v_1&b^{2-\kx}w_1\\
-1&0&-v_2\end{psmallmatrix}
$}(E2);
\end{tikzpicture}
\]
which is $\Zsf_1$.  Similarly $\updelta(\gsf_2)$ equals
\[
\begin{tikzpicture}
\node (C1) at (4.75,1) {$\scrE_3$};
\node (D1) at (6.25,1) {$\scrE_2$};
\node (E1) at (7.75,1) {$\scrE_1$};
\node (F1) at (9.25,1) {$\scrE_0$};
\draw[->] (E1)--(F1);
\draw[->] (C1)--(D1);
\draw[->] (D1)--(E1);
\node (B2) at (1.75,-1) {$\scrE_3$};
\node (C2) at (3.25,-1) {$\scrE_2$};
\node (D2) at (4.75,-1) {$\scrE_1$};
\node (E2) at (6.25,-1) {$\scrE_0$};
\draw[->] (B2)--(C2);
\draw[->] (C2)--(D2);
\draw[->] (D2)--(E2);
\draw[->](C1)--node[left]{$\begin{psmallmatrix}0&0&0\\0&0&1\\0&0&0\end{psmallmatrix}\begin{psmallmatrix}w_2 \\ -1\\ -b^{\ky+1}\end{psmallmatrix}$}(D2);
\draw[->](D1)--node[right]{$\begin{psmallmatrix}a^{2-\ky}v_1 &w_2& w_1\end{psmallmatrix}\begin{psmallmatrix}0&0&0\\0&0&1\\0&0&0\end{psmallmatrix}+\begin{psmallmatrix}0&0&1\end{psmallmatrix}\begin{psmallmatrix}
1&w_2&0\\
w_2^{\ordP-2}\!B&-a^{2-\ky}v_1&w_1\\
-b^{\ky+1}&0&-w_2\end{psmallmatrix}$}(E2);
\end{tikzpicture}
\]
which is $b^{\ky+1}\cdot\Zsf_2$.  To calculate the final entry $f\cdot\gsf_1|_{12}-g\cdot\gsf_2|_{12}$, we work on the $\scrU_1$ coordinates on $\scrU_{12}$. The first non-zero map in $f\cdot\gsf_1|_{12}-g\cdot\gsf_2|_{12}$ is given by
\[
\begin{tikzpicture}
\node (A1) at (1,1) {$\scrE_2$};
\node (B1) at (1,-1) {$\scrE_1$};

\draw[->](A1)--node[left]{$\begin{psmallmatrix}0&0&0\\0&0&f\\0&0&0\end{psmallmatrix}$}(B1);

\node () at (2,0) {$-$};
\node (C1) at (3,1) {$\scrE_2$};
\node (D1) at (6.5,1) {$\scrE_2$};
\draw[->] (C1)--
node[above]{$\begin{psmallmatrix}a^{\kx+\ky-1}&0&0\\0&a^{\ky-1}&0\\0&0&a^{\kx-1}\end{psmallmatrix}$}
(D1);
\node (D2) at (3,-1) {$\scrE_1$};
\node (E2) at (6.5,-1) {$\scrE_1$};
\draw[<-] (D2)--
node[above]{$\scriptstyle\left(\begin{smallmatrix}a^{2-\ky}&0&0\\0&a^{-1}&0\\0&0&a^{2-\kx}\end{smallmatrix}\right)$}(E2);
\draw[->](D1)--node[right]{$\begin{psmallmatrix}0&0&0\\0&0&g\\0&0&0\end{psmallmatrix}$}(E2);

\node  at (8.5,0) {$=$};

\node (A5) at (12,1) {$\scrE_2$};
\node (B5) at (12,-1) {$\scrE_1$};

\draw[->](A5)--node[left]{${\scriptstyle (f-a^{\kx-2}g)}\begin{psmallmatrix}0&0&0\\0&0&1\\0&0&0\end{psmallmatrix}$}(B5);

\end{tikzpicture}
\]
and the second possible non-zero map is given by 
\[
\begin{tikzpicture}
\node (A1) at (1,1) {$\scrE_1$};
\node (B1) at (1,-1) {$\scrE_0$};

\draw[->](A1)--node[left]{$\begin{psmallmatrix}0&0&f\end{psmallmatrix}$}(B1);

\node () at (2,0) {$-$};
\node (C1) at (3,1) {$\scrE_1$};
\node (D1) at (6.5,1) {$\scrE_1$};
\draw[->] (C1)--
node[above]{$\begin{psmallmatrix}a^{\ky-2}&0&0\\0&a^{1}&0\\0&0&a^{\kx-2}\end{psmallmatrix}$}
(D1);
\node (D2) at (3,-1) {$\scrE_0$};
\node (E2) at (6.5,-1) {$\scrE_0$};
\draw[<-] (D2)--
node[above]{$\scriptstyle 1$}(E2);
\draw[->](D1)--node[right]{$\begin{psmallmatrix}0&0&g\end{psmallmatrix}$}(E2);

\node  at (8.5,0) {$=$};

\node (A5) at (12,1) {$\scrE_1$};
\node (B5) at (12,-1) {$\scrE_0$};

\draw[->](A5)--node[left]{${\scriptstyle (f-a^{\kx-2}g)}\begin{psmallmatrix}0&0&1\end{psmallmatrix}$}(B5);

\end{tikzpicture}
\]
and so $f\cdot\gsf_1|_{12}-g\cdot\gsf_2|_{12}=(f-a^{\kx-2}g)\cdot \gsf_1|_{\scrU_{12}}$ as claimed.\\
\noindent
(2) Again by direct calculation as above, $\updelta(\hsf_1)=a^{\kx+1}\cdot\Zsf_1$ and $\updelta(\hsf_2)=\Zsf_2$. The calculation that $f\cdot\hsf_1|_{12}-g\cdot\hsf_2|_{12}=(f-a^{\ky-2}g)\cdot \hsf_1|_{\scrU_{12}}$ is also similar.\\
\noindent
(3) Here, again by direct calculation, $\updelta(\zsf_1)=v_2\cdot\Zsf_1$ and $\updelta(\zsf_2)=w_2\cdot\Zsf_2$. 
Now the only possible non-zero map in $f\cdot\zsf_1|_{12}-g\cdot\zsf_2|_{12}$ is the map $\scrE_2\to\scrE_1$ given by
\[
\begin{tikzpicture}
\node (A1) at (1,1) {$\scrE_2$};
\node (B1) at (1,-1) {$\scrE_1$};

\draw[->](A1)--node[left]{$\begin{psmallmatrix}0&0&0\\-f&0&0\\0&0&0\end{psmallmatrix}$}(B1);

\node () at (2,0) {$-$};
\node (C1) at (3,1) {$\scrE_2$};
\node (D1) at (6.5,1) {$\scrE_2$};
\draw[->] (C1)--
node[above]{$\begin{psmallmatrix}a^{\kx+\ky-1}&0&0\\0&a^{\ky-1}&0\\0&0&a^{\kx-1}\end{psmallmatrix}$}
(D1);
\node (D2) at (3,-1) {$\scrE_1$};
\node (E2) at (6.5,-1) {$\scrE_1$};
\draw[<-] (D2)--
node[above]{$\scriptstyle\left(\begin{smallmatrix}a^{2-\ky}&0&0\\0&a^{-1}&0\\0&0&a^{2-\kx}\end{smallmatrix}\right)$}(E2);
\draw[->](D1)--node[right]{$\begin{psmallmatrix}0&0&0\\-g&0&0\\0&0&0\end{psmallmatrix}$}(E2);
\end{tikzpicture}
\]
which is clearly $(f-a^{\kx+\ky-2}g)\cdot \zsf_1|_{\scrU_{12}}$.
\end{proof}

The following asserts that, although $\ksf_i$ is defined on two different intervals, the effect of applying $\dd$ always looks the same.

\begin{cor}\label{lem:effectPsiPhi}
For all $i$ such that $\ksf_i$ is defined, 
\[
\dd(\ksf_i)=(a^i\cdot \Zsf_{1}, \,\,b^{\kx+\ky-1-i}\cdot\Zsf_{2}, 0).
\]
\end{cor}
\begin{proof}
In the first interval $\ksf_i=(a^i\cdot\gsf_1,b^{\kx-2-i}\cdot \gsf_2,0)$, so the statement follows from the top line in \ref{lem:DztoZhgeneral} applied to $f=a^i$ and $g=b^{\kx-2-i}$, where $f-a^{\kx-2}g=0$ since $a=b^{-1}$.  Similarly, in the second interval $\ksf_i=(a^{i-(\kx+1)}\cdot\hsf_1,\,\,b^{(\kx+\ky-1)-i}\cdot\hsf_2,0)$ so the statement follows from the second line in \ref{lem:DztoZhgeneral}.
\end{proof}
As a consequence of \ref{lem:effectPsiPhi}, every element in the sequence
\[
(a^0\cdot \Zsf_{1}, b^{\kx+\ky-1}\cdot\Zsf_{2}, 0),\hdots,
(a^{\kx-2}\cdot \Zsf_{1}, b^{\ky+1}\cdot\Zsf_{2}, 0),
(a^{\kx+1}\cdot \Zsf_{1}, b^{\ky-2}\cdot\Zsf_{2}, 0),\hdots,
(a^{\kx+\ky-1}\cdot \Zsf_{1}, b^0\cdot\Zsf_{2}, 0)
\]
belongs to the image of $\dd$.   In contrast, by \ref{lem:Ext2linInd} the two `missing' elements in this sequence, namely $\Ysf=(a^{\kx-1}\cdot \Zsf_{1}, b^{\ky}\cdot\Zsf_{2}, 0)$ and $\Xsf=(a^{\kx}\cdot \Zsf_{1}, b^{\ky-1}\cdot\Zsf_{2}, 0)$, do not lie in the image of $\dd$.  The following corollary formalises this fact.

\begin{cor}\label{cor:Aintosum}
Given $f(a)=\upalpha_0+\upalpha_1a+\hdots+\upalpha_{\kx+\ky-1} a^{\kx+\ky-1}\in\mathbb{K}[a]$, set $g=b^{\kx+\ky-1}f(b^{-1})$, namely $g=\upalpha_0b^{\kx+\ky-1}+\upalpha_1b^{\kx+\ky-2}+\hdots+\upalpha_{\kx+\ky-1}$.  Then in $\scrC$
\[
(f\cdot\Zsf_1, \,g\cdot\Zsf_2,0)=\sum_{i=0}^{\kx-2}\upalpha_i\dd(\ksf_i)+\upalpha_{\kx-1}\Ysf+\upalpha_{\kx}\Xsf+\sum_{i=\kx+1}^{\kx+\ky-1}\upalpha_i\dd(\ksf_i).
\]
\end{cor}
\begin{proof}
By construction of $g$, we may write
\[
(f\cdot\Zsf_1, \,g\cdot\Zsf_2,0)=\sum_{i=0}^{\kx+\ky-1} \upalpha_i(a^{i}\cdot\Zsf_1,\, b^{\kx+\ky-1-i}\cdot\Zsf_2,0).
\]
The result then follows directly from \ref{lem:effectPsiPhi}, and the definition of $\Xsf$ and $\Ysf$.
\end{proof}

The next elementary relationship is an analogue of \ref{lem:DztoZhgeneral} under $\dd\colon\scrC_2\rightarrow\scrC_3$.
\begin{lemma}\label{lem:effectGHZ}
For any $f\in\scrO_{\scrU_1}$ and $g\in\scrO_{\scrU_2}$, 
\begin{align*}
\dd(f\cdot\Gsf_1,\, g\cdot\Gsf_2,\,0)&= (\phantom{a^{\kx+1}}f\cdot  \mathfrak{s}_1,\,\, b^{\ky+1}g\cdot  \mathfrak{s}_2,\,\,(f-a^{\kx-1}g)\cdot  \Gsf_1|_{\scrU_{12}})\\
\dd(f\cdot\Hsf_1,\, g\cdot\Hsf_2,\,0)&= (a^{\kx+1}f\cdot  \mathfrak{s}_1,\,\, \phantom{b^{\ky-1}}g\cdot  \mathfrak{s}_2,\,\,(f-a^{\ky-1}g)\cdot  \Hsf_1|_{\scrU_{12}})
\end{align*}
\end{lemma}
\begin{proof}
By definition, $\dd(f\cdot\Gsf_1,\, g\cdot\Gsf_2,\,0)=(\updelta(f\cdot \Gsf_1), \updelta(g\cdot\Gsf_2), f\cdot\Gsf_1|_{12}-g\cdot\Gsf_2|_{12})$.  Under the sign convention in \eqref{eqn:updelta}, which for degree two is down-across minus across-down, it is clear that $\updelta(\Gsf_1)=\mathfrak{s}_1$ and $\updelta(\Gsf_2)=b^{\ky+1}\mathfrak{s}_2$.  Now the only possible non-zero map in $ f\cdot\Gsf_1|_{12}-g\cdot\Gsf_2|_{12}$ is the map $\scrE_2\to\scrE_0$ given by
\[
\begin{tikzpicture}
\node (A1) at (1,1) {$\scrE_2$};
\node (B1) at (1,-1) {$\scrE_0$};

\draw[->](A1)--node[left]{$\begin{psmallmatrix}0&0&-f\end{psmallmatrix}$}(B1);

\node () at (2,0) {$-$};
\node (C1) at (3,1) {$\scrE_2$};
\node (D1) at (6.5,1) {$\scrE_2$};
\draw[->] (C1)--
node[above]{$\begin{psmallmatrix}a^{\kx+\ky-1}&0&0\\0&a^{\ky-1}&0\\0&0&a^{\kx-1}\end{psmallmatrix}$}
(D1);
\node (D2) at (3,-1) {$\scrE_0$};
\node (E2) at (6.5,-1) {$\scrE_0$};
\draw[<-] (D2)--
node[above]{$\scriptstyle 1$}(E2);
\draw[->](D1)--node[right]{$\begin{psmallmatrix}0&0&-g\end{psmallmatrix}$}(E2);
\end{tikzpicture}
\]
which is clearly $(f-a^{\kx-1}g)\cdot \Gsf_1|_{\scrU_{12}}$. This proves the first statement. In a similar way, $\updelta(\Hsf_1)=a^{\kx+1}\mathfrak{s}_1$, $\updelta(\Hsf_2)=\mathfrak{s}_2$ and $ f\cdot\Hsf_1|_{12}-g\cdot\Hsf_2|_{12}=(f-a^{\ky-1}g)\cdot \Hsf_1|_{\scrU_{12}}$. 
\end{proof}

As in \ref{lem:effectPsiPhi}, the following asserts that although $\Ksf_i$ is defined on two different intervals, the effect of applying $\dd$ always looks the same.

\begin{cor}\label{lem:effectAlphaBeta}
For all $i$ such that $\Ksf_i$ is defined, 
\[
\dd(\Ksf_i)=(a^i\cdot \mathfrak{s}_{1}, \,\,b^{\kx+\ky-i}\cdot\mathfrak{s}_{2}, 0).
\]
\end{cor}
\begin{proof}
In the first interval $\Ksf_i=(a^i\cdot\Gsf_1,\,\,b^{\kx-1-i}\cdot \Gsf_2,0)$, so the statement follows from the top line in \ref{lem:effectGHZ} applied to $f=a^i$ and $g=b^{\kx-1-i}$, where $f-a^{\kx-1}g=0$ since $a=b^{-1}$.  Similarly, in the second interval, $\Ksf_i= (a^{i-(\kx+1)}\cdot\Hsf_1,\,\,b^{\kx+\ky-i}\cdot\Hsf_2, 0)$ so the statement follows from the second line in \ref{lem:effectGHZ}.
\end{proof}

\subsection{Elementary Relationships Involving \texorpdfstring{$\star$}{star}}\label{sec:star}
This subsection establishes some elementary relationships between the degree one elements in \S\ref{subsec:degone}, and the degree two elements in \S\ref{subsec:degtwo}, under the operation $\star$.

\begin{lemma}\label{lem:compdegone}
In $\scrC$, the following statements hold.
\begin{align*}
\xsf\star \xsf&=(\,\phantom{a^2}v_2^{\ordP-3}A\cdot\Zsf_1, \,\,b^2w_2^{\ordP-3}B\cdot\Zsf_2,\,0)\\
\xsf\star \ysf=\ysf\star \xsf&=(\,\phantom{{}^2}av_2^{\ordP-3}A\cdot\Zsf_1,\,\,\phantom{{}^2}bw_2^{\ordP-3}B\cdot\Zsf_2,\,0)\\
\ysf\star \ysf&=(\,a^2v_2^{\ordP-3}A\cdot\Zsf_1,\,\,\phantom{b^2}w_2^{\ordP-3}B\cdot\Zsf_2,\,0)
\end{align*}
\end{lemma}
\begin{proof}
Note $\xsf\star\xsf=(\xsf_1\circ\xsf_1,\xsf_2\circ\xsf_2,0)$. Now on $\scrU_1$, the composition $\xsf_1\circ\xsf_1$ is
\[
\begin{tikzpicture}
\node () at (0,0) {};
\node (C1) at (3.25,1) {$\scrE_3$};
\node (D1) at (6.25,1) {$\scrE_2$};
\node (E1) at (7.75,1) {$\scrE_1$};
\node (F1) at (9.25,1) {$\scrE_0$};
\draw[->] (E1)--(F1);
\draw[->] (C1)--
node[above]{$\begin{psmallmatrix}v_2 \\ -a^{\kx+1}\\ -1\end{psmallmatrix}$}
(D1);
\draw[->] (D1)--(E1);
\node (B2) at (0.25,-1) {$\scrE_3$};
\node (C2) at (1.75,-1) {$\scrE_2$};
\node (D2) at (3.25,-1) {$\scrE_1$};
\node (E2) at (6.25,-1) {$\scrE_0$};
\draw[->] (B2)--(C2);
\draw[->] (C2)--(D2);
\draw[->] (D2)--
node[above]{$\begin{psmallmatrix}v_1 & v_2& b^{2-\kx}w_1\end{psmallmatrix}$}(E2);
\draw[->](C1)--node[left]{${\scriptstyle v_2^{\ordP-3}A}\begin{psmallmatrix}0\\-1\\0\end{psmallmatrix}$}(D2);
\draw[->](D1)--node[right]{${\scriptstyle v_2^{\ordP-3}A}\begin{psmallmatrix}-1&0&0\end{psmallmatrix}$}(E2);
\end{tikzpicture}
\]
which equals $v_2^{\ordP-3}A\cdot\Zsf_1$.  On $\scrU_2$, $\xsf_2\circ\xsf_2$ it is similar, with $b^2w_2^{\ordP-3}B$ replacing the coefficient $v_2^{\ordP-3}A$ in the two vertical maps above.  The first statement follows.

The second and third statements follow in an identical fashion.
\end{proof}

\begin{lemma}\label{lem:gstarxyGH}
In $\scrC$, the following hold.
\[
\begin{aligned}
-\gsf_1\circ \xsf_1&=\Gsf_1 & -\gsf_1\circ \ysf_1&=a\cdot\Gsf_1 &-\gsf_2\circ \xsf_2&=b\cdot \Gsf_2 & -\gsf_2\circ \ysf_2&=\Gsf_2\\
\phantom{-}\xsf_1\circ\gsf_1&=\Gsf_1 & \phantom{-}\ysf_1\circ\gsf_1&=a\cdot\Gsf_1&\phantom{-}\xsf_2\circ\gsf_2&=b\cdot \Gsf_2 & \phantom{-}\ysf_2\circ\gsf_2&=\Gsf_2\\
-\hsf_1\circ \xsf_1&=\Hsf_1 &-\hsf_1\circ \ysf_1&=a\cdot\Hsf_1&-\hsf_2\circ \xsf_2&=b\cdot \Hsf_2 & -\hsf_2\circ \ysf_2&=\Hsf_2\\
\phantom{-}\xsf_1\circ\hsf_1&=\Hsf_1 & \phantom{-}\ysf_1\circ\hsf_1&=a\cdot\Hsf_1&\phantom{-}\xsf_2\circ\hsf_2&=b\cdot \Hsf_2 & \phantom{-}\ysf_2\circ\hsf_2&=\Hsf_2
\end{aligned}
\]
\end{lemma}
\begin{proof}
The first column holds by inspection.  The second column follows from the first since $\ysf_1=a\cdot\xsf_1$.  The fourth is $\scrU_2$ version of the first column, and the third column follows from the fourth since $\xsf_2=b\cdot\ysf_2$.
\end{proof}

In the following, note that whenever $\ksf_i$ is defined, both $\Ksf_i$ and $\Ksf_{i+1}$ are defined.
\begin{cor}\label{cor:xandkanticommute}
For all $i$ such that $\ksf_i$ is defined, in $\scrC$ there are equalities
\begin{alignat*}{5}
\xsf\star \ksf_i&=\Ksf_i &=-\ksf_i\star\xsf,\\
\ysf\star\ksf_i&=\Ksf_{i+1}&=-\ksf_i\star\ysf.
\end{alignat*}
\end{cor}
\begin{proof}
In the first interval, $\ksf_i=(a^i\cdot\gsf_1,\,b^{\kx-2-i}\cdot \gsf_2,0)$, so by \ref{lem:gstarxyGH}
\[
-\ksf_i\star\xsf=-(a^i\cdot\gsf_1\circ \xsf_1,\,b^{\kx-2-i}\cdot \gsf_2\circ \xsf_2,0)=(a^i\cdot\Gsf_1,\,b^{\kx-1-i}\cdot\Gsf_2,0)=\Ksf_i.
\]
Similarly, in the second interval, by  \ref{lem:gstarxyGH}
\[
-\ksf_i\star\xsf=-(a^{i-(\kx+1)}\cdot\hsf_1\circ \xsf_1,\,b^{(\kx+\ky-1)-i}\cdot\hsf_2\circ \xsf_2,0)=(a^{i-(\kx+1)}\cdot\Hsf_1,\,b^{\kx+\ky-i}\cdot\Hsf_2,0)=\Ksf_i.
\]
All other claims are similar.  
\end{proof}

The following deals with $\star$ between degree one and degree two inputs.
\begin{lemma}\label{lem:Deg2starx}
In $\scrC$, for all $i,j\geq 0$, there are equalities
\begin{alignat*}{3}
\xsf\star (a^i\cdot \Zsf_{1}, \,\,b^{j}\cdot\Zsf_{2}, 0) &= -(\,\,\,\,\,\,a^i\cdot \mathfrak{s}_{1}, \,\,b^{j+1}\cdot\mathfrak{s}_{2}, 0) = (a^i\cdot \Zsf_{1}, \,\,b^{j}\cdot\Zsf_{2}, 0)\star \xsf\phantom{.}\\
\ysf\star (a^i\cdot \Zsf_{1}, \,\,b^{j}\cdot\Zsf_{2}, 0) &= -(a^{i+1}\cdot \mathfrak{s}_{1}, \,\, \,\,\,\,\,\,b^{j}\cdot\mathfrak{s}_{2}, 0) = (a^i\cdot \Zsf_{1}, \,\,b^{j}\cdot\Zsf_{2}, 0)\star \ysf.
\end{alignat*}
\end{lemma}
\begin{proof}
The first statement follows from the fact that $\xsf_1\circ\Zsf_1=-\mathfrak{s}_1=\Zsf_1\circ\xsf_1$ and $\xsf_2\circ\Zsf_2=-b\cdot\mathfrak{s}_2=\Zsf_2\circ\xsf_2$, as can be directly verified.  The second statement follows similarly.
\end{proof}

\subsection{Notation for Induction}\label{sec:inductivesetup}
This subsection lays down some notation useful for the induction in \ref{thm:Umaintext}, and furthermore introduces some key homotopies. 

Recall  the polynomials $A_j$ and $B_j$ from \S\ref{subsec:AandB}.  
\begin{notation}\label{not:geq}
For $i\geq 3$, set $A_{\geq i}=\sum_{j\geq i}v_2^{j-\ordP}A_{j}$ and $B_{\geq i}=\sum_{j\geq i}w_2^{j-\ordP}B_{j}$.
\end{notation}
\begin{remark}\label{rem:Ageq3}
By \eqref{eqn:Aassum} and the definition above, $A=A_{\geq 3}$.
\end{remark}
Multiplying the definition by the appropriate power of $v_2$, it is clear that
\begin{equation}
\begin{aligned}
v_2^{\ordP-i}A_{\geq i} &= A_i + v_2^{\ordP-i}A_{\geq i+1}\\
w_2^{\ordP-i}B_{\geq i} &= B_i + w_2^{\ordP-i}B_{\geq i+1}.
\end{aligned}\label{eq:splitAgeq}
\end{equation}

\begin{lemma}\label{lem:AiBi2}
For all $i\geq 3$, $A_{\geq i}=a^{\kx+\ky-\ordP}B_{\geq i}$ and $a^{2-\kx-\ky}v_2^{\ordP-i}A_{\geq i}=b^{i-2}w_2^{\ordP-i}B_{\geq i}$.
\end{lemma}
\begin{proof}
The statement $A_{\geq i}=a^{\kx+\ky-\ordP}B_{\geq i}$ is immediate from the last statement in \ref{lem:AiBi1}. The second statement follows by the first, since $v_2=aw_2$ by the glue \eqref{eq!glue}.
\end{proof}

The following, which are based on the shape $\zsf$,  will be the key inductive homotopies.
\begin{notation}\label{def:eik}
For any $i\geq 3$, and for any $0\leq k\leq i-1$, set
\begin{equation}
\esf_{i,k} \colonequals
(a^k(v_2^{\ordP-i-1}A_{\geq i+1})\cdot\zsf_1,
\,\, b^{i-k-1}(w_2^{\ordP-i-1}B_{\geq i+1})\cdot\zsf_2,
\, 0)\in\scrC_1.\label{eqn:eik}
 \end{equation}
\end{notation}
By \eqref{eq:splitAgeq}, all entries in \eqref{eqn:eik} are polynomials.  

\begin{cor}\label{cor:inductwithe}
With notation as above,
\[
\dd( \esf_{i,k})=
(\,a^k(v_2^{\ordP-i}A_{\geq i+1})\cdot\Zsf_1,\,\, b^{i-k-1}(w_2^{\ordP-i}B_{\geq i+1})\cdot\Zsf_2,\,0).
\]
\end{cor}
\begin{proof}
Set $f=a^k(v_2^{\ordP-i-1}A_{\geq i+1})$ and $g=b^{i-k-1}(w_2^{\ordP-i-1}B_{\geq i+1})$, which are both polynomials by \eqref{eq:splitAgeq}. Now by the first statement in \ref{lem:AiBi2}, $A_{\geq i+1}=a^{\kx+\ky-\ordP}B_{\geq i+1}$, so 
\begin{align*}
a^kv_2^{\ordP-i-1}A_{\geq i+1}&=a^{k+\kx+\ky-\ordP}v_2^{\ordP-i-1}B_{\geq i+1}\tag{$A_{\geq i+1}=a^{\kx+\ky-\ordP}B_{\geq i+1}$}\\
&=a^{k+\kx+\ky-i-1}w_2^{\ordP-i-1}B_{\geq i+1}\tag{by glue \eqref{eq!glue}}\\
&=a^{\kx+\ky-2}b^{i-k-1}w_2^{\ordP-i-1}B_{\geq i+1}.\tag{$ab=1$}
\end{align*}
Thus $f=a^{\kx+\ky-2}g$, equivalently  $a^{2-\kx-\ky}f=g$.  The result is then a direct application of the general \ref{lem:DztoZhgeneral}, where the third entry is now zero since  $a^{2-\kx-\ky}f=g$.
%
\end{proof}

\begin{lemma}\label{lem:explusxe}
For any $i\geq 3$, and for any $0\leq k\leq i-1$, there are equalities
\begin{enumerate}
\item\label{lem:explusxe1} $\esf_{i,k}\star\xsf=\esf_{i,k-1}\star\ysf$ and $\xsf\star\esf_{i,k}=\ysf\star \esf_{i,k-1}$ provided that $k\geq 1$.
\item\label{lem:explusxe2} $ 
\xsf\star \esf_{i,k} +  \esf_{i,k}\star\xsf=-(\,a^kv_2^{\ordP-(i+1)}A_{\geq i+1}\cdot\Zsf_1, \,\,b^{i-k}w_2^{\ordP-(i+1)}B_{\geq i+1}\cdot\Zsf_2,\,0)$
\end{enumerate}
\end{lemma}
\begin{proof}
(1) Since $k\geq 1$, the element $\esf_{i,k-1}$ is defined.  The result then follows since
\begin{align*}
\xsf\star\esf_{i,k}&=(a^kv_2^{\ordP-i-1}A_{\geq i+1}\cdot\xsf_1\circ\zsf_1,
\,\, b^{i-k-1}w_2^{\ordP-(i+1)}B_{\geq i+1}\cdot\xsf_2\circ\zsf_2,
\, 0)\\
&=(a^{k-1}v_2^{\ordP-(i+1)}A_{\geq i+1}\cdot(a\cdot\xsf_1)\circ\zsf_1,
\,\, b^{i-k}w_2^{\ordP-(i+1)}B_{\geq i+1}\cdot (b^{-1}\cdot\xsf_2)\circ\zsf_2,
\, 0)\\
&=(a^{k-1}v_2^{\ordP-(i+1)}A_{\geq i+1}\cdot \ysf_1\circ\zsf_1,
\,\, b^{i-k}w_2^{\ordP-(i+1)}B_{\geq i+1}\cdot \ysf_2\circ\zsf_2,
\, 0)\\
&=\ysf\star\esf_{i,k-1}.
\end{align*}
The other claim is similar.\\
(2) By definition, $\xsf\star \esf_{i,k} +  \esf_{i,k}\star\xsf$ equals
\[
(a^kv_2^{\ordP-(i+1)}A_{\geq i+1}\cdot(\xsf_1\circ\zsf_1+\zsf_1\circ\xsf_1),
\,\, b^{i-k-1}w_2^{\ordP-(i+1)}B_{\geq i+1}\cdot(\xsf_2\circ\zsf_2+\zsf_2\circ\xsf_2),
\, 0).
\]
The statement follows since $\xsf_1\circ\zsf_1+\zsf_1\circ\xsf_1=-\Zsf_1$ and $\xsf_2\circ\zsf_2+\zsf_2\circ\xsf_2=-b\cdot\Zsf_2$, as can be directly verified. 
\end{proof}

The following products all vanish for very elementary reasons.
\begin{lemma}\label{lem:kstark}
Wherever $\Ksf_*$, $\ksf_*$ and $\esf_{*,*}$ are defined, the following products
\[
\begin{array}{llll}
\ksf_* \star\esf_{*,*} &\quad \esf_{*,*} \star \ksf_*&\quad\ksf_*\star\ksf_{*'} &\quad \esf_{*,*}\star\esf_{*',*'}\\
\ksf_*\star \Xsf &\quad  \Xsf\star\ksf_* &\quad  \ksf_*\star \Ysf &\quad \Ysf\star\ksf_* \\
\esf_{*,*}\star \Xsf &\quad  \Xsf\star\esf_{*,*}  &\quad   \esf_{*,*}\star \Ysf &\quad \Ysf\star\esf_{*,*} \\
\Ksf_*\star\ksf_* &\quad\ksf_*\star\Ksf_* &\quad \Ksf_*\star\esf_{*,*} &\quad \esf_{*,*}\star\Ksf_* \\
\Ksf_i\star \xsf &\quad \xsf\star\Ksf_i &\quad  \Ksf_i\star \ysf &\quad \ysf\star\Ksf_i 
\end{array}
\]
all equal $0_\scrC$.
\end{lemma}
\begin{proof}
Most follow at once from the shapes of these elements. The first row holds since all compositions involving $\zsf$, $\gsf$ and $\hsf$ are identically zero.  The second and third row holds since all compositions of $\Zsf$ with $\gsf$, $\hsf$ and $\zsf$ are zero. For the fourth line, $\Ksf_*$ is based on $\Gsf$ or $\Hsf$, so its only nonzero maps are $\scrE_2\rightarrow\scrE_0$, whilst $\ksf_*$ is based on $\gsf$ or $\hsf$, both of which have zero map $\scrE_3\rightarrow\scrE_2$.  Thus any composition of $\ksf_*$ with $\Ksf_*$ is zero. Similarly $\esf_{*,*}$ is based on $\zsf$, which again has zero map $\scrE_3\rightarrow\scrE_2$ and the vanishing follows.  The last line holds since all compositions of $\xsf$ and $\ysf$ with $\Gsf$ and $\Hsf$ are zero, as again can be explicitly observed.
\end{proof}

\subsection{Summary}\label{rem:shapes2}
As a summary of this section, the closed elements $\xsf,\ysf$ induce generators of $\mathrm{H}^1(\scrC)$, and likewise $\Xsf,\Ysf$ induce generators of $\mathrm{H}^2(\scrC)$, whilst $\mathfrak{s}$ generates $\mathrm{H}^3(\scrC)$.  These, together with their corresponding underlying shapes, are summarised as follows
\[
\begin{array}{r|ccccc}
\textrm{element} & \xsf & \ysf & \Xsf & \Ysf & \upxi \\
\hline
\textrm{shape} &  &  & \Zsf & \Zsf & \mathfrak{s} \\
\hline
\textrm{see} & \S\ref{subsec:degone} && \ref{def:XY} &\ref{def:XY}& \ref{def:s}
\end{array}
\]
where we note that $\xsf$ and $\ysf$ are not usefully defined by shapes.  In contrast, the elements
\[
\begin{array}{r|cccccc}
\textrm{element} & \ksf_i & \esf_{i,k}& \Ksf_i  & \dd\ksf_i & \dd\esf_{i,k} & \dd\Ksf_i  \\
\hline
\textrm{shape} & \gsf, \hsf & \zsf& \Gsf, \Hsf  & \Zsf & \Zsf & \mathfrak{s} \\
\hline
\textrm{see} & \ref{not:definitionk} & \ref{def:eik} & \ref{def:Ki} &\ref{lem:effectPsiPhi} & \ref{cor:inductwithe} & \ref{lem:effectAlphaBeta}
\end{array}
\]
will provide homotopies needed to control the  Kadeishvili algorithm in \S\ref{sec:Ainfty} below.

\section{The \texorpdfstring{$\Ainf$}{Ainfty} Minimal Model}\label{sec:Ainfty}

In this section, regardless of $\charac\K$, we describe the $\Ainf$-minimal model of the DG-algebra $\scrC$ constructed in \S\ref{sec:DGAC}, and in the process pin down all its higher products.

\subsection{Recap on Minimal Models}\label{subsec:MinModelRecap}
The main result will be proved by a very careful use of the $\Ainf$-minimal model due to Kadeishvili \cite{Kadeishvili}, which we state to set notation.  As in \cite{Kadeishvili}, all $\Ainf$-algebras considered in this paper are strictly unital, meaning there is an element $1\in\mathsf{A}_0$ which is the identity with respect to multiplication $\msf_2$, and further $\msf_n(\hdots,1,\hdots)=0$ for all $n\geq 3$.

Under Setup~\ref{setup:key}, consider the DG-algebra $\scrC$ in \S\ref{sec:DGAC}, and consider the finite dimensional vector space $\mathsf{A}\colonequals\mathrm{H}^*(\scrC)$.    Following \cite{Kadeishvili}, for a homogeneous element $\asf\in\mathsf{A_i}$ set $\hat{\asf}=(-1)^{|\asf|}\asf$ where $|\asf|=i$.

\begin{thm}[{\cite[Theorem 1]{Kadeishvili}}]\label{thm:kadeishvili}
There is an $\Ainf$-structure on $\mathsf{A}=\mathrm{H}^*(\scrC)$ given by 
\[
\msf_n\colon \otimes^n \mathsf{A}\rightarrow\mathsf{A}
\qquad\mbox{ for every }n\ge1
\]
with $\msf_1=0$, where on homogeneous inputs $\msf_2(\asf_1,\asf_2) = -\hat{\asf}_1\asf_2$, and furthermore an $\Ainf$-morphism
\[
\fsf_n\colon \otimes^n\mathsf{A}\rightarrow\scrC
\qquad\mbox{ for every }n\ge1
\]
for which $\fsf_1\colon\mathsf{A}\rightarrow\scrC$ is a quasi-isomorphism.
\end{thm}
The construction is recalled below.  For notational convenience, write $\Ainf$-maps
as e.g.\ $\msf_n(\asf_1,\dots,\asf_n)$ rather than $\msf_n(\asf_1\otimes\hdots\otimes\asf_n)$.

\begin{remark}\label{rem:mndeg3inputs}
By definition of an $\Ainf$-algebra, the $\Ainf$-structure maps $\msf_n$ are graded of degree $2-n$, namely
\[
\msf_n(\mathsf{A}_{i_1}\otimes\hdots\otimes\mathsf{A}_{i_n}) \subseteq \mathsf{A}_{i_1+\hdots+i_n+2-n}.
\]
Thus, for example, on degree one inputs $\asf_1,\hdots,\asf_n\in\mathsf{A}_1$,
$\msf_n(\asf_1,\hdots,\asf_n)\in\mathsf{A}_2$ for all $n\ge1$.
In the setup $\mathsf{A}=\mathsf{A}_0\oplus\hdots\oplus\mathsf{A}_3$ below,
it follows at once that on homogeneous inputs $\bsf_i\in\mathsf{A}_{d_i}$ with all $d_i\geq 1$, if either
\begin{enumerate}
\item some $d_i\ge3$ or
\item there are two distinct  $i_1$ and $i_2$ with  $d_{i_1}=d_{i_2}=2$
\end{enumerate}
then $\msf_n(\bsf_1,\hdots,\bsf_n)=0$.
\end{remark}

\begin{notation}
If $\asf\in\scrC$ with $\dd(\asf)=0$,  write $[\asf]$ for $\asf$ viewed in cohomology $\mathrm{H}^*(\scrC)=\mathsf{A}$.
To implement Kadeishvili's construction we must, once and for all, choose (closed) elements $\{b_i\}$ of $\scrC$ which descend to a basis $\{[b_i]\}$ of $\mathsf{A}$. This defines an injective map of vector spaces
\[
\upiota\colon\mathsf{A} \longrightarrow \scrC
\]
sending $[b_i]\mapsto b_i$, which is a quasi-isomorphism of complexes of vector spaces (where $\mathsf{A}$ has trivial boundary maps).
We make this choice now, using the generators of \ref{rem:shapes2}:
\[
\upiota\colon\qquad
[1_\scrC]\mapsto 1_\scrC,\quad
\begin{array}{l} 
[\xsf]\mapsto\xsf \\
{}[\ysf]\mapsto\ysf
\end{array},\quad
\begin{array}{l} 
[\Xsf]\mapsto\Xsf \\
{}[\Ysf]\mapsto\Ysf
\end{array},\quad
[\upxi]\mapsto\upxi,
\]
where recall $1_\scrC=(1,1,0)$. For inductive purposes, set $\fsf_1=\upiota$.
\end{notation}

The task is to construct all the higher products $\msf_n\colon \mathsf{A}^{\otimes n}\to \mathsf{A}$.
The construction works inductively, constructing $\msf_n$ and $\fsf_n$ in tandem, employing the following auxiliary sequence $\Usf_n$.
Removing the $k=0$, $j=n$ term from \cite[(2)]{Kadeishvili}, following \cite{Kadeishvili}\footnote{\cite{Kadeishvili} writes instead $+\sum(-1)^{k+1}$, which is equivalent, and also writes the second sum $\sum_{j=2}^{n-1}$, with the convention that terms are zero when they do not make sense.} for any $n\ge2$
define $\Usf_n\colon\mathsf{A}^{\otimes n}\to \scrC$ in terms of already defined $\fsf_{<n}$ and $\msf_{<n}$ by
\begin{equation}
\begin{aligned}
&\Usf_n(\asf_1,\hdots,\asf_n)\colonequals \sum_{\ell=1}^{n-1}(-1)^{|\fsf_\ell(\asf_1,\hdots,\asf_\ell)|+1}\,\fsf_\ell(\asf_1,\hdots,\asf_\ell)\star\fsf_{n-\ell}(\asf_{\ell+1},\hdots,\asf_{n})\\
&\qquad\qquad-\sum_{k=0}^{n-2}\sum_{j=2}^{\mathfrak{m}_{n,k}}(-1)^{k}\,\fsf_{n-j+1}(\hat{\asf}_1,\hdots,\hat{\asf}_k,\msf_j(\asf_{k+1},\hdots,\asf_{k+j}),\asf_{k+j+1},\hdots,\asf_n)
\end{aligned}\label{eqn:defineU}
\end{equation}
where $\mathfrak{m}_{n,k}=\mathrm{min}\{n-k,n-1\}$. The condition \cite[($2^\prime$)]{Kadeishvili} that the $\fsf_n$ determine an $\Ainf$-morphism is 
\begin{equation}
(\upiota \msf_n-\Usf_n)(\asf_1,\hdots,\asf_n)=\dd\fsf_n(\asf_1,\hdots,\asf_n).\label{eqn:mUf}
\end{equation}
As explained in more detail below, passing to cohomology by applying $[-]$, equation \eqref{eqn:mUf} defines $\msf_n$, which in turn defines $\upiota\msf_n$, which in turn (up to choice) determines $\fsf_n$.

\begin{notation}\label{notation:coho}
For closed $\csf_i\in\scrC$, to ease notation throughout we will write $\msf_n(\csf_{1},\hdots,\csf_{n})$ as shorthand for $\msf_n([\csf_{1}],\hdots,[\csf_n])$, and similarly for $\fsf_n$ and $\Usf_n$. In other words, since the inputs to $\msf_n$, $\fsf_n$ and $\Usf_n$ must be cohomology classes, our convention is that should inputs not be this, make them so.  This is a mild abuse of notation, but it is unambiguous.
\end{notation}

Concretely, the above translates into the following.  For  $\redbullet,\bluebullet\in\{\xsf,\ysf,\Xsf,\Ysf,\upxi\}$,
by the above abuse of notation
\begin{equation*}
\Usf_2(\redbullet,\bluebullet)\colonequals (-1)^{|\upiota(\redbullet)|+1}\upiota(\redbullet)\star\upiota(\bluebullet) = (-1)^{|\!\redbullet\!|+1}(\redbullet \star \bluebullet).\label{eqn:U2}
\end{equation*}
The right hand side is (closed and) determined by $\fsf_1$, and thus so are both $\Usf_2(\redbullet,\bluebullet)$ and $[\Usf_2(\redbullet,\bluebullet)]$.  
Using \eqref{eqn:mUf} defines $\msf_2(\redbullet,\bluebullet)$, as expressing this known $[\Usf_2(\redbullet,\bluebullet)]$ in terms of the chosen basis $[b_i]$, there exists scalars $\upalpha_i$ such that
\[
\msf_2(\redbullet,\bluebullet)\stackrel{\scriptstyle\eqref{eqn:mUf}}{=}[\Usf_2(\redbullet,\bluebullet)]=\sum\upalpha_i [b_i].
\]
Applying $\upiota$ it follows that $\upiota\msf_2(\redbullet,\bluebullet)=\sum\upalpha_i b_i$, and so on the chain level \eqref{eqn:mUf} is simply
\begin{equation}
\sum\upalpha_i b_i-(-1)^{|\!\redbullet\!|+1}(\redbullet \star \bluebullet)=\dd\fsf_2(\redbullet,\bluebullet).\label{eqn:definef2}
\end{equation}
The left hand side is all determined, so determines some choice $\fsf_2(\redbullet,\bluebullet)\colon\mathsf{A}^{\otimes 2}\to\scrC$.  It should be noted that $\fsf_2$ is not unique, but some choice of $\fsf_2(\redbullet,\bluebullet)$ satisfying \eqref{eqn:definef2} can always be made.  In the course of the proof below we will use this freedom explicitly, see  \ref{lem:Uis0} and (0) in the proof of \ref{thm:Umaintext}.

We next move to determine $\msf_3$ and $\fsf_3$.  Now by definition
\begin{equation}
\begin{aligned}
\Usf_3(\redbullet,\bluebullet,\greenbullet) \colonequals&\ \phantom{+}(-1)^{|\!\redbullet\!|+1} \redbullet\star\, \fsf_2(\bluebullet,\greenbullet) + (-1)^{|\!\redbullet\!|+|\!\bluebullet\!|} \fsf_2(\redbullet,\bluebullet)\star \greenbullet\\
&+ (-1)^{|\!\redbullet\!|}\fsf_2(\redbullet,\msf_2(\bluebullet,\greenbullet)) - \fsf_2(\msf_2(\redbullet,\bluebullet),\greenbullet),
\end{aligned}\label{eqn:U3}
\end{equation}
which is fully determined by the previous data of $\msf_2$ and $\fsf_2$. Expressing this known cohomology class $[\Usf_3(\redbullet,\bluebullet,\greenbullet) ]$ in terms of the basis $[b_i]$, there exists scalars $\upbeta_i$ such that
\[
\msf_3(\redbullet,\bluebullet,\greenbullet) \stackrel{\scriptstyle \eqref{eqn:mUf}}{=}[\Usf_3(\redbullet,\bluebullet,\greenbullet)]=\sum\upbeta_i [b_i].
\]
Applying $\upiota$ it follows that $\upiota\msf_3(\redbullet,\bluebullet,\greenbullet)=\sum\upbeta_i b_i$, and so on the chain level \eqref{eqn:mUf} is simply
\begin{equation}
\sum\upbeta_i b_i-\Usf_3(\redbullet,\bluebullet,\greenbullet)=\dd\fsf_3(\redbullet,\bluebullet,\greenbullet).\label{eqn:definef3}
\end{equation}
Again, since the left hand side is all determined, this gives some non-unique choice of $\fsf_3(\redbullet,\bluebullet,\greenbullet)\colon\mathsf{A}^{\otimes 3}\to\scrC$.  

The calculation continues in this manner.  The chain-level $\Usf_n$ is determined by smaller $\msf_i$ and $\fsf_i$.  This then determines $[\Usf_n]$, which can be written in terms of the basis $[b_i]$.  In turn, applying $[-]$ to \eqref{eqn:mUf}, the expression for $[\Usf_n]$ determines $\msf_n$ and thus $\upiota\msf_n$, which in turn by \eqref{eqn:mUf} gives some non-unique choice of $\fsf_n$.

\subsection{Degree One \texorpdfstring{$\msf_2$}{m2} Inputs}
This subsection calculates $\msf_2$ on all degree one inputs, namely combinations of $\xsf$ and $\ysf$.  The point is that this is described by a very particular linear combination of $\Xsf$ and $\Ysf$, up to some homotopies of the form $\dd(\ksf_i)$, plus in all cases some higher order term of the form $\dd(\esf)$.  This latter term will, in later subsections, allow all higher products $\msf_n$ to be computed inductively.

\begin{prop}\label{prop:m2chaindegone}
In $\scrC$, the following statements hold
\begin{align*}
\xsf\star\xsf&=
\uplambda_{30}\Xsf+\uplambda_{21}\Ysf +\dd(\uplambda_{12}\ksf_{\kx-2}+\uplambda_{03}\ksf_{\kx-3}+\esf_{3,0})\\
\xsf\star\ysf=\ysf\star\xsf&=
\uplambda_{21}\Xsf+\uplambda_{12}\Ysf+\dd(\uplambda_{30}\ksf_{\kx+1}+ \uplambda_{03}\ksf_{\kx-2}+\esf_{3,1})\\
\ysf\star\ysf&=
\uplambda_{12}\Xsf+\uplambda_{03}\Ysf +\dd(\uplambda_{30}\ksf_{\kx+2}+\uplambda_{21}\ksf_{\kx+1}+\esf_{3,2}).
\end{align*}
\end{prop}
\begin{proof}
By \ref{rem:Ageq3}, $A=A_{\geq 3}$.  Applying \eqref{eq:splitAgeq} to $i=3$ gives $v_2^{\ordP-3}A_{\geq 3} = A_3 + v_2^{\ordP-3}A_{\geq 4}$, and so combining gives $v_2^{\ordP-3}A=A_3+v_2^{\ordP-3}A_{\geq 4}$.  Similarly $w_2^{\ordP-3}B=B_3+w_2^{\ordP-3}B_{\geq 4}$.

Thus, decomposing \ref{lem:compdegone} we may write
\begin{align*}
\xsf\star\xsf&=
\big(\,
\phantom{a^2}A_3\cdot\Zsf_1,
\,\,
b^2B_{3}\cdot\Zsf_2,
\,0 
\big)
+
\big(\, 
\phantom{a^2}v_2^{\ordP-3}A_{\geq 4}\cdot\Zsf_1,
\,\,
b^2w_2^{\ordP-3}B_{\geq 4}\cdot\Zsf_2,
\,0 
\big)\\
\xsf\star\ysf=\ysf\star\xsf&=
\big(\,
\phantom{{}^2}aA_3\cdot\Zsf_1,
\,\,
\phantom{{}^2}bB_{3}\cdot\Zsf_2,
\,0 
\big)
+
\big(\, 
\phantom{{}^2}av_2^{\ordP-3}A_{\geq 4}\cdot\Zsf_1,
\,\,
\phantom{{}^2}bw_2^{\ordP-3}B_{\geq 4}\cdot\Zsf_2,
\,0 
\big)\\
\ysf\star\ysf&=
\big(\,
a^2A_3\cdot\Zsf_1,
\,\,
\phantom{b^2}B_{3}\cdot\Zsf_2,
\,0 
\big)
+
\big(\, 
a^2v_2^{\ordP-3}A_{\geq 4}\cdot\Zsf_1,
\,\,
\phantom{b^2}w_2^{\ordP-3}B_{\geq 4}\cdot\Zsf_2,
\,0 
\big).
\end{align*}
By \ref{cor:inductwithe}, the final terms can be rewritten, to give 
\begin{align*}
\xsf\star\xsf&=
\big(\,
\phantom{a^2}A_3\cdot\Zsf_1,
\,\,
b^2B_{3}\cdot\Zsf_2,
\,0 
\big)
+
\dd(\esf_{3,0})\\
\xsf\star\ysf=\ysf\star\xsf&=
\big(\,
\phantom{{}^2}aA_3\cdot\Zsf_1,
\,\,
\phantom{{}^2}bB_{3}\cdot\Zsf_2,
\,0 
\big)
+
\dd(\esf_{3,1})\\
\ysf\star\ysf&=
\big(\,
a^2A_3\cdot\Zsf_1,
\,\,
\phantom{b^2}B_{3}\cdot\Zsf_2,
\,0 
\big)
+
\dd(\esf_{3,2}).
\end{align*}
Now writing \ref{def:AandB} backwards
\begin{alignat*}{5}
A_3&=\uplambda_{03}a^{\kx-3}&&+\uplambda_{12}a^{\kx-2}&&+\uplambda_{21}a^{\kx-1}&&+\uplambda_{30}a^{\kx}\\
aA_3&=\uplambda_{03}a^{\kx-2}&&+\uplambda_{12}a^{\kx-1}&&+\uplambda_{21}a^{\kx}&&+\uplambda_{30}a^{\kx+1}\\
a^2A_3&=\uplambda_{03}a^{\kx-1}&&+\uplambda_{12}a^{\kx}&&+\uplambda_{21}a^{\kx+1}&&+\uplambda_{30}a^{\kx+2}.
\end{alignat*}
For each of these three polynomials in $a$, consider the corresponding polynomial $g\in\mathbb{K}[b]$ defined in \ref{cor:Aintosum}.  By definition \ref{def:AandB}, for the top polynomial this is $b^2B_3$, for the middle this is $bB_3$, and for the bottom this is $B_3$.  Hence, by \ref{cor:Aintosum} we may write 
\begin{alignat*}{5}
\big(\,\phantom{a^2}A_3\cdot\Zsf_1,
\,\,b^2B_{3}\cdot\Zsf_2,
\,0 \big)
&=\uplambda_{03}\dd(\ksf_{\kx-3})&&+\uplambda_{12}\dd(\ksf_{\kx-2})&&+\uplambda_{21}\Ysf&&+\uplambda_{30}\Xsf\\
\big(\,\phantom{{}^2}aA_3\cdot\Zsf_1,
\,\,\phantom{{}^2}bB_{3}\cdot\Zsf_2,\,0 \big)
&=\uplambda_{03}\dd(\ksf_{\kx-2})&&+\uplambda_{12}\Ysf&&+\uplambda_{21}\Xsf&&+\uplambda_{30}\dd(\ksf_{\kx+1})\\
\big(\,a^2A_3\cdot\Zsf_1,
\,\,\phantom{b^2}B_{3}\cdot\Zsf_2,\,0 \big)
&=\uplambda_{03}\Ysf&&+\uplambda_{12}\Xsf&&+\uplambda_{21}\dd(\ksf_{\kx+1})&&+\uplambda_{30}\dd(\ksf_{\kx+2}),
\end{alignat*}
and the statement follows. 
\end{proof}

In particular, in the notation of \S\ref{subsec:MinModelRecap}, since on degree one inputs $\Usf_2(\redbullet,\bluebullet)=\redbullet\star\bluebullet$, applying $[-]$ then $\upiota$ to \ref{prop:m2chaindegone} shows that 
\begin{equation}
\begin{aligned}
\upiota\msf_2(\xsf,\xsf)&=
\uplambda_{30}\Xsf+\uplambda_{21}\Ysf \\
\upiota\msf_2(\xsf,\ysf)=\upiota\msf_2(\ysf,\xsf)&=
\uplambda_{21}\Xsf+\uplambda_{12}\Ysf\\
\upiota\msf_2(\ysf,\ysf)&=
\uplambda_{12}\Xsf+\uplambda_{03}\Ysf.
\end{aligned}\label{eqn:iotam2}
\end{equation}
Substituting \eqref{eqn:iotam2} and \ref{prop:m2chaindegone} directly into the left side of \eqref{eqn:definef2}, it follows that we may choose
\begin{equation}
\begin{aligned}
\fsf_2(\xsf,\xsf)&=-(\uplambda_{12}\ksf_{\kx-2}+\uplambda_{03}\ksf_{\kx-3}+\esf_{3,0})\\
\fsf_2(\xsf,\ysf)=\fsf_2(\ysf,\xsf)&=-(\uplambda_{30}\ksf_{\kx+1}+ \uplambda_{03}\ksf_{\kx-2}+\esf_{3,1})\\
\fsf_2(\ysf,\ysf)&=-(\uplambda_{30}\ksf_{\kx+2}+\uplambda_{21}\ksf_{\kx+1}+\esf_{3,2}).
\end{aligned}\label{f2degone}
\end{equation}

\subsection{All \texorpdfstring{$\msf_2$}{m2} Products}
By Remark \ref{rem:mndeg3inputs}, when all inputs have degree one or higher, 
if one of the inputs for $\msf_2$ has degree three, or if both have degree two, then the product 
is necessarily zero.
Thus the only remaining $\msf_2$ to consider are those having
one input of degree one, and one input of degree two. 
\begin{lemma}\label{lem:m2chaindegall} 
In $\scrC$, the following statements hold.
\begin{align*}
\xsf\star\Xsf&=-\upxi= \Xsf\star\xsf & \xsf\star\Ysf&=-\dd(\Ksf_{\kx-1})=\Ysf\star\xsf\\
\ysf\star\Ysf&=-\upxi= \Ysf\star\ysf & \ysf\star\Xsf&=-\dd(\Ksf_{\kx+1})=\Xsf\star\ysf
\end{align*}
\end{lemma}
\begin{proof}
Since $\Ysf=(a^{\kx-1}\cdot \Zsf_{1}, b^{\ky}\cdot\Zsf_{2}, 0)$ and $\Xsf=(a^{\kx}\cdot \Zsf_{1}, b^{\ky-1}\cdot\Zsf_{2}, 0)$, it follows directly from \ref{lem:Deg2starx} that
\begin{align*}
\xsf\star\Xsf&=-(a^{\kx}\cdot \mathfrak{s}_{1}, \,\,b^{\ky}\cdot\mathfrak{s}_{2}, 0)= \Xsf\star\xsf & \xsf\star\Ysf&=-(a^{\kx-1}\cdot \mathfrak{s}_{1}, \,\,b^{\ky+1}\cdot\mathfrak{s}_{2}, 0) =\Ysf\star\xsf\\
\ysf\star\Ysf&=-(a^{\kx}\cdot \mathfrak{s}_{1}, \,\,b^{\ky}\cdot\mathfrak{s}_{2}, 0)= \Ysf\star\ysf & \ysf\star\Xsf&=-(a^{\kx+1}\cdot \mathfrak{s}_{1}, \,\,b^{\ky-1}\cdot\mathfrak{s}_{2}, 0)=\Xsf\star\ysf.
\end{align*}
Now the statement follows since $\upxi= (a^{\kx}\cdot \mathfrak{s}_{1}, b^{\ky}\cdot\mathfrak{s}_{2},0)$, and  \ref{lem:effectAlphaBeta}.
\end{proof}
Now $\msf_2(\redbullet,\bluebullet)=(-1)^{|\!\redbullet\!|+1}[\redbullet \star \bluebullet]$, so applying $[-]$ then $\upiota$ to \ref{lem:m2chaindegall}  gives
\begin{equation}
\begin{aligned}
\upiota\msf_2(\xsf,\Xsf) &= -\upxi  &\qquad&& \upiota\msf_2(\xsf,\Ysf) &=0_\scrC \\
\upiota\msf_2(\Xsf,\xsf) &= \phantom{-}\upxi  &\qquad&&  \upiota\msf_2(\Ysf,\xsf) &= 0_\scrC  \\
\upiota\msf_2(\ysf,\Ysf) &= -\upxi  &\qquad&&\upiota\msf_2(\ysf,\Xsf) &= 0_\scrC  \\
\upiota\msf_2(\Ysf,\ysf) &= \phantom{-}\upxi  &\qquad&&\upiota\msf_2(\Xsf,\ysf) &= 0_\scrC. 
\end{aligned}\label{eqn:iotam2degall}
\end{equation}
Substituting \eqref{eqn:iotam2degall} and \ref{lem:m2chaindegall} directly into the left hand side of \eqref{eqn:definef2}, it follows that
\begin{equation*}
\begin{aligned}
0_\scrC &= \dd\fsf_2(\xsf,\Xsf)  &\quad&& \phantom{-}\dd(\Ksf_{r-1}) &=\dd\fsf_2(\xsf,\Ysf) \\
0_\scrC &=\dd\fsf_2(\Xsf,\xsf)   &\quad&&  -\dd(\Ksf_{r-1}) &=\dd\fsf_2(\Ysf,\xsf)  \\
0_\scrC &=\dd\fsf_2(\ysf,\Ysf)   &\quad&&  \phantom{-}\dd(\Ksf_{\kx+1}) &=\dd\fsf_2(\ysf,\Xsf) \\
0_\scrC &=\dd\fsf_2(\Ysf,\ysf)   &\quad&&  -\dd(\Ksf_{\kx+1}) &=\dd\fsf_2(\Xsf,\ysf),
\end{aligned}
\end{equation*}
thus we may choose 
\begin{equation}
\begin{aligned}
\fsf_2(\xsf,\Xsf) &= 0_\scrC &  \fsf_2(\ysf,\Ysf) &= 0_\scrC&   \fsf_2(\xsf,\Ysf) &= \Ksf_{\kx-1}  &  \fsf_2(\ysf,\Xsf) &=\Ksf_{\kx+1} \\
\fsf_2(\Xsf,\xsf) &= 0_\scrC &  \fsf_2(\Ysf,\ysf) &= 0_\scrC&  \fsf_2(\Ysf,\xsf) &= -\Ksf_{\kx-1}    &  \fsf_2(\Xsf,\ysf) &= -\Ksf_{\kx+1}.
\end{aligned}\label{eqn:f2degall}
\end{equation}

\subsection{Simple Maps and Consequences of Degree}\label{sec:fn0bydegree}
As the collection of $\asf_1\otimes\hdots\otimes\asf_n$ with each $\asf_i\in\{\xsf,\ysf,\Xsf,\Ysf,\upxi\}$
is linearly independent in the tensor algebra $\mathrm{T}_\mathbb{K}(\mathsf{A})$, the map $\fsf_n$ may be defined by making
a choice for each $\fsf_n(\asf_1,\hdots,\asf_n)$ independently, compatible with \eqref{eqn:mUf}.  The following is clear, and is well-known.

\begin{lemma}\label{lem:Uis0}
Suppose that $\fsf_i$ is defined for $1\le i\le n-1$, and further that the defining equation
\eqref{eqn:defineU} gives $\Usf_n(\asf_1,\hdots,\asf_n)=0_\scrC$ where $\asf_1,\hdots,\asf_n\in\{\xsf,\ysf,\Xsf,\Ysf,\upxi\}$.
Then $\msf_n(\asf_1,\hdots,\asf_n)=0_\mathsf{A}$ and we may choose the map $\fsf_n$ so that
$\fsf_n(\asf_1,\hdots,\asf_n)=0_\scrC$.
\end{lemma}
\begin{proof}
Since $\upiota$ is injective, \eqref{eqn:mUf} is satisfied only if
$\msf_n(\asf_1,\hdots,\asf_n) = [\Usf_n(\asf_1,\hdots,\asf_n)] = 0_\mathsf{A}$.
Since $\dd(0_\scrC) = 0_\scrC = (\upiota\msf_n-\Usf_n)(\asf_1,\hdots,\asf_n)$,
the choice $\fsf_n(\asf_1,\hdots,\asf_n)\colonequals 0_\scrC$ is compatible with \eqref{eqn:mUf}.
\end{proof}

\begin{convention}\label{convention}
If $\Usf_n(\asf_1,\hdots,\asf_n)=0_\scrC$, we choose $\fsf_n$ so that
$\fsf_n(\asf_1,\hdots,\asf_n)=0_\scrC$.
\end{convention}

Usually controlling $\Usf_n$ and $\fsf_n$ is hard, but properties of both will be significantly simplified in our setting by the following.

\begin{defin}
We say that a collection of maps $\left\{\fsf_\ell\colon\mathsf{A}^{\otimes \ell}\rightarrow\scrC\right\}_{\ell=1}^k$ is
{\em simple} if
for all $1\le\ell\le k$, and for all inputs $\asf_i\in\{\xsf,\ysf,\Xsf,\Ysf,\upxi\}$,
\[
\fsf_\ell(\asf_1,\hdots,\asf_\ell) = (\Ssf,\Tsf,0)\in\scrC_{d_1+\cdots+d_\ell + 1 - \ell}
\]
for some chains $\Ssf,\Tsf$ that depend on $\ell$ and the $\asf_i$, where $d_i=|\asf_i|$.
\end{defin}
It will turn out, as part of our inductive process, that we may choose the $\fsf_n$ to be simple to any degree. As such, later we will require the following two results.
\begin{lemma}\label{lem:fstarf}
Fix $n\ge2$. Suppose that the set of maps $\left\{\fsf_1,\hdots,\fsf_{n-1}\right\}$ is simple
and consider inputs $\asf_i\in\{\xsf,\ysf,\Xsf,\Ysf,\upxi\}$ for $1\le i\le n$ satisfying $\sum (|\asf_i|-1) \ge 2$.  Then the following statements hold.
\begin{enumerate}
\item\label{lem:fstarf1}
For all $\ell$ such that $1\le\ell\le n-1$,
\[
\fsf_\ell(\asf_1,\hdots,\asf_\ell)\star\fsf_{n-\ell}(\asf_{\ell+1},\hdots,\asf_{n}) = 0_\scrC.
\]
\item\label{cor:U0f0}
$\Usf_n(\asf_{1},\hdots,\asf_{n}) = 0_\scrC$.
\end{enumerate}
\end{lemma}
\begin{proof}
(1) Set $d_i=|\asf_i|$.  Each map $\fsf_n$ has degree $1-n$, so that
\[
\fsf_\ell(\asf_1,\hdots,\asf_\ell) \in \scrC_{d_1+\cdots+d_\ell + 1 - \ell}
\quad\textrm{and}\quad
\fsf_{n-\ell}(\asf_{\ell+1},\hdots,\asf_{n}) \in \scrC_{d_{\ell+1}+\cdots+d_n + 1 - (n-\ell)}.
\]
By assumption, both $\fsf_\ell$ and $\fsf_{n-\ell}$ are simple, and so multiplying these two chains together by $\star$ still gives an element of the form $(\Ssf,\Tsf,0)\in\scrC_d$,
where 
\[
d\colonequals (d_1+\hdots+d_\ell + 1 - \ell) + (d_{\ell+1}+\cdots+d_n + 1 - (n-\ell)) = 2 - n + \sum d_i.
\]
If $d\ge4$, then trivially $\Ssf=\Tsf=0$ since by definition $\scrC_4=0\oplus 0\oplus \vC^1(\scrU,\scrHom^3(\scrE,\scrE))$ and $\scrC_{\geq 5}=0$.  Since $d\ge4$ is equivalent to $\sum (d_i-1) \ge 2$, the claim follows.\\
(2)
By \eqref{lem:fstarf1}, the top right sum in \eqref{eqn:defineU} is zero.
Now again using the assumption that the $\fsf_i$ are all simple, any term of the bottom line of \eqref{eqn:defineU} is of the form
\[
\fsf_{n-j+1}(\hat{\asf}_1,\hdots,\hat{\asf}_k,\msf_j(\asf_{k+1},\hdots,\asf_{k+j}),\asf_{k+j+1},\hdots,\asf_n) = (\Ssf,\Tsf,0) \in\scrC_{e}
\]
where $e=\sum d_i + 2 - n$. Since by hypothesis $\sum d_i\ge n+2$, necessarily $e\ge4$ and thus again $\Ssf=\Tsf=0$.
\end{proof}

Whilst \ref{lem:fstarf} deals with $\sum (|\asf_i|-1) \ge 2$, the case $\sum (|\asf_i|-1) =1$ is mildly more tricky, and requires more assumptions.

\begin{prop}\label{prop:Un0}
For $n\ge3$, suppose that the set of maps $\left\{\fsf_1,\hdots,\fsf_{n-1}\right\}$ are simple, and have been chosen to adhere to Convention~\textnormal{\ref{convention}}.
\begin{enumerate}
\item\label{prop:Un01}
Consider inputs $\asf_i\in\{\xsf,\ysf,\Xsf,\Ysf\}$ with $1\le i\le n$, where precisely one of the $\asf_i$ has degree two. Then
\begin{equation}\label{eqn:Un2}
\Usf_n(\asf_1,\hdots,\asf_n) =
 \sum_{\ell=1}^{n-1}(-1)^{|\fsf_\ell(\asf_1,\hdots,\asf_\ell)|+1}\,\fsf_\ell(\asf_1,\hdots,\asf_\ell)\star\fsf_{n-\ell}(\asf_{\ell+1},\hdots,\asf_{n}).
 \end{equation}
\item\label{prop:Un02}
Suppose in addition that 
\begin{enumerate}
\item
for any choices of $\bsf_*\in\{\xsf,\ysf\}$
for $2\le\ell\le n-1$ each $\fsf_\ell(\bsf_1,\dots,\bsf_\ell)\in\scrC_1$ is
a $\K$-linear combination of $\ksf_i$ terms and $\esf_{j,j'}$ terms, and
\item
for $\bsf\in\{\xsf,\ysf\}$ and $\Bsf\in\{\Xsf,\Ysf\}$ both 
$\fsf_2(\bsf,\Bsf)$ and $\fsf_2(\Bsf,\bsf)$ are polynomial multiples of some $\Ksf_i$.
\end{enumerate}
Then $\Usf_n(\asf_1,\hdots,\asf_n) = 0_\scrC$.
\end{enumerate}
\end{prop}

\begin{proof}
(1) The bottom line of \eqref{eqn:defineU} is a linear combination of terms of the form
$\fsf_\ell(\bsf_1,\hdots,\bsf_\ell)$ where $\bsf_i\in\{\xsf,\ysf,\Xsf,\Ysf\}$, $2\leq \ell\leq n-1$, and 
\[
|\bsf_1| + \hdots + |\bsf_\ell| = \ell + 2.
\]
By \ref{lem:fstarf}\eqref{cor:U0f0} applied to $\ell$, $\Usf_\ell(\bsf_1,\hdots,\bsf_\ell)=0_\scrC$.  By the assumption that Convention~\textnormal{\ref{convention}} holds, $\fsf_\ell(\bsf_1,\hdots,\bsf_\ell)=0_\scrC$. Thus the bottom line of \eqref{eqn:defineU} is zero, proving the claim.\\
(2) We work by induction on $n\ge3$. For the initial case $n=3$,
\[
\Usf_3(\asf_1,\asf_2,\asf_3) = \pm\asf_1\star\fsf_2(\asf_2,\asf_3) \pm\fsf_2(\asf_1,\asf_2)\star\asf_3.
\]
Term by term, if the inputs to $\fsf_2$ have degree one, then the singleton is $\Xsf$ or $\Ysf$
and so the term vanishes by assumption (a) and \ref{lem:kstark}. 
On the other hand, if one of the arguments to $\fsf_2$ has degree two, then the singleton is
$\xsf$ or $\ysf$, and so the term vanishes by assumption (b) and \ref{lem:kstark}.
Thus $\Usf_3(\asf_1,\asf_2,\asf_3)=0_\scrC$.

If  $n>3$, then each individual term in \eqref{eqn:Un2} vanishes as follows:
\begin{itemize}
\item
$\asf_1\star\fsf_{n-1}(\asf_2,\hdots,\asf_n)$:
if $|\asf_1|=1$ then by induction $\Usf_{n-1}(\asf_2,\hdots,\asf_n)=0_\scrC$, so by the assumption that Convention~\textnormal{\ref{convention}} has been adhered to, $\fsf_{n-1}(\asf_2,\hdots,\asf_n)=0_\scrC$, and thus the term vanishes. Otherwise $\asf_1=\Xsf$ or $\Ysf$, and so the term vanishes by (a) and \ref{lem:kstark}.
\item
$\fsf_2(\asf_1,\asf_2)\star\fsf_{n-2}(\asf_3,\hdots,\asf_n)$: if $n=4$, then by (a) and (b) this is a product of a multiple of $\Ksf_*$ with $\ksf_i$ and $\esf_{j,j'}$ terms, so vanishes by \ref{lem:kstark}. 
For $n>4$, if $|\asf_1|=|\asf_2|=1$ then the second factor $\fsf_{n-2}(\asf_3,\hdots,\asf_n)=0_\scrC$ again by induction and Convention~\textnormal{\ref{convention}}. Otherwise, one of $\asf_1,\asf_2$ has degree two, so the term vanishes by (b), \eqref{eqn:f2degall} and \ref{lem:kstark}.
\item
$\fsf_\ell(\asf_1,\hdots,\asf_\ell)\star\fsf_{n-\ell}(\asf_{\ell+1},\hdots,\asf_n)$ with $\ell, n-\ell\ge3$:
one of the two factors vanishes again by induction and Convention~\textnormal{\ref{convention}}.
\end{itemize}
Similarly, the terms $\fsf_{n-2}(\asf_1,\hdots,\asf_{n-2})\star\fsf_2(\asf_{n-1},\asf_n)$
and $\fsf_{n-1}(\asf_1,\hdots,\asf_{n-1})\star\asf_n$ vanish by symmetry.
Thus $\Usf_n(\asf_1,\hdots,\asf_n)=0_\scrC$, as required.
\end{proof}

\subsection{All Higher \texorpdfstring{$\msf_n$}{mt} Products}

The following is the main result of this paper. Note that the sum in \eqref{thm:Umaintext1} generalises the right hand side of the displayed equations in \ref{prop:m2chaindegone}.

\begin{thm}\label{thm:Umaintext}
For any $n\geq 2$ and any decomposition $n=j+k$ with $j,k\geq 0$, the following statements hold.
\begin{enumerate}
\item\label{thm:Umaintext1} For all sequences $\asf_1,\hdots,\asf_n\in\{\xsf,\ysf\}$, where there are $j$ occurrences of $\xsf$ and $k$ occurrences of $\ysf$,  in $\scrC$ there is an equality
\[
\Usf_{n}(\asf_1,\hdots,\asf_n)=\uplambda_{j+1,k}\,\Xsf+\uplambda_{j,k+1}\Ysf - \dd\fsf_{n}(\asf_1,\hdots,\asf_n)
\]
where 
\[
\fsf_{n}(\asf_1,\hdots,\asf_n)=-\left(\sum_{i=0}^{j-1}\uplambda_{i,n+1-i}\,\ksf_{\kx+i-(j+1)}+
\sum_{i=j+2}^{n+1}\uplambda_{i,n+1-i}\,\ksf_{\kx+i-(j+1)}
+\esf_{n+1,k}\right).
\]
In particular, $\Usf_n$ is not affected by the order of the sequence of degree one inputs.
\item Suppose that $n\geq 3$ and $\asf_1,\hdots,\asf_n\in\{\xsf,\ysf,\Xsf,\Ysf,\upxi\}$, with $\sum |\asf_i|\ge n+1$. Then $\Usf_n(\asf_1,\hdots,\asf_n)=0_\scrC$. 
\end{enumerate}
\end{thm}

The case $n=2$ has already been considered: since $\Usf_2(\asf_1,\asf_2) = (-1)^{|\asf_1|+1}\asf_1\star\asf_2$, on $\asf_i\in\{\xsf,\ysf,\Xsf,\Ysf,\upxi\}$ the possible nonzero values of $\Usf_2$ are \ref{lem:compdegone} and \ref{lem:m2chaindegall}.

\begin{proof}
Fix some $n\geq 3$.  By induction we can \emph{choose} $\fsf_{i}$ for $2\leq i<n$ such that the following three conditions hold.
\begin{enumerate}
\item[(1$'$)] The condition (1) holds for all $i<n$.  Indeed, the case $n=2$ for (1) was established in \ref{prop:m2chaindegone}, after recalling that on degree one inputs $\Usf_2(\redbullet,\bluebullet)=\redbullet\star\bluebullet$. 
\item[(S)] The collection $\{\fsf_i\}_{i=1}^{n-1}$ is simple, since both $\fsf_1$ and $\fsf_2$ are simple by definition, \eqref{f2degone} and \eqref{eqn:f2degall}.
\item[(0)] Whenever $\Usf_i(\asf_1,\hdots,\asf_i)=0_\scrC$ for some $i<n$ and $\asf_1,\hdots,\asf_i\in\{\xsf,\ysf,\Xsf,\Ysf,\upxi\}$, then $\fsf_i(\asf_1,\hdots,\asf_i)=0_\scrC$. That is, we adhere to Convention~\ref{convention}.
\end{enumerate}

With these choices, given degree one inputs $\bsf_1,\hdots,\bsf_{n-1}\in\{\xsf,\ysf\}$ where there are $\mathfrak{j}$ occurrences of $\xsf$ and $\mathfrak{k}$ occurrences of $\ysf$, the following statements hold.
\begin{enumerate}
\item[(a)]  After passing to cohomology, writing in terms of the basis, then applying $\upiota$, as explained in \S\ref{subsec:MinModelRecap} (and underneath \ref{prop:m2chaindegone}) $\msf_{n-1}(\bsf_1,\hdots,\bsf_{n-1})=\uplambda_{\mathfrak{j}+1,\mathfrak{k}}\,\Xsf
+\uplambda_{\mathfrak{j},\mathfrak{k}+1}\Ysf$.
\item[(b)] There are equalities
\begin{align*}
\fsf_{n-1}(\bsf_1,\hdots,\bsf_{n-1})&=-\left(\sum_{i=0}^{\mathfrak{j}-1}\uplambda_{i,n-i}\,\ksf_{\kx+i-(\mathfrak{j}+1)}+
\sum_{i=\mathfrak{j}+2}^{n}\uplambda_{i,n-i}\,\ksf_{\kx+i-(\mathfrak{j}+1)}
+\esf_{n,\mathfrak{k}}\right)\\
&=-\left(\sum_{i=\mathfrak{k}+2}^{n}\uplambda_{n-i,i}\,\ksf_{\kx+\mathfrak{k}-i}+
\sum_{i=0}^{\mathfrak{k}-1}\uplambda_{n-i,i}\,\ksf_{\kx+\mathfrak{k}-i}
+\esf_{n,\mathfrak{k}}\right)
\end{align*}
where the second line is just the reindexing $i\mapsto n-i$, using $n-1=\mathfrak{j}+\mathfrak{k}$.
\item[(c)] For any $2\le\ell\le n-2$, 
\[
\fsf_\ell(\bsf_1,\hdots,\bsf_\ell)\star\fsf_{n-\ell}(\bsf_{\ell+1},\hdots,\bsf_n) = 0_\scrC.
\]
Indeed, each $\fsf_\ell$ evaluated on degree one inputs is a linear combination of $\ksf_*$ and $\esf_{*,*}$, and these multiply to zero by \ref{lem:kstark}.
\end{enumerate}

Now consider $\Usf_n(\asf_1,\hdots,\asf_n)$ for $\asf_1,\hdots,\asf_n\in\{\xsf,\ysf\}$, where there are $j$ occurrences of $\xsf$ and $k$ occurrences of $\ysf$.   
We consider the two rows of \eqref{eqn:defineU} separately.
In the first row, for each $2\le\ell\le n-2$, the term
$\fsf_\ell(\asf_1,\hdots,\asf_\ell)\star\fsf_{n-\ell}(\asf_{\ell+1},\hdots,\asf_n)$
vanishes by (c).
In the second row of \eqref{eqn:defineU}, whenever $j$ is such that $2\leq j\leq n-2$ the term
\[
\fsf_{n-j+1}(\hat{\asf}_1,\hdots,\hat{\asf}_k,\msf_j(\asf_{k+1},\hdots,\asf_{k+j}),\asf_{k+j+1},\hdots,\asf_n)
\]
has $n-j\geq 2$ arguments of degree one and a single argument $\msf_j(\asf_{k+1},\hdots,\asf_{k+j})$
of degree two.  This vanishes by \ref{prop:Un0}\eqref{prop:Un02} and (0) above, since the hypotheses there hold by (b) and \eqref{eqn:f2degall}.
These vanishing statements leave only
\begin{equation}
\begin{aligned}
\Usf_n(\asf_1,\hdots,\asf_n) =&\phantom{-} \asf_1\star\, \fsf_{n-1}(\asf_2,\hdots,\asf_n) +  \fsf_{n-1}(\asf_1,\hdots,\asf_{n-1})\star \asf_n\,\\
&- \fsf_{2}(\asf_1,\msf_{n-1}(\asf_2,\hdots,\asf_n)) - \fsf_{2}(\msf_{n-1}(\asf_1,\hdots,\asf_{n-1}),\asf_n).
\end{aligned}\label{eqn:genUndegone}
\end{equation}
The heart of the proof is to use the inductive hypothesis to prove the following claim.

\medskip
\noindent
\textbf{Claim A.} With notation as above, \eqref{eqn:genUndegone} simplifies to
\[
\Usf_n(\asf_1,\hdots,\asf_n)=(\,a^{k}v_2^{\ordP-(n+1)}A_{\geq n+1}\cdot\Zsf_1, \,\,b^jw_2^{\ordP-(n+1)}B_{\geq n+1}\cdot\Zsf_2,\,0).
\]  
The proof of this claim splits into two cases. 

\medskip
\noindent
\emph{Case 1: $\asf_1=\asf_n$.}  In this case, the number of occurrences $\mathfrak{k}$ of $\ysf$ in $\asf_1,\hdots,\asf_{n-1}$ equals the number of occurrences of $\ysf$ in $\asf_2,\hdots,\asf_n$.  Similarly for the number of $\xsf$'s.  The inductive hypothesis (a) implies that $\msf_{n-1}(\asf_2,\hdots,\asf_n)=\msf_{n-1}(\asf_1,\hdots,\asf_{n-1})$, and this clearly has degree two.  But by \eqref{eqn:f2degall} $\fsf_2$ is antisymmetric on any mix of degree one and degree two inputs, thus the bottom two terms in \eqref{eqn:genUndegone} cancel to give
\begin{align}
\Usf_n(\asf_1,\hdots,\asf_n) &=
 \asf_1\star\, \fsf_{n-1}(\asf_2,\hdots,\asf_n) +  \fsf_{n-1}(\asf_1,\hdots,\asf_{n-1})\star \asf_n.\label{eqn:gettingthere}
\end{align}
Now again since the number of occurrences $\mathfrak{k}$ of $\ysf$ in $\asf_1,\hdots,\asf_{n-1}$ equals the number of occurrences of $\ysf$ in $\asf_2,\hdots,\asf_n$, and similarly for the $\xsf$'s,  the inductive hypothesis (b) implies that $\fsf_{n-1}(\asf_2,\hdots,\asf_n)=\fsf_{n-1}(\asf_1,\hdots,\asf_{n-1})$, and furthermore both terms equal expressions of the form
\[
\begin{cases}
-\sum \uplambda_*\ksf_* - \esf_{n,k}
 &\mbox{if } \asf_1=\asf_n=\xsf \mbox{ (so }\mathfrak{k}=k)\\
-\sum \uplambda_*\ksf_* - \esf_{n,k-1} &\mbox{if } \asf_1=\asf_n=\ysf \mbox{ (so }\mathfrak{k}=k-1).
\end{cases}
\]
Substituting this into \eqref{eqn:gettingthere}, and using the fact that every degree one input anti-commutes with $\ksf_i$ by \ref{cor:xandkanticommute}, the above \eqref{eqn:gettingthere} simplifies to
\begin{align*}
\Usf_{n}(\asf_1,\hdots,\asf_n) &=
\begin{cases}
-(\xsf\star \esf_{n,k\phantom{-1}}+  \esf_{n,k}\star\xsf) &\mbox{if } \asf_1=\asf_n=\xsf\\
-(\ysf\star \esf_{n,k-1} +  \esf_{n,k-1}\star\ysf) &\mbox{if } \asf_1=\asf_n=\ysf
\end{cases}\\
&=-(\xsf\star \esf_{n,k\phantom{-1}}+  \esf_{n,k}\star\xsf) \tag{by \ref{lem:explusxe}\eqref{lem:explusxe1}}\\
&=(\,a^{k}v_2^{\ordP-(n+1)}A_{\geq n+1}\cdot\Zsf_1, \,\,b^jw_2^{\ordP-(n+1)}B_{\geq n+1}\cdot\Zsf_2,\,0)\tag{by \ref{lem:explusxe}\eqref{lem:explusxe2}}
\end{align*}
since $n=j+k$, which verifies Claim A in case 1.

\medskip
\noindent
\emph{Case 2: $\asf_1\neq\asf_n$.}  Write $k_1$ for the number of occurrences of $\ysf$ in $\asf_1,\hdots,\asf_{n-1}$, and $k_2$ for the number of occurrences of $\ysf$ in $\asf_2,\hdots,\asf_{n}$.  Thus, by inductive hypothesis (a), 
\begin{align*}
\msf_{n-1}(\asf_1,\hdots,\asf_{n-1})&=\uplambda_{n-k_1,k_1}\Xsf+\uplambda_{n-k_1-1,k_1+1}\Ysf\\
\msf_{n-1}(\asf_2,\hdots,\asf_n)&=\uplambda_{n-k_2,k_2}\Xsf+\uplambda_{n-k_2-1,k_2+1}\Ysf,
\end{align*}
and so the general formula \eqref{eqn:genUndegone} reads 
\begin{alignat*}{6}
\Usf_n(\asf_1,\hdots,\asf_n) &=\phantom{-} \asf_1\star\, \fsf_{n-1}(\asf_2,\hdots,\asf_n) &&+  \fsf_{n-1}(\asf_1,\hdots,\asf_{n-1})\star \asf_n\,\\
&\,\,- \fsf_{2}(\asf_1,\uplambda_{n-k_2,k_2}\Xsf+\uplambda_{n-k_2-1,k_2+1}\Ysf) &&- \fsf_{2}(\uplambda_{n-k_1,k_1}\Xsf+\uplambda_{n-k_1-1,k_1+1}\Ysf,\asf_n).
\end{alignat*}
Substituting in the inductive hypothesis (b) for the $\fsf_{n-1}$ terms, 
\begin{alignat*}{6}
\Usf_n(\asf_1,\hdots,\asf_n) &=-\asf_1\star(\textstyle\sum \uplambda_*\ksf_* - \esf_{n,k_2}) && -(\textstyle\sum \uplambda_*\ksf_* - \esf_{n,k_1})\star\asf_n\\
&\,\,- \fsf_{2}(\asf_1,\uplambda_{n-k_2,k_2}\Xsf+\uplambda_{n-k_2-1,k_2+1}\Ysf) &&- \fsf_{2}(\uplambda_{n-k_1,k_1}\Xsf+\uplambda_{n-k_1-1,k_1+1}\Ysf,\asf_n).
\end{alignat*}
Since $\asf_n$ anti-commutes with any $\ksf_*$ by \ref{cor:xandkanticommute}, and up to sign on the bottom right we can swap the order of the inputs by \eqref{eqn:f2degall}, it follows that $\Usf_n(\asf_1,\hdots,\asf_n)-(\asf_1\star\esf_{n,k_2}+\esf_{n,k_1}\star\asf_n)$ equals
\begin{equation}
\begin{aligned}
&-\asf_1\star(\textstyle\sum \uplambda_*\ksf_* ) && +\asf_n\star(\textstyle\sum \uplambda_*\ksf_*)\\
&- \fsf_{2}(\asf_1,\uplambda_{n-k_2,k_2}\Xsf+\uplambda_{n-k_2-1,k_2+1}\Ysf) &&+ \fsf_{2}(\asf_n,\uplambda_{n-k_1,k_1}\Xsf+\uplambda_{n-k_1-1,k_1+1}\Ysf).
\end{aligned}\label{eqn:wantzero}
\end{equation}

\noindent
\textbf{Claim B.} The expression \eqref{eqn:wantzero} is zero. 

\medskip
To ease notation, the proof of Claim B splits into two subcases.

\smallskip
\noindent
\emph{Subcase \textnormal{B(}$1$\textnormal{)}} $\asf_1=\xsf$, so $\asf_n=\ysf$.  The precise indices on the top leftmost term in \eqref{eqn:wantzero} are
\begin{align*}
-\asf_1\star(\textstyle\sum \uplambda_*\ksf_*)&=-\xsf\star\left(\sum_{i=k_2+2}^{n}\uplambda_{n-i,i}\,\ksf_{\kx+k_2-i}+
\sum_{i=0}^{k_2-1}\uplambda_{n-i,i}\,\ksf_{\kx+k_2-i}\right).
\end{align*}
By \ref{cor:xandkanticommute} $\xsf\star\ksf_i=\Ksf_i$, thus $-\asf_1\star(\textstyle\sum \uplambda_*\ksf_*)$ is the sum $-\sum \uplambda_{n-i,i}\Ksf_{\kx+k_2-i}$ containing all $\uplambda_{n-i,i}$ terms except $i=k_2$ and $i=k_2+1$.  Now since $\asf_1=\xsf$, by \eqref{eqn:f2degall}, it follows that
\[
- \fsf_{2}(\asf_1,\uplambda_{n-k_2,k_2}\Xsf+\uplambda_{n-k_2-1,k_2+1}\Ysf)=-\uplambda_{n-k_2-1,k_2+1}\Ksf_{\kx-1}.
\]
Thus the two leftmost terms in \eqref{eqn:wantzero} equal $-\sum \uplambda_{n-i,i}\Ksf_{\kx+k_2-i}$, where the sum contains all $\uplambda_{n-i,i}$ terms except $i=k_2$.

On the other hand, the precise indices on the top rightmost term in \eqref{eqn:wantzero} are
\begin{align*}
\asf_n\star(\textstyle\sum \uplambda_*\ksf_*)&=\ysf\star\left(\sum_{i=k_1+2}^{n}\uplambda_{n-i,i}\,\ksf_{\kx+k_1-i}+
\sum_{i=0}^{k_1-1}\uplambda_{n-i,i}\,\ksf_{\kx+k_1-i}\right).
\end{align*}
Since $\asf_1=\xsf$, it follows that $k_1=k_2-1$, and thus
\begin{align*}
\asf_n\star(\textstyle\sum \uplambda_*\ksf_*)&=\ysf\star\left(\sum_{i=k_2+1}^{n}\uplambda_{n-i,i}\,\ksf_{\kx+k_2-1-i}+
\sum_{i=0}^{k_2-2}\uplambda_{n-i,i}\,\ksf_{\kx+k_2-1-i}\right).
\end{align*}
By \ref{cor:xandkanticommute} $\ysf\star\Ksf_i=\Ksf_{i+1}$, thus $\asf_n\star(\textstyle\sum \uplambda_*\ksf_*)$ is the sum $\sum \uplambda_{n-i,i}\Ksf_{\kx+k_2-i}$ containing all $\uplambda_{n-i,i}$ terms except $i=k_2$ and $i=k_2-1$.  Again by \eqref{eqn:f2degall} it follows that
\begin{align*}
\fsf_{2}(\asf_n,\uplambda_{n-k_1,k_1}\Xsf+\uplambda_{n-k_1-1,k_1+1}\Ysf)=\uplambda_{n-k_1,k_1}\Ksf_{\kx+1}=\uplambda_{n-k_2+1,k_2-1}\Ksf_{\kx+1}.
\end{align*}
Thus the two rightmost terms in \eqref{eqn:wantzero} equal $+\sum \uplambda_{n-i,i}\Ksf_{\kx+k_2-i}$ where the sum contains all $\uplambda_{n-i,i}$ terms except $i=k_2$.

Combining, the leftmost terms in \eqref{eqn:wantzero} cancel the rightmost terms, and thus \eqref{eqn:wantzero} is zero. This verifies Claim B in subcase B($1$).

\smallskip
\noindent
\emph{Subcase \textnormal{B(}$2$\textnormal{)}} $\asf_1=\ysf$, so $\asf_n=\xsf$ and now $k_2=k_1-1$.  In a similar manner to the above,  
\begin{align*}
-\asf_1\star(\textstyle\sum \uplambda_*\ksf_*)&=-\ysf\star\left(\sum_{i=k_2+2}^{n}\uplambda_{n-i,i}\,\ksf_{\kx+k_2-i}+
\sum_{i=0}^{k_2-1}\uplambda_{n-i,i}\,\ksf_{\kx+k_2-i}\right)\\
&=-\ysf\star\left(\sum_{i=k_1+1}^{n}\uplambda_{n-i,i}\,\ksf_{\kx+k_1-1-i}+
\sum_{i=0}^{k_1-2}\uplambda_{n-i,i}\,\ksf_{\kx+k_1-1-i}\right)\\
&=-\sum_{i=k_1+1}^{n}\uplambda_{n-i,i}\,\Ksf_{\kx+k_1-i}-
\sum_{i=0}^{k_1-2}\uplambda_{n-i,i}\,\Ksf_{\kx+k_1-i},
\end{align*}
and so the two leftmost terms of \eqref{eqn:wantzero} now equal $-\sum \uplambda_{n-i,i}\Ksf_{\kx+k_1-i}$ where the sum contains all $\uplambda_{n-i,i}$ terms except $i=k_1$.  Similarly the rightmost terms of \eqref{eqn:wantzero} give the same sum, with the sign swapped.  Thus again the leftmost and rightmost terms cancel, so \eqref{eqn:wantzero} is zero.  This verifies Claim B in subcase B(2).

\medskip
Thus Claim B holds, and so
\[
\Usf_{n}(\asf_1,\hdots,\asf_n) = -(\asf_1\star\esf_{n,k_2}+\esf_{n,k_1}\star\ysf).
\]
Now if $\asf_1=\xsf$, necessarily $\asf_n=\ysf$ and so there $k_1=k-1$ and $k_2=k$.  Similarly, if $\asf_1=\ysf$, then $k_1=k$ and $k_2=k-1$.  Consequently, it follows that
\begin{align*}
\Usf_{n}(\asf_1,\hdots,\asf_n) &=
\begin{cases}
-(\xsf\star \esf_{n,k\phantom{-1}}+  \esf_{n,k-1}\star\ysf) &\mbox{if } \asf_1=\xsf\\
-(\ysf\star \esf_{n,k-1} +  \esf_{n,k}\star\xsf) &\mbox{if } \asf_1=\ysf
\end{cases}\\
&=-(\xsf\star \esf_{n,k\phantom{-1}}+  \esf_{n,k}\star\xsf) \tag{by \ref{lem:explusxe}\eqref{lem:explusxe1}}\\
&=(\,a^{k}v_2^{\ordP-(n+1)}A_{\geq n+1}\cdot\Zsf_1, \,\,b^jw_2^{\ordP-(n+1)}B_{\geq n+1}\cdot\Zsf_2,\,0).\tag{by \ref{lem:explusxe}\eqref{lem:explusxe2}}
\end{align*}
This completes the proof of Claim A in case 2, and thus the proof of Claim A overall.

\medskip
Thus always $\Usf_n(\asf_1,\hdots,\asf_n)=(\,a^{k}v_2^{\ordP-(n+1)}A_{\geq n+1}\cdot\Zsf_1, \,\,b^jw_2^{\ordP-(n+1)}B_{\geq n+1}\cdot\Zsf_2,\,0)$, and from here the proof is elementary.  Decomposing the right hand side using \eqref{eq:splitAgeq}, 
\begin{align*}
 \Usf_{n}&(\asf_1,\hdots,\asf_n)\\
&=(\, a^kA_{n+1}\cdot \Zsf_1, \,\, b^jB_{n+1}\cdot\Zsf_2,\,0)+(\,a^kv_2^{\ordP-(n+1)}A_{\geq n+2}\cdot\Zsf_1, \,\,b^jw_2^{\ordP-(n+1)}B_{\geq n+2}\cdot\Zsf_2,\,0)\\
 &=(\, a^kA_{n+1}\cdot \Zsf_1, \,\, b^{j}B_{n+1}\cdot\Zsf_2,\,0)+\dd(\esf_{n+1,k}).\tag{by \ref{cor:inductwithe}}
\end{align*}
Now by definition (reading \ref{def:AandB} backwards)
\[
a^kA_{n+1}=\uplambda_{0,n+1}a^{\kx-(n+1)+k}+\hdots+\uplambda_{j,k+1}\,a^{\kx-1}+\uplambda_{j+1,k}\,a^{\kx}+\hdots+\uplambda_{n+1,0}\,a^{\kx+k},
\] 
which is a polynomial in $a$. The corresponding polynomial $g\in\mathbb{K}[b]$ defined in \ref{cor:Aintosum} is equal to $b^jB_{n+1}$, again by definition \ref{def:AandB}. Hence, applying \ref{cor:Aintosum} we may write
\begin{align*}
(\, a^kA_{n+1}\cdot \Zsf_1, \,\, b^{j}B_{n+1}\cdot\Zsf_2,\,0)&=\uplambda_{j+1,k}\,\Xsf
+\uplambda_{j,k+1}\Ysf\\
&+ \dd \left(\sum_{i=0}^{j-1}\uplambda_{i,n+1-i}\,\ksf_{\kx+i-(j+1)}+
\sum_{i=j+2}^{n+1}\uplambda_{i,n+1-i}\,\ksf_{\kx+i-(j+1)}\right).
\end{align*}
Consequently we can choose $\fsf_n$ on degree one inputs to satisfy the condition in (1).  

\medskip
For (2), consider $\Usf_n(\asf_1,\hdots,\asf_n)$
with $\asf_i\in\mathsf{A}_{d_i}$. 
If $\sum d_i\ge n+2$, then $\Usf_n(\asf_1,\hdots,\asf_n)=0_\scrC$ by \ref{lem:fstarf}\eqref{cor:U0f0}.
If $\sum d_i = n+1$, that is one of the inputs has degree two and all the others have degree one,
then $\Usf_n(\asf_1,\hdots,\asf_n)=0_\scrC$ by \ref{prop:Un0}\eqref{prop:Un02}: as before, the hypotheses there hold by induction hypothesis (b), (0) and \eqref{eqn:f2degall}.

\medskip
This establishes (1) and (2) for $n$.  By \ref{lem:Uis0}, we can choose the map $\fsf_n$ to be zero whenever at least one input is not degree one.  Since further on degree one inputs $\fsf_n$ is a combination of $\ksf$ and $\esf$, it follows that $\fsf_n$ is simple.  This verifies (S) and (0) replacing $n$ by $n+1$.  Thus (1$'$), (S) and (0) all hold replacing $n$ by $n+1$, allowing the induction to proceed.  
\end{proof}

\subsection{Summary}
The previous subsections combine to verify the following, which is the main result in the introduction.  To be consistent with the notation there (where $\xsf$ and $\Xsf$ are the classes in cohomology), this subsection introduces the one further abuse of notation that $\xsf,\Xsf$ now also denote their classes in cohomology $[\xsf],[\Xsf]$, etc.  Since no more calculations with chains are needed, this introduces no ambiguities.

Thus consider the graded vector space $\mathsf{A}=\bigoplus_{i\in\mathbb{Z}}\mathsf{A}_i$, where
\[
\mathsf{A}_i=
\begin{cases}
\mathbb{K} & \mbox{if }i=0,3\\
\mathbb{K}^2 & \mbox{if }i=1,2\\
0&\mbox{else},
\end{cases}
\]
and write $1$ for the basis of $\mathsf{A}_0$, $\xsf,\ysf$ for the basis of $\mathsf{A}_1$, and $\Xsf,\Ysf$ for the basis of $\mathsf{A}_2$, and $\upxi$ for the basis of $\mathsf{A}_3$.  Recall that all $\Ainf$-algebras in this paper are strictly unital.
 
\begin{cor}\label{thm: main}
The following defines an $\Ainf$-structure on $\mathsf{A}$, and furthermore the resulting $\Ainf$-algebra is quasi-isomorphic to the \textnormal{DG}-algebra $\scrC$. 
\begin{enumerate}
\item\label{maintext:one} For any $n\geq 2$ and any decomposition $n=j+k$ with $j,k\geq 0$,
\[
\msf_{n}(\underbrace{\mathsf{x},\hdots,\mathsf{x}}_j,\underbrace{\mathsf{y},\hdots,\mathsf{y}}_k)=\uplambda_{j+1,k}\,\Xsf+\uplambda_{j,k+1}\Ysf,
\] 
where the $\uplambda$'s are the coefficients from the glue in \eqref{eq!glue}.  
\item More generally, $\msf_n$ with $n\geq 2$ applied only to degree one inputs (so, combinations of $\mathsf{x}$ and $\mathsf{y}$) does not depend on the order of those degree one inputs, and thus is determined by \eqref{maintext:one} above. 
\item
The only other non-zero products are 
\[
-\msf_2(\xsf,\Xsf) = \upxi=\msf_2(\Xsf,\xsf) \qquad {-\msf_2(\ysf,\Ysf)} = \upxi =\msf_2(\Ysf,\ysf).
\]\end{enumerate}
\end{cor}
\begin{proof}
This is now a direct consequence of \ref{thm:Umaintext} and \eqref{eqn:iotam2degall}, given those followed the construction for Kadeishvili's Theorem \ref{thm:kadeishvili} outlined in \S\ref{subsec:MinModelRecap}.
\end{proof}

\section{Corollaries}\label{sec:DefThCor}

This section works over $\mathbb{K}=\mathbb{C}$, or any algebraically closed field of characteristic zero (which we still denote $\mathbb{C}$), and gives the deformation theory and mirror symmetry consequences of the previous sections.  Throughout, given scalars $\uplambda_{jk}\in\mathbb{C}$ subject to \ref{setup:key}, consider the variety $\scrX$ defined by \eqref{eq!glue}, with associated $\mathbb{P}^1\cong\Curve\subset\scrX$.

\subsection{NC Deformation Theory Summary}\label{subsec:NCdeftheory}
Maintaining the setup above, the sheaf $\scrO_\Curve\in\coh\scrX$ gives rise to the noncommutative deformation functor
\[
\scrD\textnormal{ef}_{\scrO_\Curve}\colon\mathfrak{art}_1\to\textnormal{Set}
\]
defined in e.g.\ the survey \cite[2.1]{Kawamata_survey} or \cite{BoothDerived}.  
As is well known e.g.\ \cite[\S2]{Kawamata_survey}, this functor admits a prorepresentable hull, written $\Lambda_{\mathrm{def}}$, called the NC deformation algebra.  

As is also well-known (see e.g.\ \cite[p615]{TodaUtah}, or \cite[\S5]{Kawamata_survey}), $\Lambda_{\mathrm{def}}$ can be explicitly presented using the $\Ainf$-structure of any DGA quasi-isomorphic to $\End^{\mathrm{DG}}_X(\scrI)$ as in \S\ref{sec:DGAC}, where $\scrI$ is an injective resolution of $\scrO_\Curve$. Indeed, the  $\Ainf$-morphisms give a morphism
\[
\mathsf{m}\colonequals\sum_{i\geq 2}\mathsf{m}_i\colon\bigoplus_{i\geq 2}\Ext^1_\scrX(\scrO_\Curve,\scrO_\Curve)^{\otimes i}\to\Ext^2_\scrX(\scrO_\Curve,\scrO_\Curve)
\]
which dualises to give a presentation
\[
\Lambda_{\mathrm{def}}\cong\frac{\mathbb{C}\llangle \Ext^1_\scrX(\scrO_\Curve,\scrO_\Curve)^*\rrangle}{\textnormal{Image}(\mathsf{m}^*)}=\frac{\mathbb{C}\llangle x,y\rrangle}{\textnormal{Image}(\mathsf{m}^*)}.
\] 
By \ref{prop:DGAcheck}, we can compute the above instead using the $\Ainf$-algebra structure on the more manageable DGA $\scrC$, and thus below we will freely use \ref{thm: main} to describe $\Lambda_{\mathrm{def}}$.

\subsection{Superpotentials and Necklaces}\label{subsec:necklacesmain}
Following the conventions in e.g.\ \cite{DWZ, BW2, Davison}, consider the $\C$-linear map $\partial_{x}\colon \ring\to\ring$ which simply `strikes off' the leftmost $x$ of each monomial.  Thus,  on monomials 
\[
\partial_{x}(m)=\begin{cases} n & \text{if } m = xn \\ 0 & \text{otherwise,} \end{cases}
\]
with $\partial_{y}$ being defined similarly.  In contrast, the cyclic derivative is the $\C$-linear map $\dcyc_{x}\colon \ring\to\ring$ which on monomials sends
\[
x_{i_1}\hdots x_{i_t} \mapsto
\sum_{j=1}^t \partial_{x}(x_{i_j}x_{i_{j+1}}\hdots x_{i_t}\cdot x_{i_1}\hdots x_{i_{j-1}}),
\]
with $\dcyc_y$ being defined similarly.  For $f\in\ring$, the Jacobi algebra is defined to be
\[
\Jac(f) \colonequals \frac{\ring}{\lcl\dcyc_{x}f, \dcyc_{y}f\rcl}
\]
where $\lcl\dcyc_{x}f, \dcyc_{y}f\rcl$ is the closure of the two-sided ideal $(\dcyc_{x}f, \dcyc_{y}f)$.

\medskip 
Recall from the introduction that the free necklace polynomial is defined to be
\[
\NeckPoly_{j,k}(x,y)\colonequals  \frac{1}{j+k}\sum_{m\in\Orb_{j,k}}|m| \cdot p_m.
\]

\begin{remark}
It is immediate from the definition that, if we instead work up to cyclic rotation, we may replace $|m|p_m$ with the $m$ distinct representatives of the orbit.  In this way,  up to cyclic permutation, $\NeckPoly_{j,k}$ is clearly then equal to all terms with $j$ occurrences of $x$ in the free algebra expansion of $\frac{1}{j+k}(x+y)^{j+k}$.
\end{remark}

\begin{notation}
Write $\Mono_{j,k}$ for the sum of all monomials in $x$ and $y$, inside the free algebra, where there are $j$ occurrences of $x$, and $k$ occurrences of $y$. 
\end{notation}

For calibration, $\Mono_{2,2}=xxyy+xyyx+yyxx+yxxy+xyxy+yxyx$.  The following is elementary, where $\mathbb{K}=\mathbb{C}$ is used to allow the denominators in the proof.

\begin{lemma}\label{lem:NecksDiff}
$\dcyc_x(\NeckPoly_{j+1,k})=\Mono_{j,k}$ and $\dcyc_y(\NeckPoly_{j,k+1})=\Mono_{j,k}$.
\end{lemma}
\begin{proof}
This follows at once from the definitions, since
\[
\dcyc_x(\NeckPoly_{j+1,k})= \sum_{m\in\Orb_{j+1,k}}\frac{|m|}{j+1+k} \cdot \dcyc_x (p_m)=\Mono_{j,k}.\qedhere
\]
\end{proof}

\subsection{Deformation Theory Corollaries} 
The following is now immediate, and is one of the main results.

\begin{cor}\label{cor:NCdefs}
The \textnormal{NC} deformation algebra $\Lambda_{\mathrm{def}}$ of $\scrO_\Curve\in\coh\scrX$ is described by
\[
\Lambda_{\mathrm{def}}\cong\Jac(\mathsf{W})=\frac{\ringtwo}{\lcl\dcyc_{x}\mathsf{W}, \dcyc_{y}\mathsf{W}\rcl}
\]
where $\mathsf{W}=\sum\uplambda_{jk}\NeckPoly_{j,k}\in\ringtwo$ is the sum of free necklace polynomials, and the $\uplambda_{jk}$ are the data in the glue \eqref{eq!glue!intro}.
\end{cor}
\begin{proof}
As in \cite[below 4.3]{Kawamata_survey}, and summarised in \S\ref{subsec:NCdeftheory} above, the NC deformation algebra $\Lambda_{\mathrm{def}}$ can be presented as the completed tensor algebra over $\mathsf{A}_1^*$, subject to the relations induced by the map
\begin{equation}\label{eqn:dualgivespres}
\mathsf{A}_2^*\to \hat{\mathrm{T}}_\mathbb{C}(\mathsf{A}_1^*) 
\end{equation}
which is dual to $\sum_{i=2}^\infty\msf_i$.  Now, in \ref{thm: main} the term $\uplambda_{j,k}\Xsf$ appears precisely in those $\msf_n$ when there are $j-1$ occurrences of $\xsf$ and $k$ occurrences of $\ysf$.  Similarly the term $\uplambda_{j,k}\Ysf$ appears precisely in those $\msf_n$ when there are $j$ occurrences of $\xsf$ and $k-1$ occurrences of~$\ysf$.  Writing $x=\xsf^*$ and $y=\ysf^*$, then under \eqref{eqn:dualgivespres} it follows that
\begin{align*}
\Xsf^*&\mapsto\sum_{j+k\geq 3}\uplambda_{j,k}\Mono_{j-1,k}(x,y),\\
\Ysf^*&\mapsto\sum_{j+k\geq 3}\uplambda_{j,k}\Mono_{j,k-1}(x,y).
\end{align*}
The first relation is thus $\sum_{j+k\geq 3}\uplambda_{j,k}\Mono_{j-1,k}$, which equals $\dcyc_{x}\mathsf{W}$ by \ref{lem:NecksDiff}.  In a similar way, the second relation equals $\dcyc_{y}\mathsf{W}$. 
\end{proof}

The above then immediately recovers Katz--Namba \cite{Katz, Namba}.

\begin{cor}\label{cor:Cdefs}
Classical commutative deformations of $\scrO_\Curve\in\coh\scrX$ are prorepresented by
\[
\Lambda_{\mathrm{def}}^{\ab}\cong\Jac(\mathsf{W})^{\ab}=\frac{\commringtwo}{(\dcyc_{x}\mathsf{V}, \dcyc_{y}\mathsf{V})}
\]
where $\mathsf{V}=\sum\uplambda_{j,k}\NeckPoly_{j,k}^{\,\ab}=\sum\frac{\uplambda_{j,k}}{j+k}{j+k \choose k}x^jy^k$.
\end{cor}
\begin{proof}
This follows from \ref{cor:NCdefs}, since as is standard the versal space for commutative deformations is the abelianisation of the noncommutative version (e.g.\ \cite[(13)]{TodaUtah}).
\end{proof}

\subsection{Mirror Models}
The next application of \ref{thm: main} is categorical.  Given a quiver with superpotential $(Q,\mathsf{W})$, Ginzburg \cite{Ginz} associates a 3-CY category $\scrD_\mathsf{W}$.  It is a basic question to find geometric models for such categories, on both the A- and B-sides of mirror symmetry.  

\begin{cor}\label{cor:mirror}
Let $\mathsf{W}\in\mathbb{C}\langle x,y\rangle$ and consider the associated \textnormal{3-CY} category $\scrD_\mathsf{W}$.  If there exist scalars $\uplambda_{jk}$ for which $\mathsf{W}=\sum\uplambda_{jk}\NeckPoly_{j,k}$,  then there exists a smooth \textnormal{3}-fold $\scrX$ and rational curve $\Curve\subset\scrX$ such that 
\[
\Db(\coh\scrX)\supset \langle \scrO_\Curve\rangle\cong \scrD_\mathsf{W}.
\]
\end{cor}
\begin{proof}
On one hand, the category $ \langle \scrO_\Curve\rangle$ can be described using the $\Ainf$-structure on the DG-endomorphism ring of $\scrO_\Curve$, which is computed in \ref{thm: main} above.  On the other hand,
by definition $\scrD_\mathsf{W}=\scrD_{\mathrm{fd}}(\Upgamma)$ where $\Upgamma$ is the Ginzburg DGA associated to the quiver with potential.  This category has a canonical heart, generated by the simple modules  indexed over the vertices of the quiver, which in our case (the two-loop quiver) consists of a unique simple $S$.  Thus $\scrD_\mathsf{W}$ is described using the $\Ainf$-structure on the DG-endomorphism ring of $S$, which by \cite[\S A.15]{Keller} is entirely encoded by the pairing and the superpotential on degree one inputs.  But this is precisely the $\Ainf$-structure described in \ref{thm: main} above, and so the result follows.
\end{proof}

\subsection{Relationship to Physics}\label{subsec:physics}
The majority of the physics literature for rational curves (e.g.\ \cite{Katz,AspinwallKatz,CM}) considers superpotentials viewed inside the commutative power series ring $\commringtwo$.  However, \cite[\S3.3.1]{Ferrari} observed that ordering ambiguities arise when D5 branes are considered, and proposed the matrix rule \cite[(3.18)]{Ferrari} over 10 years before noncommutative deformation theory entered rational curves \cite{DW1}.  This subsection explains why Ferrari's predicted rule is consistent with \ref{cor:NCdefs}, and in the process establishes  \ref{cor:NCdefs} as the mathematically precise formulation of these physical predictions.

In 2003 Ferrari \cite[(3.18)]{Ferrari} predicted that in any commutative potential, $x^jy^k$ should get replaced by a noncommutative term in variables $X$ and $Y$, under the rule
\[
x^jy^k\longrightarrow \frac{j!\, k!}{(j+k)!}\oint_{C_0}\frac{dz}{2\uppi \mathrm{i}}\,z^{-k-1}\mathrm{tr}(X+Yz)^{j+k}.
\]
Taking the coefficient to the other side, the prediction can be rewritten
\[
\begin{psmallmatrix} j+k\\k\end{psmallmatrix}x^jy^k\longrightarrow \oint_{C_0}\frac{dz}{2\uppi \mathrm{i}}\,z^{-k-1}\mathrm{tr}(X+Yz)^{j+k}
\]
and so the commutative potential $\mathsf{V}=\sum\frac{\uplambda_{jk}}{j+k} {j+k \choose k}x^jy^k$ of \ref{cor:Cdefs} (known to \cite{Katz, Namba}) should be replaced with an element in the free algebra, via the rule
\begin{equation}
\mathsf{V} \longrightarrow \mathbb{W}=\sum_{j,k}\frac{\uplambda_{jk}}{j+k}\oint_{C_0}\frac{dz}{2\uppi \mathrm{i}}\,z^{-k-1}\mathrm{tr}(X+Yz)^{j+k}.\label{eqn:keyrule}
\end{equation}
There is a mild ambiguity about whether the right hand side should be viewed in the free algebra $F$ in variables $X$ and $Y$, or in its quotient $F/[F,F]$. Regardless, each contour integral in the right hand side \eqref{eqn:keyrule} can be viewed as an element of $F$, where the integral is determined by simply computing the residue.  As the coefficient of the $1/z$ term is precisely the sum of all terms containing $k$ occurrences of $Y$ in the expansion of $(X+Y)^{j+k}$, and this equals $(j+k)\NeckPoly_{j,k}$, it follows that\footnote{To calibrate the trace, use \cite[(3.19), (3.20)]{Ferrari} applied to $V_Y=0=V$ and $V_X(X)=X^{3}$.  There, $E=z^2w^3$ and so the additional glueing term is $\delta_wE=3z^2w^2$.  Thus in the notation of this paper, $\uplambda_{30}=3$ whilst all other $\uplambda_{jk}=0$. Ferrari \cite[(3.20)]{Ferrari} predicts $\mathbb{W}=V_X(X)=X^3$, which equals $3\frac{1}{3}X^3=\uplambda_{30}\NeckPoly_{3,0}$.}  
\[
\frac{\uplambda_{jk}}{j+k}\oint_{C_0}\frac{dz}{2\uppi \mathrm{i}}\,z^{-j-1}\mathrm{tr}(X+Yz)^{j+k}=\frac{\uplambda_{jk}}{j+k}\cdot(j+k)\NeckPoly_{j,k}(X,Y)
\]
and so $\mathbb{W}=\sum\uplambda_{jk}\NeckPoly_{j,k}$, which is the noncommutative potential in \ref{cor:NCdefs}.


\begin{thebibliography}{DWZ}

\bibitem[A]{Artin}
M.~Artin, \emph{Some numerical criteria for contractability of curves on algebraic surfaces}, Amer.\ J.\ Math.\ \textbf{84} (1962), 485--496. 

\bibitem[A+]{Dbook}
P.~.S.~Aspinwall, T.~Bridgeland, A.~Craw, M.~R.~Douglas, M.~Gross, A.~Kapustin, G.~W.~Moore, G.~Segal, B.~Szendr\H{o}i, P.~M.~H.~Wilson, \emph{Dirichlet branes and mirror symmetry},
Clay Mathematics Monographs, \textbf{4}. American Mathematical Society, Providence, RI; Clay Mathematics Institute, Cambridge, MA,  2009. x+681 pp. ISBN: 978-0-8218-3848-8.


\bibitem[AK]{AspinwallKatz}
P.~S.~Aspinwall, S.~Katz, \emph{Computation of superpotentials for D-branes}, Comm.\ Math.\ Phys.\ \textbf{264} (2006), no.~1, 227--253.


\bibitem[B]{BoothDerived}
M.~Booth, \emph{The derived deformation theory of a point}, Math.\ Z.\ \textbf{300} (2022), no.~3, 3023--3082.

\bibitem[BCP]{magma}
W.~Bosma, J.~Cannon, and C.~Playoust, \emph{The Magma algebra system. I. The user language}, J.\ Symbolic Comput., \textbf{24} (1997), 235--265.
\url{http://www.magma.usyd.edu}

\bibitem[BW]{BW2}
G.~Brown, M.~Wemyss, \emph{Local normal forms of noncommutative functions}, \href{https://arxiv.org/abs/2111.05900}{\sf arXiv:2111.05900}.



\bibitem[CS]{CS}
A.~Canonaco, P.~Stellari, \emph{A tour about existence and uniqueness of dg enhancements and lifts},
 J.\ Geom.\ Phys.\  \textbf{122}  (2017), 28--52.

\bibitem[CM]{CM}
C.~Curto, D.~R.~Morrison, \emph{Threefold flops via matrix factorization}, J.\ Algebraic Geom.\ \textbf{22} (2013), no.~4, 599--627.

\bibitem[D]{Davison}
B.~Davison, \emph{Refined invariants of finite-dimensional Jacobi algebras}, 
\href{https://arxiv.org/abs/1903.00659}{\sf arXiv:1903.00659}, to appear Algebraic Geometry.

\bibitem[DW]{DW1}
W.~Donovan, M.~Wemyss, \emph{Noncommutative deformations and flops}, Duke Math.\ J.\ \textbf{165} (2016), no.~8, 1397--1474. 

\bibitem[DWZ]{DWZ}
H.~Derksen, J.~Weyman, A.~Zelevinsky, \emph{Quivers with potentials and their representations.\ I.\ Mutations.} Selecta Math.\ (N.S.) \textbf{14} (2008), no.~1, 59--119.

\bibitem[F]{Ferrari}
F.~Ferrari, \emph{Planar diagrams and Calabi-Yau spaces}, Adv.\ Theor.\ Math.\ Phys.\ \textbf{7} (2003), no.~4, 619--665. 

\bibitem[G]{Ginz}
V.~Ginzburg, \emph{Calabi--Yau algebras}, \href{https://arxiv.org/abs/math/0612139}{\sf arXiv:math/0612139}.


\bibitem[H]{Hartshorne}
R.~Hartshorne, \emph{Algebraic geometry}. Graduate Texts in Mathematics, No.\ 52.\ Springer-Verlag, New York-Heidelberg, 1977. xvi+496 pp.


\bibitem[J]{Jimenez}
J.~Jim\'{e}nez, \emph{Contraction of nonsingular curves}, Duke Math.\ J.\ \textbf{65} (1992), no.~2, 313--332.

\bibitem[K1]{Kadeishvili}
T.~V.~Kadeishvili, \emph{On the theory of homology of fiber spaces}, (Russian) Uspekhi Mat.\ Nauk \textbf{35} (1980), no.~3(213), 183--188.
Translation in Russian Math.\ Surveys 35:3 (1980), 231--238.

\bibitem[K2]{Katz}
S.~Katz, \emph{Versal deformations and superpotentials for rational curves in smooth threefolds}, Symposium in Honor of C.\ H.\ Clemens, (Salt Lake City, UT, 2000), 129--136, Contemp.\ Math., \textbf{312}, Amer.\ Math.\ Soc., Providence, RI, 2002. 

\bibitem[K3]{Kawamata_survey}
Y.~Kawamata, \emph{On non-commutative formal deformations of coherent sheaves on an algebraic variety}, EMS Surv.\ Math.\ Sci.\ \textbf{8} (2021), no.~1-2, 237--263. 

\bibitem[K4]{Keller}
B.~Keller, \emph{Deformed Calabi-Yau completions}, with an appendix by Michel Van den Bergh,
J.\ Reine Angew.\ Math.\ \textbf{654} (2011), 125--180.

\bibitem[KS]{KS}
M.~Kontsevich, Y.~Soibelman, \emph{Notes on {$\Ainf$}-algebras, {$\Ainf$}-categories and non-commutative geometry}, 153--219, Lecture Notes in Phys., \textbf{757}, Springer, 2009. 

\bibitem[LS]{LS}
V.~A.~Lunts, O.~M.~Schn\"{u}rer, \emph{New enhancements of derived categories of coherent sheaves and applications}, J.\ Algebra  \textbf{446}  (2016), 203--274.

\bibitem[N]{Namba}
M.~Namba, \emph{On maximal families of compact complex submanifolds of complex manifolds}, Tohoku Math.\ J.\ (2) \textbf{24} (1972), 581--609.

\bibitem[S]{stacks}
The Stacks Project Authors, \emph{Stacks Project}, {\url{https://stacks.math.columbia.edu}}, {2018}.

\bibitem[T]{TodaUtah}
Y.~Toda, \emph{Non-commutative deformations and Donaldson-Thomas invariants}, Algebraic geometry: Salt Lake City 2015, 611--631, Proc.\ Sympos.\ Pure Math., 97.1, Amer.\ Math.\ Soc., Providence, RI, 2018.


\end{thebibliography}
\end{document}